\documentclass[10pt,twoside]{article}
\usepackage{amsmath,amssymb,amsthm,mathptmx,hyperref,nicefrac}
\usepackage{titlesec}
\usepackage{indentfirst}
\usepackage[inner=2.4cm,outer=2.65cm,bottom=3cm]{geometry}
\titleformat*{\section}{\normalsize\bfseries}
\titleformat*{\subsection}{\normalsize\itshape}
\numberwithin{equation}{section}
\theoremstyle{remark}

\theoremstyle{definition}

\newtheorem{theorem}{Theorem}[section]
\newtheorem{proposition}{Proposition}[section]
\newtheorem{lemma}{Lemma}[section]
\newtheorem{corollary}{Corollary}[section]
\newcommand{\f}[1]{\pmb{#1}}
\DeclareMathOperator{\R}{\mathbb{R}}
\DeclareMathOperator{\C}{\mathcal{C}}
\DeclareMathOperator{\F}{\mathcal{F}}
\DeclareMathOperator{\AC}{\mathcal{AC}}

\DeclareMathOperator{\V}{\f H^1_{0,\sigma}}
\DeclareMathOperator{\Vd}{(\f H^{1}_{0,\sigma})^*}

\DeclareMathOperator{\Ha}{\f L^2_{\sigma}}
\DeclareMathOperator{\Hb}{\f{H}^1_0}
\DeclareMathOperator{\Hc}{\f H^2}

\DeclareMathOperator{\He}{\f{H}^1}
\DeclareMathOperator{\Le}{\f{L}^2}
\DeclareMathOperator{\sgn}{sgn}

\newcommand{\Hrand}[1]{{\f H^{\nicefrac{#1}{2}}(\partial\Omega)}}

\DeclareMathOperator{\Se}{\mathcal{S}}

\DeclareMathOperator{\M}{\mathcal{M}}

\DeclareMathOperator{\ra}{\rightarrow}
\DeclareMathOperator{\de}{\text{d}}
\DeclareMathOperator{\tr}{tr}

\newcommand{\dreidots}{\text{\,\multiput(0,-2)(0,2){3}{$\cdot$}}\,\,\,\,}

\newcommand{\dreidotkom}{\text{\,\multiput(0,0)(0,2){2}{$\cdot$}\put(0,0){,}}\,\,\,\,}

\newcommand{\pat}[2]{\frac{\partial #1}{\partial #2}}
\DeclareMathOperator{\di}{\nabla \cdot}

\newcommand{\ov}[1]{\overline{#1}}

\newcommand{\rot}[1]{[ #1 ]_{\f X}}
\newcommand{\rott}[1]{[ #1 ]_{-\f X}}
\DeclareMathOperator{\curl}{\nabla \times}

\newcommand{\intt}[1]{\int_{0}^T\left[{ #1}\right] \de t}
\newcommand{\intte}[1]{\int_{0}^T{ #1} \de t}
\newcommand{\inttet}[1]{\int_{0}^{t}{ #1} \de s}
\newcommand{\inttett}[1]{\int_{0}^{t}\left [{ #1} \right ]\de s}

\newcommand{\inttes}[1]{\int_{s}^{t}{ #1} \de \tau}

\renewcommand{\ll}[1]{\langle\hspace{-0.75mm}\langle{#1}\rangle\hspace{-0.75mm}\rangle}

\newtheorem{rem}{Remark}
\newtheorem{defi}{Definition}

\DeclareMathOperator{\sym}{{sym}}
\DeclareMathOperator{\Lap}{\Delta_{\f \Lambda}}

\DeclareMathOperator{\skw}{skw}
\newcommand{\sy}[1]{(\nabla \f #1)_{{\sym}}}
\newcommand{\sk}[1]{(\nabla \f #1)_{\skw}}
\renewcommand{\t}{\partial_t  }

\newcommand{\syd}[1]{(\nabla \fd {#1})_{\sym}}

\newcommand{\fn}[1]{\f {{#1}}_{n}}
\newcommand{\fd}[1]{\f {{#1}}_{\delta}}

\newcommand{\fk}[1]{\f {{#1}}}

\newcommand{\vv}{\tilde{\f v}}
\newcommand{\dd}{\tilde{\f d}}

\newcommand{\syv}{(\nabla \tilde{\f v})_{{\sym}}}
\newcommand{\skv}{(\nabla \tilde{\f v})_{\skw}}
\renewcommand{\o}{\otimes}

\newcommand{\tq}{\tilde{\f q}}
\newcommand{\te}{\tilde{\f e}}

\author{%
\textsc{ Robert Lasarzik}\footnote{
Department of Mathematics,
Secr. MA 5-3, Technical University Berlin,
Straße des 17. Juni 136,
D-10623 Berlin, Germany
}}
\title{\begin{Large}
\textbf{Weak-strong uniqueness for measure-valued solutions to the Ericksen--Leslie model equipped with the Oseen--Frank free energy}\footnote{This work was funded by CRC 901 {\em Control of self-organizing nonlinear systems: Theoretical methods and concepts of application} (Project A8)\/.
}
\end{Large}
}
\begin{document}
\markboth{Weak-strong uniqueness for measure-valued solutions to the Ericksen--Leslie model}{R.~Lasarzik}
\date{Version \today}
\maketitle
%%%%%%%%%%%%%%%%%%%%%%%%%%%%%%%%%%%%%%%%%%%%%
\begin{abstract}\normalsize We analyze the Ericksen--Leslie system equipped with the Oseen--Frank energy in three space dimensions. Recently, the author introduced the concept of measure-valued solutions to this system and showed the global existence of these generalized solutions. In this paper, we show that suitable measure-valued solutions, which fulfill an associated energy inequality, enjoy the weak-strong uniqueness property, i.\,e.~the measure-valued solution agrees with a strong solution if the latter exists. The weak-strong uniqueness is shown by a relative energy inequality for the associated nonconvex energy functional.
\end{abstract}
\setcounter{tocdepth}{1}
\tableofcontents
%%%%%%%%%%%%%%%%%%%%%%%%%%%%%%%
%% Introduction
%%%%%%%%%%%%%%%%%%%%%%%%%%%%%%%
\section{Introduction}\label{sec:intro}
Nonlinear partial differential equations require generalized solution concepts. Uniqueness of solutions, however often is an open problem because of lack of regularity. A prominent example are the Navier-Stokes equations in three space dimensions with Leray's~\cite{leray} weak solutions and Serrin's~\cite{serrin} uniqueness result. 
Already there the concept of weak-strong uniqueness is applied: a generalized solution is compared with a solution exhibiting more regularity. 
The generalized solution concept of measure-valued solutions was first introduced by Tartar~\cite{tartar} using conventional Young measures. 
This concept is also well-known in the context of the
 Euler equations. DiPerna and Majda~\cite{DiPernaMajda} define measure-valued solutions to the Euler equation by  introducing  so-called \textit{generalized Young-measures}, capturing oscillation and concentration effects for sequences that are only bounded in $L^1$. Another step in the analysis of such sequences is done in Alibert and Bouchitt{\'e}~\cite{alibert} and we will heavily rely on the techniques introduced there. 
Brenier, De Lellis, and Sz{\'e}kelyhidi~\cite{weakstrongeuler}
prove the first weak-strong uniqueness result for measure-valued solutions, involving generalized Young measures. They consider the incompressible Euler equation and certain hyperbolic conservation laws.
Some of the techniques introduced in~\cite{weakstrongeuler} are also used in this work. 
Additionally, there are works on the weak-strong uniqueness of measure-valued solutions 
by {Demoulini, Stuart and Tzavaras}~\cite{demoulini} in the context of elastodynamics
 and by {Gwiazda, {\'S}wierczewska-Gwiazda \& Wiedemann}~\cite{gwiazda} for a class of compressible Euler equations. 
 
 An important argument of this article relies on the formulation of a relative energy, which is often called relative entropy in this context ( see {Feiereisl, Jin \& Novotn{\'y}}~\cite{Feireislrelative}). 
 The idea of an relative entropy to compare two solutions goes back to Dafermos~\cite{dafermos}.
 
 For a strictly convex entropy function $\eta: \R \ra \R$, the relative entropy of two solutions~$u$ and $\tilde{u}$ is given by (see Ref.~\cite[Sec.~5.3]{dafermos2})
\begin{align}
\mathcal{E}:= \eta (u) - \eta(\tilde{u}) - \langle\eta'(\tilde{u}),( u -\tilde{u})\rangle \,.\label{relencon}
\end{align}
The strict convexity of $\eta$ guarantees that $\mathcal{E}$ is nonnegative.

 The approach of the relative energy for convex functionals has been used, e.g., to show the weak-strong uniqueness property of solutions (see {Feireisl and Novotn\'y}~\cite{novotny}), the stability of an equilibrium state (see Feireisl~\cite{feireislstab}), the convergence to a singular limit problem (see {Breit, Feireisl and Hofmanova}~\cite{breit} as well as {Feireisl}~\cite{feireislsingular}), or to derive {\em a posteriori} 
estimates of numerical solutions (see {Fischer}~\cite{fischer}).
Another possible application is the definition of a generalized solution concept, the so-called dissipative solutions. The formulation of such a concept relies on an inequality instead of an equality (see Lions~\cite[Sec.~4.4]{lionsfluid}).
 
Our aim is to generalize the above concept of weak-strong uniqueness for measure-valued solutions via a relative energy approach to the three dimensional Ericksen--Leslie model equipped with the Oseen--Frank energy describing nematic liquid crystal flow.

Nematic liquid crystals are anisotropic fluids. The rod-like molecules build, or are dispersed in, a fluid and are directionally ordered. This ordering and its direction heavily influences the properties of the material such as light scattering or flow behaviour. This gives rise to many applications, where \textit{liquid crystal displays} are only the most prominent ones. The \textit{Ericksen--Leslie model} is the most common model to describe nematic liquid crystals. The direction of the aligned molecule is modelled by a unit-vector field and the fluid flow by a velocity field. Since this model was proposed by Ericksen~\cite{Erick2} and Leslie~\cite{leslie} in the 60ies, it is extensively studied.

%%%%%%%%%%%%%%%%%%%%%%%%%%
%% Review 
%%%%%%%%%%%%%%%%%%%%%%%%%%%%
The first mathematical analysis of  a simplified Ericksen--Leslie model is done by Lin and Liu~\cite{linliu1}. They show global existence of weak solutions and local existence of strong solutions. Additionally, they manage to generalise this results to a more realistic model~\cite{linliu3}. They also show partial regularity of weak solutions to the considered system~\cite{linliu2}. Following this works, there have been many articles considering slightly more complicated models, see~\cite{prohl},~\cite{allgemein},~\cite{isothermal},~or~\cite{wal1} for example. To the best of the author's knowledge, the only generalisation with respect to the free energy potential is performed by Emmrich and the author in~\cite{unsere}. 

There are also results on \textit{local existence of solutions} to realistic models (see for instance~\cite{localin3d},~\cite{recent} or~\cite{Pruess2}). Especially, local strong solutions are known to exists to different simplifications of the system considered in this article.
The full (thermodynamically consistent) Ericksen--Leslie system with the one constant approximation of the Oseen--Frank energy is considered in~\cite{Pruess2}. Whereas, the simplified Ericksen--Leslie system with the full Oseen--Frank energy is studied in~\cite{localin3d}. 
For more on liquid crystals, we refer to Emmrich, Klapp and Lasarzik~\cite{sabine}. 
Since finite time singularities in nematic liquid crystals have been observed  experimentally~\cite{singul2} and analytically~\cite{blow}, it seems appropriate to inquire a weakened solution concept such as measure-valued solutions.

Recently, the existence of measure-valued solutions to the Ericksen--Leslie model equipped with the Oseen--Frank energy has been proven by the author~(see~\cite{masswertig}). 
It is the first solution concept for the Ericksen--Leslie model in its full generality. The measure-valued solutions are a very weak solution concept. The Gradient of the director is only represented by a generalized Young measure. This article is dedicated to the proof of the weak-strong uniqueness property of the measure-valued solutions to the Ericksen--Leslie model.
The main theorem (see Theorem~\ref{thm:main}) shows that a  suitable measure-valued solution coincides with a strong solution emanating from the same initial data, as long as the latter exists. 
A suitable measure-valued solution fulfills an additional energy inequality. This energy inequality can only be shown to hold under the assumption of Parodi's relation~(see~\eqref{parodi} below), which follows from Onsager's reciprocal principle. The weak-strong uniqueness can also shown to hold for suitable measure-valued solutions without assuming Parodi's relation, but these solutions are not known to exist.

The novelty of this paper is the generalization of the relative energy approach to a system with a nonconvex energy. Since the Oseen--Frank energy is not convex, the relative energy defined by~\eqref{relencon} is not necessarily positive anymore. In this paper we use an alternative way to define the relative energy (see~\eqref{relEn}) for the nonconvex energy Oseen--Frank energy and derive a relative energy inequality resulting in the weak-strong uniqueness property of the solutions. This new approach can hopefully be used to prove other properties, like \textit{a posteriori} estimates or the existence of dissipative solutions, of the system too.

%%%%%%%%%%%%%%%%%%%%%%%%%%%
%% Outline
%%%%%%%%%%%%%%%%%%%%%%%%%%%%

The paper is organised as follows: In Section~\ref{sec:not}, we collect some notation. Section~\ref{sec:model} contains the model, the definition of the suitable generalized solutions, the definition of the strong solutions and the main result.
In Section~\ref{sec:pre}, some auxiliary lemmas are collected and the energy equality for the strong solution is shown.
While Section~\ref{sec:rel} introduces the relative energy and some associated estimates, Section~\ref{sec:int} collects certain integration-by-parts formulae.
The proof of the main result is carried out in Section~\ref{sec:main}. 
In Appendix~\ref{sec:app}, some tensor calculations are collected.

%%%%%%%%%%%%%%%%%%%%%%%
%%%%%%%%%%%%%%%%%%%%%%%%%
%%%%%%%%%%%%%%%%%%%%%%%%%
%%%%%%%%%%%%%%%%%%%%%%%%%%
%%%%%%%%%%%%%%%%%%%%%%%%%%%

%%%%%%%%%%%%%%%%%%%%%%%%%%%
%% Mathematics
%%%%%%%%%%%%%%%%%%%%%%%%%%%
\subsection*{Notation\label{sec:not}}
Vectors of $\R^3$ are denoted by bold small Latin letters. Matrices of $\R^{3\times 3}$ are denoted by bold capital Latin letters. We also use tensors of higher order, which are denoted by bold capital Greek letters.
Moreover, numbers are denoted be small Latin or Greek letters, and capital Latin letters are reserved for potentials.
The euclidean scalar product in $\R^3$ is denoted by a dot $ \f a \cdot \f b : = \f a ^T \f b = \sum_{i=1}^3 \f a_i \f b_i$,  for $ \f a, \f b \in \R^3$ and the Frobenius product in $\R^{3\times 3}$ by a double point $ \,\f A: \f B:= \tr ( \f A^T \f B)= \sum_{i,j=1}^3 \f A_{ij} \f B_{ij}$, for $\f A , \f B \in \R^{3\times 3}$.
Additionally, the scalar product in the space of Tensors of order three is denoted by three dots 
\begin{align*}
\f \Upsilon \dreidots\, \f \Gamma : =\left [ \sum_{j,k,l=1} ^3 \f \Upsilon_{jkl} \f \Gamma_{jkl}\right ], \quad    \f \Upsilon \in \R^{3\times 3 \times 3 },  \, \f \Gamma \in \R^{3\times 3 \times 3}  .
\end{align*}
The associated norms are all denoted by $| \cdot |$, where also the norms of tensors of higher order are denoted in the same way 
\begin{align*}
| \f \Lambda|^2 := \sum_{i,j,k,l=1}^3 
\f \Lambda_{ijkl}^2\,,\quad\text{for }\f \Lambda \in \R^{3^4} \quad 
\text{and }\quad| \f \Theta |^2  := \sum_{i,j,k,l,m,n=1}^3 \f \Theta ^2_{ijklmn}\,,\quad\text{for } 
\f \Theta \in \R^{3^6}\,
\end{align*}
respectively.
%%%%%%%%%%%%%%%%%%%%%%%%%%
%%%%%%%%%%%%%%%%%%%%%%%%%%
%%%%%%%%%%%%%%%%%%%%%%%%%%%%
Similar, we define the products of tensors of different order.
The product of a tensor of third order and a matrix and a  vector is defined by
\begin{align*}
\f \Gamma : \f A := \left [ \sum_{j,k=1}^3 \f \Gamma_{ijk}\f A_{jk}\right ]_{i=1}^3\, ,  \,   \f \Gamma \cdot \f A := \left [ \sum_{k=1}^3 \f \Gamma_{ijk}\f A_{kl}\right ]_{i,j,l=1}^3\, ,  \,    \f \Gamma \cdot \f a  := \left [ \sum_{k=1}^3 \f \Gamma_{ijk}\f a_k\right ]_{i,j=1}^3 \, , \, \f \Gamma\in \R^{3 \times 3\times 3 } , \, \f A \in \R^{3\times 3},\,\f a \in \R^3 .
\end{align*}
%\rand{Hier noch produkte vierter Stufe mit zweiter und erster Stufe...}
The product of a tensor of fourth order with a matrix and a vector is defined by
\begin{align*}
\f \Lambda : \f A : =\left [ \sum_{k,l=1} ^3 \f \Lambda_{ijkl} \f A_{kl}\right ]_{i,j=1}^3\,, \, 
\f \Lambda : \f a : =\left [ \sum_{l=1} ^3 \f \Lambda_{ijkl} \f a_{l}\right ]_{i,j,k=1}^3\,, 
 \,    \f \Lambda \in \R^{3^4 },  \,  \f A \in \R^{3\times 3 } \, 
 \f a 
 \in 
 \R^3 .
\end{align*}
%%%%%%%%%%%
The product of tensors of fourth and third order is given by
\begin{align*}
\f \Lambda  : \f \Gamma : =\left [ \sum_{k,l=1} ^3 \f \Lambda_{ijkl} \f \Gamma _ {klm}\right ]_{i,j,m=1}^3 , \, \f \Lambda  \dreidots \f \Gamma : =\left [ \sum_{j,k,l=1} ^3 \f \Lambda_{ijkl} \f \Gamma_{jkl}\right ]_{i=1}^3, \,    \f \Lambda \in \R^{3^4},  \, \f  \Gamma \in \R^{3\times 3\times 3}  .
\end{align*}
The product of a tensor of fourth order and a matrix or a tensor of third order is defined via
\begin{align*}
 \f A : \f \Theta : ={}& \left [ \sum_{i,j=1} ^3 \f A_{ij} \f \Theta_{ijklmn}  \right ]_{k,l,m,n=1}^3 , \, \f \Theta \dreidots \f \Gamma : = \left [ \sum_{l,m,n=1} ^3 \f \Theta_{ijklmn} \f \Gamma_{lmn}\right ]_{i,j,k=1}^3 ,  \,  
  \f \Theta \in \R^{3^6},\f A \in \R^{3\times 3} ,\f \Gamma \in \R^{3\times 3 \times 3}\,.
\end{align*}
The product of a vector and a tensor of fourth order is defined differently. The definition is adjusted to the cases of this work:
 \begin{align*}
\f a \cdot \f \Theta :={}& \left [ \sum_{k=1} ^3 \f a_{k} \f \Theta_{ijklmn}  \right ]_{i,j,l,m,n=1}^3,
\, 
  \f \Theta \in \R^{3^6},\f a \in \R^{3} \,.
\end{align*}
The standard matrix and matrix-vector multiplication is written without an extra sign for bre\-vi\-ty,
$$\f A \f B =\left [ \sum _{j=1}^3 \f A_{ij}\f B_{jk} \right ]_{i,k=1}^3 \,, \quad  \f A \f a = \left [ \sum _{j=1}^3 \f A_{ij}\f a_j \right ]_{i=1}^3\, , \quad  \f A \in \R^{3\times 3},\,\f B \in \R^{3\times3} ,\, \f a \in \R^3 .$$
The outer vector product is given by
 $\f a \otimes \f b := \f a \f b^T = \left [ \f a_i  \f b_j\right ]_{i,j=1}^3$ for two vectors $\f a , \f b \in \R^3$ and by $ \f A \o \f a := \f A \f a ^T = \left [ \f A_{ij}  \f a_k\right ]_{i,j,k=1}^3 $ for a matrix $ \f A \in \R^{3\times 3} $ and a vector $ \f a \in \R^3$. 
The symmetric and skew-symmetric parts of a matrix are given by 
$\f A_{\sym}: = \frac{1}{2} (\f A + \f A^T)$ and 
$\f A _{\skw} : = \frac{1}{2}( \f A - \f A^T)$, respectively ($\f A \in \R^{3\times  3}$).
For the product of two matrices $\f A, \f B \in \R^{3\times 3 }$, we observe
 \begin{align*}
 \f A: \f B = \f A : \f B_{\sym}\,, \quad \text{if } \f A^T= \f A\quad \text{and}\quad
  \f A: \f B = \f A : \f B_{\skw}\,, \quad \text{if } \f A^T= -\f A\, .
 \end{align*}
Furthermore, it holds $\f A^T\f B : \f C = \f B : \f A \f C$ for
$\f A, \f B, \f C \in \R^{3\times 3}$ and
$ \f a\otimes \f b : \f A = \f a \cdot \f A \f b$ for
$\f a, \f b \in \R^3$, $\f A \in \R^{3\times 3 }$  and hence $ \f a \otimes \f a : \f A = \f a \cdot \f A \f a =  \f a \cdot \f A_{\sym} \f a$.
%%%%%%%%%%%%%%%%%%%%%%%%%%%%%%%%
%%%%%%%%%%%%%%%%%%%%%%%%%%%%%%%%%
%%%%%%%%%%%%%%%%%%%%%%%%%%%%%%%%%

%%%%%%%%% Nabla operator
We use  the Nabla symbol $\nabla $  for real-valued functions $f : \R^3 \to \R$, vector-valued functions $ \f f : \R^3 \to \R^3$ as well as matrix-valued functions $\f A : \R^3 \to \R^{3\times 3}$ denoting
\begin{align*}
\nabla f := \left [ \pat{f}{\f x_i} \right ] _{i=1}^3\, ,\quad
\nabla \f f  := \left [ \pat{\f f _i}{ \f x_j} \right ] _{i,j=1}^3 \, ,\quad
\nabla \f A  := \left [ \pat{\f A _{ij}}{ \f x_k} \right ] _{i,j,k=1}^3\, .
\end{align*}
 The divergence of a vector-valued and a matrix-valued function is defined by
\begin{align*}
\di \f f := \sum_{i=1}^3 \pat{\f f _i}{\f x_i} = \tr ( \nabla \f f)\, , \quad  \di \f A := \left [\sum_{j=1}^3 \pat{\f A_{ij}}{\f x_j}\right] _{i=1}^3\, .
\end{align*}

 Additionally, we abbreviate $\nabla \nabla $ and $\di \di $ by $\nabla^2 $ and $\nabla ^2 :$ respectively.
For a given tensor of fourth order, we abbreviate the associated second order operator by 
$\Delta_{\f \Lambda}\f d : =\di \f \Lambda : \nabla \f d  $ acting on functions $\f d \in \C^1(\Omega \times [0,T];\R^3)$.

Throughout this paper, let $\Omega \subset \R^3$ be a bounded domain of class $\C^{3,1}$.
We rely on the usual notation for spaces of continuous functions, Lebesgue and Sobolev spaces. Spaces of vector-valued functions are  emphasised by bold letters, for example
$
\f L^p(\Omega) := L^p(\Omega; \R^3)$,
$\f W^{k,p}(\Omega) := W^{k,p}(\Omega; \R^3)$.
The standard inner product in $L^2 ( \Omega; \R^3)$ is just denoted by
$ (\cdot \, , \cdot )$, in $L^2 ( \Omega ; \R^{3\times 3 })$
by $(\cdot ; \cdot )$, and in $L^2 ( \Omega ; \R^{3\times 3\times 3 })$ by   $(\cdot \dreidotkom \cdot )$.

The space of smooth solenoidal functions with compact support is denoted by $\mathcal{C}_{c,\sigma}^\infty(\Omega;\R^3)$. By $\f L^p_{\sigma}( \Omega) $, $\V(\Omega)$,  and $ \f W^{1,p}_{0,\sigma}( \Omega)$, we denote the closure of $\mathcal{C}_{c,\sigma}^\infty(\Omega;\R^3)$ with respect to the norm of $\f L^p(\Omega) $, $ \f H^1( \Omega) $, and $ \f W^{1,p}(\Omega)$ respectively.

The dual space of a Banach space $V$ is always denoted by $ V^*$ and equipped with the standard norm; the duality pairing is denoted by $\langle\cdot, \cdot \rangle$. The duality pairing between $\f L^p(\Omega)$ and $\f L^q(\Omega)$ (with $1/p+1/q=1$), however, is denoted by $(\cdot , \cdot )$, $( \cdot ; \cdot )$, or $( \cdot \dreidotkom \cdot )$. 

The unit ball in d dimensions is denoted by $B_d:= \{ \f x \in \R^d ; | \f x | < 1\}$ and the sphere in $d$-dimensions by  $\Se^{d-1}:= \{ \f x \in \R^d ; | \f d |=1  \}$.

For $Q\subset \R^d$, the Radon measures are denoted by $\mathcal{M}(Q)$, the positive Radon measures by $\mathcal{M}^+(Q)$, and probability measures by $\mathcal{P}(Q)$. We recall that the Radon measures equipped with the total variation are a Banach space and  for compact sets Q, it can be characterized by~$\mathcal{M}(Q) = ( \C(Q))^*$ (see~\cite[Theorem~4.10.1]{edwards}). 
The integration of a function $f\in \C(Q)$ with respect to a measure $\mu\in \mathcal{M}(Q)$ is denoted by $ \int_Qf(\f h ) \mu(\de \f h)\,.$ In case of the Lebesgue measure we just write 
$ \int_Qf(\f h ) \de \f h\,.$

The cross product of two vectors is denoted by $\times $. We introduce the notation $ \rot{\cdot}$, which is defined via
\begin{align}
\rot{\cdot } : \R^d \ra \R^{d\times d}\, , \quad \rot{ \f h} := \begin{pmatrix}
0& - \f h_3 &\f h_2\\
\f h_3 & 0 & - \f h_1 \\
- \f h_2 & \f h_1 & 0
\end{pmatrix}\, .
\end{align}
The $i$-th component of the vector $\f h\in \R^3$ is denoted by $\f h_i$. 
The mapping $\rot{\cdot}$ has some nice properties, for instance
\begin{align*}
\rot{\f a}\f b = \f a \times \f b \, ,\quad \rot{\f a} ^T \rot{\f b} = (\f a \cdot \f b) I - \f b \otimes \f a\, ,
\end{align*}
for all $\f a$, $\f b \in \R^3$, where $I$ denotes the identity matrix in $\R^{3\times 3}$ or 
\begin{align*}
 \quad \rot{\f a} : \nabla \f b = \rot{\f a} : \sk b = \f a \cdot \curl \f b \, , \quad \di \rot{ \f a} = - \curl \f a \, , \quad \frac{1}{2} \rot{\curl \f a} = \sk a \,,
\end{align*}
for all $ \f a, \f b \in \C^1(\ov \Omega)$.
Displaying the cross product by this matrix makes the operation associative. 

Additionally, we define $ \rott{\cdot } : \R^{3 \times 3} \ra \R^3 $, it is the left inverse of $\rot{\cdot} $ and given by
\begin{align*}
\rott{\f A} : = \begin{pmatrix}
\f A_{3,2} \\ \f A_{1,3} \\ \f A_{2,1}
\end{pmatrix}\,,\quad \text{for all } \f A \in \R^{3\times 3}\,.%\text{ und damit } \rot{\rott{\cdot }} = I\,.
\end{align*} 
For this mapping holds $ \rott{\rot{\f a}} = \f a $ and, thus $ 2 \rott{\sk{ a}}= \curl \f a$, for all $ \f a \in \C^1(\ov{ \Omega }; \R^3)$. 

We also use the Levi--Civita tensor $\f \Upsilon\in \R^{3^3}$. Let $\mathfrak{S}_3$ be the symmetric group of all permutations of $(1,2,3)$. The sign of  a given permutation  $\sigma \in \mathfrak{S}_3$ is denoted by $\sgn \sigma $.
The Tensor~$\f \Upsilon$ is defined via
\begin{align*}
\f \Upsilon_{ijk}:= \begin{cases}
\sgn{\sigma},  &  ( i ,j,k) = \sigma( 1,2,3)\text{ with } \sigma\in \mathfrak{S}_3 ,\\ 
0, & \text{ else}\, .
\end{cases}
\end{align*}
This tensor allows it two write the cross product as 
\begin{align*}
\f a \times \f b = \f \Upsilon :(\f a \otimes \f b) =\f \Upsilon _{ijk} \f a_j \f b _k\, ,\quad \text{for all }\f a , \f b \in \R^d \, 
\end{align*}
and the curl via
\begin{align*}
\curl \f d = \f\Upsilon_{ijk} \partial_j \f d_k \,, \quad\text{for all }\f d \in \C^1 ( \Omega)\, .
\end{align*}

For a given Banach space $ V$, Bochner--Lebesgue spaces are denoted  by $ L^p(0,T; V)$. Moreover,  $W^{1,p}(0,T; V)$ denotes the Banach space of abstract functions in $ L^p(0,T; V)$ whose weak time derivative exists and is again in $ L^p(0,T; V)$ (see also
Diestel and Uhl~\cite[Section~II.2]{diestel} or
Roub\'i\v{c}ek~\cite[Section~1.5]{roubicek} for more details).
By  $\C([0,T]; V) $, and $ \C_w([0,T]; V)$, we denote the spaces of abstract functions mapping $[0,T]$ into $V$ that are absolutely continuous, continuous, and continuous with respect to the weak topology in $V$, respectively.
We often omit the time interval $(0,T)$ and the domain $\Omega$ and just write, e.g., $L^p(\f W^{k,p})$ for brevity.

Finally, by $c>0$, we denote a generic positive constant and by $C_\delta$ a constant depending on $\delta$.

\section{Model and main result\label{sec:model}}
\subsection*{Governing equations}
%%%
Let $ \Omega$ be of class $\C^{3,1}$.
We consider the Ericksen--Leslie model as introduced in~\cite{masswertig}. 
 The governing equations  read as
\begin{subequations}\label{eq:strong}
\begin{align}
\t {\f v}  + ( \f v \cdot \nabla ) \f v + \nabla p + \di \f T^E- \di  \f T^L&= \f g, \label{nav}\\
\f d \times \left (\t {\f d }+ ( \f v \cdot \nabla ) \f d -\sk{v}\f d + \lambda \sy{v} \f d + \f q\right ) & =0,\label{dir}\\
\di \f v & = 0,
\\
| \f d |&=1.
\end{align}%
\end{subequations}

We recall that $\f v : \ov{\Omega}\times [0,T] \ra \R^3$ denotes the velocity  of the fluid, $\f d:\ov{\Omega}\times[0,T]\ra \R^3$ represents the orientation of the rod-like molecules, and $p:\ov{\Omega}\times [0,T] \ra\R$ denotes the pressure.
The Helmholtz free energy potential~$F$, which is described rigorously in the next section, is assumed to depend only on the director and its gradient, $F= F( \f d, \nabla \f d)$.
The free energy functional~$\mathcal{F}$  is defined by
\begin{align*}
\mathcal{F}: \He \ra \R , \quad \mathcal{F}(\f d):= \int_{\Omega} F( \f d, \nabla \f d) \de \f x \,,
\end{align*}
and $\f q$ is its variational derivative (see Furihata and Matsuo~\cite[Section 2.1]{furihata}),
\begin{subequations}\label{abkuerzungen}
\begin{align}\label{qdefq}
\f q :=\frac{\delta \mathcal{F}}{\delta \f d}(\f d) =  \pat{F}{\f d}(\f d , \nabla\f d)-\di \pat{F}{\nabla \f d}(\f d, \nabla \f d)\, .
\end{align}
The Ericksen stress tensor $\f T^E$ is given by
\begin{equation}
\f T^E = \nabla \f d^T \pat{F}{\nabla \f d}( \f d , \nabla\f d ) \, .\label{Erik}
\end{equation}
The Leslie tensor is given by
\begin{align}
\begin{split}
\f T^L ={}&  \mu_1 (\f d \cdot \sy{v}\f d )\f d \otimes \f d +\mu_4 \sy{v}
 + {(\mu_5+\mu_6)} \left (  \f d \otimes\sy{v}\f d \right )_{\sym}
\\
& +{(\mu_2+\mu_3)} \left (\f d \otimes \f e  \right )_{\sym}
 +\lambda \left ( \f d \otimes \sy{v}\f d  \right )_{\skw} + \left (\f d \otimes \f e  \right )_{\skw}\, ,
\end{split}\label{Leslie}
\end{align}
where
\begin{align}
\f e : = \t {\f d} + ( \f v \cdot \nabla ) \f d - \sk v\f d\, .\label{e}
\end{align}

We emphasise that Parodi's law, 
\begin{equation}
 \lambda = \mu_2 + \mu _3\,,\label{parodi}
\end{equation}
 is neither essential for the reformulation nor the existence of measure-valued solutions, but is essential to prove the existence of \textit{suitable} measure-valued solutions (see Remark~\ref{rem:ex}) for which the weak-strong uniqueness property holds.

To ensure the dissipative character of the system, we assume that
\begin{align}
\begin{gathered}
\mu_1  > 0, \quad \mu_4 > 0, \quad (\mu_5+\mu_6)- \lambda (\mu_2+\mu_3)>0,\quad  \mu_1 +  \lambda (\mu_2+\mu_3)>0  \, ,
\\
4  \big( (\mu_5+\mu_6)- \lambda (\mu_2+\mu_3)\big)>
\big((\mu_2+\mu_3) -\lambda\big)^2\,.
\end{gathered}\label{con}
\end{align}
\end{subequations}
Finally, we impose boundary and initial conditions as follows:
\begin{subequations}\label{anfang}
\begin{align}
\f v(\f x, 0) &= \f v_0 (\f x) \quad\text{for } \f x \in \Omega ,
&\f v (  \f x, t ) &= \f 0  &\text{for }( t,  \f x ) \in [0,T] \times \partial \Omega , \\
\f d (  \f x, 0 ) & = \f d_0 ( \f x) \quad\text{for } \f x \in \Omega ,
&\f d (  \f x ,t ) & = \f d_1 ( \f x )  &\text{for }( t,  \f x ) \in [0,T] \times \partial \Omega .
\end{align}
\end{subequations}
We always assume that $\f d_1= \f d_0$ on $\partial \Omega$, which is a compatibility condition providing regularity.
%%%%%%%%%%%%%%%%%%%%%
%%%%%%%%%%%%%%%%%%%%%
%%%%%%%%%%%%%%%%%%%%%
%%%%%%%%%%%%%%%%%%%%%%%

\subsection*{The general Oseen--Frank energy}
The \textit{Oseen--Frank} energy is given by~(see~Leslie~\cite{leslie}) \
\begin{align*}
F(\f d , \nabla \f d) := \frac{K_1}{2} (\di \f d )^2 +\frac{K_2}{2}( \f d \cdot \curl \f d )^2  + \frac{K_3}{2} |\f d \times \curl \f d|^2 \,,
\end{align*}
where $K_1,K_2,K_3>0$.
This energy can be reformulated using the norm one restriction, to
\begin{align}
\begin{split}
2 F( \f d , \nabla \f d)&:= k_1 ( \di \f d) ^2 +  k_2 | \curl \f d |^2 + k_3 | \f d |^2 ( \di \f d )^2 +   k_4 ( \f d\cdot \curl \f d )^2  +  k_5 | \f d \times \curl \f d |^2 \, .
\end{split} \label{frei}
\end{align} 
where $ k_1=k_3=K_1/2$, $k_2={\min\{K_2,K_3\}}/{2}$, $k_4 = K_2-k_2$, and $ k_5 =K_3-k_2$ are again positive constants.
We remark that $| \f d |^2| \curl \f d |^2 = ( \f d \cdot \curl \f d )^2 + | \f d \times \curl \f d |^2 $.

We introduce short notations for the derivatives of the free energy~\eqref{frei} with respect to $\nabla \f d$ and $ \f d$. The free energy~\eqref{frei} can be seen as a function $F: \R^3 \times \R^{3\times 3}\ra \R $, where we replace $\f d$ in definition~\eqref{frei} by $\f h\in \R^3$ and $\nabla \f d $ by $ \f S\in \R^{3\times 3 }$. By  an easy vector calculation, we find
\begin{align*}
2 F( \f h, \f S)  &=  k_1 \tr (\f S) ^2 + k_2 | ( \f S)_{\skw}|^2  + k_3 | \f h |^2 \tr ( \f S)^2  + k_4 ( \rot{ \f h } : ( \f S) _{\skw})^2 + \\&\quad 4k_5| (\f S)_{\skw} \f  h|^2 \, ,
\end{align*}
see Section~\ref{sec:not} for the definition of the matrix $\rot{\cdot}$.

We abbreviate the derivative of $F$ with respect to $\f h$ by $F_{\f h}$ and the derivative with respect to $\f S$ by $F_{\f S}$, where
\begin{align*}
F_{\f S} : \R^3 \times \R^{3\times 3 } \ra \R^{3\times 3 } \, , \quad \text{and  } F_{\f h} : \R^3 \times \R^{3\times 3 } \ra \R^3 \, .
\end{align*}
These derivatives are given by
\begin{align}
\begin{split}
F_{\f S}(\f h ,\f S) & =   k _1 \tr (\f S) I + k_2 (\f S)_{\skw}  + k_3 \tr( \f S) | \f h|^2 I  +  k_4  \rot{ \f h}( \rot{\f h} :(\f S)_{\skw})\\& \quad  + 4 k_5 ((\f S)_{\skw}\f h \otimes \f h ) _{\skw} \\
F_{\f h} ( \f h , \f S) &= k_3 \tr(\f S)^2 \f h + 2 k_4 ( \rot{\f h} :(\f S)_{\skw}) \rott{( \f S)_{\skw}} + 4 k_5 ( \f S)^T_{\skw}(\f S) _{\skw} \f h 
\, ,
\end{split}\label{FSFh}
\end{align}
see Section~\ref{sec:not} for the definition of $\rott{\cdot}$.

To abbreviate, we define the tensor of order 4 $\f \Lambda \in \R^{3^4}$ and a tensor of order 6 $ \f \Theta \in \R^{3^6}$ via
\begin{align}
\f \Lambda_{ijkl} &: ={} k_1 \f \delta_{ij} \f \delta_{kl} + k_2 ( \f \delta_{ik}\f \delta_{jl}-\f\delta_{il}\f\delta_{jk})\,,\label{Lambda}
\intertext{and}
\f \Theta_{ijklmn} &:={}k_3 \f \delta_{ij}\f \delta_{lm}\f \delta_{kn} 
%+ k_4 \f \delta_{kn} ( \f \delta_{il}\f \delta_{jm} - \f \delta_{im}\f \delta_{jl})\notag \\& + ( k_5 -k_4) \left (  \f \delta_{il}\f \delta_{mn}\f \delta_{jk} - \f \delta_{mi}\f \delta_{ln}\f \delta_{jk} - \f \delta_{lj}\f \delta_{mn}\f \delta_{ik} + \f \delta_{jm}\f \delta_{ln}\f \delta_{ik}  \right )
%
% \\
%
%&
+ k_5  \left (  \f \delta_{il}\f \delta_{mn}\f \delta_{jk} - \f \delta_{mi}\f \delta_{ln}\f \delta_{jk} - \f \delta_{lj}\f \delta_{mn}\f \delta_{ik} + \f \delta_{jm}\f \delta_{ln}\f \delta_{ik}  \right )\notag
\\
& 
+k_4 \left (  \f \delta_{kn}\f \delta_{jm}\f \delta_{il} + \f \delta_{km}\f \delta_{jl}\f \delta_{in} + \f \delta_{kl}\f \delta_{jn}\f \delta_{im} - \f \delta_{kn}\f \delta_{jl}\f \delta_{im}- \f \delta_{km}\f \delta_{jn}\f \delta_{il} - \f \delta_{kl}\f \delta_{jm}\f \delta_{in}  \right )
\,, \label{ThetaOF}
\end{align}
respectively.
The free energy can be written as 
\begin{align}
2 F(\f d, \nabla \f d ) = \nabla \f d : \f \Lambda : \nabla \f d +  \nabla \f d \otimes  \f d  \dreidots \f \Theta \dreidots  \nabla \f d \otimes \f d  \,. \label{tensoren}
\end{align}
The partial derivatives~\eqref{FSFh} inserted in definition~\eqref{qdefq} 
gives the variational derivative in the case of the Oseen--Frank energy via
\begin{align}
\begin{split}
\f q ={}&- k_1 \nabla \di \f d - k_2 \curl \curl \f d - k_3\nabla (\di \f d | \f d|^2) - k_4 \di \left ( \rot{\f d} ( \f d \cdot \curl \f d ) \right ) - 4 k_5 \di \left (  ( \nabla \f d)_{\skw} \f d \o \f d\right )_{\skw}\\& + k_3 (\di \f d)^2 \f d  + k_4 ( \f d \cdot \curl \f d) \curl \f  d + 4 k_5  ( \nabla \f d)_{\skw}^T ( \nabla \f d)_{\skw}\f d
\\={}&-  \Lap\f d - \di \left ( \f d \cdot \f \Theta \dreidots \nabla \f d \o \f d \right ) + \nabla \f d : \f \Theta \dreidots \nabla \f d \o \f d
 \,.
\end{split}
\label{qdef}
\end{align}
 
The Tensor $\f \Lambda$ is strongly elliptic, i.e.~there is a $\eta>0$ such that $ \f a \otimes  \f b : \f \Lambda : \f a \otimes \f b   \geq \eta | \f a|^2 | \f b|^2 $ for all $\f a, \f b \in \R^3$. Indeed, it holds 
\begin{align*}
 \f a \otimes  \f b : \f \Lambda : \f a \otimes \f b  = k_1 (\f a \cdot \f b)^2 + k_2 ( | \f a |^2 | \f b|^2-( \f a \cdot \f b )^2 ) \geq \min\{k_1,k_2\} | \f a |^2 | \f b|^2\,.
\end{align*}

In the next section, we introduce the concept of measure-valued solutions. The proof of existence of measure-valued solutions to the Ericksen--Leslie model equipped with the Oseen--Frank energy is done by the author in~\cite{masswertig}. We also refer to~\cite{masswertig} for a more extensive introduction into the concept of generalized gradient Young measures.
\subsection*{Measure-valued solutions}
\begin{defi}[measure-valued solutions]\label{def:meas}
The tupel $( ( \f v ,\f d ) , ( \nu^o,m , \nu^\infty) , ( \mu , \nu^\mu))$ consisting of the pair $(\f v , \f d)$ of velocity field $\f v$ and director field  $\f d$, the generalized gradient Young measure $(\mu, \nu^\mu)$ and the defect measure $(\mu , \nu^\mu)$ (see below)  is  said to be a measure-valued solution to~\eqref{eq:strong} if
%%%%%%%%%%%%%%%%%%%%%%%%%%%%%%%%%%%%%%%%%
%%%%%%%%%%%%%%%%%%%%%%%%%%%%%%%%%%%%%%%%%%%%
\begin{align}
\begin{split}
\f v &\in L^\infty(0,T;\Ha)\cap  L^2(0,T;\V)
\cap W^{1,2}(0,T; ( \f W^{1,3}_{0,\sigma}(\Omega))^*),
\\ \f d& \in L^\infty(0,T;\He)\cap   W^{1,2}  (0,T; (  \f  L^{\nicefrac{3}{2}}) ),\\
\{\nu^o _{(\f x,t)}\}&  \subset \mathcal{P} ( \R^{3\times 3})\,, \text{ a.\,e.~in $\Omega\times (0,T)$} \, ,\\
 \{m_t\} &\subset \mathcal{M}^+(\ov\Omega)\,,\text{ a.\,e.~in $ (0,T)$} \, ,  \\
 \{\nu^\infty _{(\f x,t)}\} &\subset \mathcal{P}(\ov B_3\times \Se^{3^2-1})\,,\text{ $m_t $-a.\,e.~in }\ov\Omega \text{ and a.\,e.~in }   (0,T)\, ,\\
 \{\mu_t\} &\subset \mathcal{M}^+(\ov\Omega)\,,\text{ a.\,e.~in $ (0,T)$} \, ,  \\
\{ \nu^{\mu}_{(\f x ,t)}\}&\subset \mathcal{P}(\Se^{3^3-1})\,, \text{ $\mu_t$-a.\,e.~in }\ov\Omega \text{ and a.\,e.~in }   (0,T)\, 
\end{split}\label{measreg}
\end{align}
and if
\begin{subequations}\label{meas}
\begin{align}
\begin{split}
&\int_0^T (\t \f v(t), \f \varphi(t)) \de t + \int_0^T ((\f v(t)\cdot \nabla) \f v(t), \f \varphi(t)) \de t  - \intte{ \ll{\nu_t,\f S^T F_{\f S}(\f h, \f S):\nabla \f \varphi(t)   }  }
%\\ 
%& 
%+ \intte{ ( \langle \f S^T F_{\f S}^{OF}(\f h, \f S) , \nu^\infty \rangle ; \nabla \f \varphi ) _m  } 
\\&-2 \int_0^T \ll{\mu_t, \f \Gamma \dreidots (\f \Gamma\cdot \nabla\f \varphi(t))}\de t + \intte{(\f T^L(t): \nabla \f \varphi(t) ) } =\intte{ \left \langle \f g (t),\f \varphi(t)\right \rangle }\, ,\quad
\end{split}\label{eq:velo}
\intertext{as well as}
\begin{split}
&\intte{\left ( \f d (t) \times\left ( \t \f d(t)+( \f v(t)\cdot \nabla ) \f d(t) -  \sk{v(t)}\f d(t) + \lambda \sy{v(t)} + \f q\right ), \f \psi(t)\right )}=0
\end{split}
\label{eq:mdir}
%\\&
%+\lambda\intte{\left(\f d (t) \times  \sy{v(t)} \f d(t) , \f \psi(t)\right)}
%
\intertext{with}
\begin{split}
&\intte{\left ( \f d (t) \times \f q , \f \psi(t)\right )} ={} 
 \intte{(\rot{\f d(t)}  F_{ \f S} ( \f d(t) , \nabla \f d (t)) ; \nabla \f \psi(t) )} \\ &  \hspace{3.5cm}+\intte{ \ll{\nu_t, \f \Upsilon :\left (\f  S    (F_{\f S}(\f h, \f S))^T\right ) \cdot\f \psi(t)   }}+\intte{ \ll{\nu_t, \left (\f h \times F_{\f h}(\f h, \f S)\right ) \cdot\f \psi(t)   }}\,, \quad
\end{split}
\label{eq:q}
\end{align}%
\end{subequations}
holds for all $ \f \varphi \in \mathcal{C}_c^\infty(\Omega \times ( 0,T);\R^3))$ with $ \di \f \varphi =0$ and $ \f \psi \in  \mathcal{C}_c^\infty(\Omega \times ( 0,T);\R^3))$, respectively.
Additionally, the norm restriction of the director holds, i.\,e. $|\f d (\f x,t)|=1$ for a.\,e. $(\f x, t)\in \Omega\times (0,T)$, the oscillation measure of a linear function is the gradient of the director
\begin{align}
  \int_{\R^{3\times 3} } \f S  \nu^o_{(\f x, t)} ( \de \f S)   = \nabla \f d(\f x,t) \, \quad \text{for a.e. } (\f x ,t) \in \Omega\times (0,T)\,,\label{identify}
\end{align}
and the
initial conditions $( \f v_0, \f d_0)\in \Ha \times \Hc$ with $ \f d_0 \in \Hrand{7}$ shall be fulfilled in the weak sense and the boundary conditions in the sense of the trace.
We remark that the trace is well defined for the function $\f d \in L^\infty(0,T;\He)$, which is the expected value of the oscillation measure $\nu^o$. 

The dual pairings are defined as
\begin{align*}
 \ll{\mu_t,f} :={}& \int_{\ov \Omega} \int_{\Se ^{d^3-1}}  \sum_{i,j=1}^3f(\f \Gamma) \nu^\mu_{(\f x ,t)} ( \de \f \Gamma) \mu_t(\de \f x )\,\intertext{ for $f \in \C(\Se^{3^3-1};\R) $  and}
\ll{\nu_t, f } :={}& 
%\int_{\Omega} \langle \nu^o_{(\f x, t)},  f(\f x, t , \f d(\f x,t), \f S )  \rangle\de \f x + \int_{\ov\Omega}\langle \nu_{(\f x, t)}^\infty, \tilde{f}( \f x,t , \f h , \f S)  \rangle m_t (\de \f x)\\
%={}&
 \int_{\Omega} \int_{\R^{d\times d} } f(\f x, \f d(\f x, t), \f S)  \nu^o_{(\f x, t)} ( \de \f S)\de \f x \\&+ \int_{\ov\Omega}\int_{\Se^{d^2-1} \times \ov B_d} \tilde{f} (\f x,  \tilde{\f h} , \tilde{\f S}) \nu_{(\f x, t)}^\infty ( \de \tilde{\f S}, \de \tilde{\f h}) m_t (\de \f x)
\end{align*}
for $f\in \mathcal{R}$ (see~\eqref{transi} below).
\end{defi}
We refer to the section~\ref{sec:not} for the definition of the tensor $\f \Upsilon$ and to \eqref{transi} for the definition of the transformed function~$\tilde{f}$.
\begin{rem}
We often abuse  the notation by writing $ \ll{\nu_t , f(\f h , \f S)}$. Thereby we mean the generalized Young measure applied to the continuous function $(\f h, \f S)\mapsto f (\f h ,\f S)$. 
\end{rem}
We refer to the section~\ref{sec:not} for the definition of the tensor $\f \Upsilon$.
The transformed function $\tilde{f}:\ov \Omega  \times  B_d \times B_{d \times d}\ra \R$, the so-called recession function is given by
\begin{align}
\tilde{f} ( {\f x }, \tilde{\f h} ,\tilde{\f  S} ) : = 
f ( \f x , \frac{\tilde{\f h}}{\sqrt{1-|\tilde{\f h}|^2}}, \frac{\tilde{\f S}}{\sqrt{1-|\tilde{\f S}|^2}}))  ( 1-|\tilde{ \f h}|^2  )( 1- | \tilde{\f S}|^2) \,. \label{transi}
\end{align}
The class of function for which the above representation is valid are those functions, $f\in \C(\ov \Omega  \times \R^d \times \R^{d\times d})$ such that $\tilde{f}$ admits a continuous extension on the closure of its domain
\begin{align*}
\mathcal{F}:= \Big \{ f \in \C ( \ov \Omega \times [0,T] \times \R^d \times \R^{d\times d }  ) |& \exists \tilde{ g}\in \C(\ov \Omega \times [0,T] \times \ov B_d\times \ov B_{d\times d}\, ;\\ &\, \tilde{f}= \tilde{g}\text{ on }\ov \Omega \times [0,T] \times B_d \times B_{d\times d}  	\Big \}\,.
\end{align*} 
An straightforward calculation shows that function with quadratic growth in $\f h$ and $\f S$ is unchanged by the transformation~\eqref{transi}.
Most of the appearing terms in the above definition have this structure. This implies that the transformation of $\f h \times F_{\f h}(\f h ,\f S)$ remains the function itself. Only the linear terms in $F_{\f S}$ are changed by multiplying them with $1-|\tilde{\f h}|^2$, such that for example
\begin{align*}
\widetilde{\f S ^T F_{\f S}}(\tilde{\f h}, \tilde{\f S})= \tilde{\f S}^T F_{\f S}(\tilde{\f h}, \tilde{\f S})- k_1 | \tilde{ \f h }|^2 \tr{(\tilde{\f S})}\tilde{\f S}^T - k_2| \tilde{\f h}|^2 \tilde{\f S}^T ( \tilde{\f S})_{\skw}\,.
\end{align*}
\begin{rem}
We often use some abuse of the notation by writing $ \ll{\nu_t , f(\f h , \f S)}$. Thereby we mean the generalized Young measure applied to the continuous function $(\f h, \f S)\mapsto f (\f h ,\f S)$. 
\end{rem}
In comparison to weak solutions (see~\cite{unsere}) The Ericksen stress $\f T^E$ and the variational derivative $\f d \times \f q$ are in this measure-valued formulation represented by generalized Young measures. 
A \textbf{generalized Young measure} on $\ov \Omega \times [0,T]$ with values in $\R^d \times \R^{d\times d}$ is a triple $( \nu_{(\f x,t)}^o , m_t , \nu_{(\f x,t)}^\infty) $ consisting of  
\begin{itemize}
\item a parametrized family of probability measures $\{\nu_{(\f x,t)}^o\} \subset \mathcal{P}(\R^{d\times d})$ for a.\,e. $(\f x,t)  \in \Omega \times (0,T)$,
\item a positive measure $\{m_t\}\subset \M^+(\ov \Omega)$, for a.\,e. $t\in (0,T)$ and
\item a parametrized family of probability measures $\{ \nu_{(\f x,t)}^\infty\}\subset \mathcal{P}( \ov B_d \times \Se^{d^2-1})$, for $m_t$-a.\,e.~$\f x  \in \ov\Omega $ and a.\,e.~$t\in(0,T)$. 
\end{itemize}
As in~\cite[page 552]{rindler} we call $\nu^o$ \textit{oscillation measure}, $m_t$ \textit{concentration measure} and $\nu^\infty$ the \textit{concentration angle measure}. 

\begin{defi}
\label{def:gradmeas}
A generalized Young measure is said to be a generalized gradient Young measure, if there exists a sequence of Functions $\{ \f d _k\}\subset L^\infty(0,T; \Hc)$ with $ \sup_k \left \| \nabla \f d_k(|\f d_k|+1)\right \|_{L^\infty(\Le)}<\infty$ and a function $\f d\in L^\infty(0,T;\He)$  such that the convergence
\begin{align*}
\int_{\Omega} f( \f x, t , \f d_{k} ( \f x,t ) , \nabla \f d_{k} ( \f x,t) ) \de \f x \ra \ll{\nu_t,f  }\, , 
\end{align*}
 and formula~\eqref{identify} holds for all functions $f\in\F$ and \eqref{identify} holds for a.e.~$(\f x ,t) \in \Omega \times (0,T)$ .
\end{defi}
\begin{rem}\label{rem:boundary}
The approximating sequence $\f d_k$ can be chosen such that it fulfills the prescribed Dirichlet boundary conditions (compare to \cite[Thm.~2.1]{masswertig}).

\end{rem}
In the case of the Ericksen stress, an additional defect measure is of need to describe the limit of the regularised system we considered in~\cite{unsere}.
A \textbf{defect measure} on $\ov \Omega \times [0,T]$ with values in $ \R^{d\times d\times d }$
is a pair $(\mu_t, \nu^\mu )$ consisting of 
\begin{itemize}
\item a positive measure $\mu_t\in\M^+(\ov\Omega)$, for a.\,e. $t\in (0,T)$ and
\item a parametrized family of probability measures $\{ \nu_y^\mu\}_{y\in \ov Q} \in \mathcal{P}(  \Se^{d^3-1})$, for $\mu_t$-a.\,e.~$\f x \in \ov \Omega$ and a.\,e.~$t\in(0,T)$. 
\end{itemize}
We refer to~\cite{masswertig} for more details on the convergence in the sense of generalized Young measures.

In order to prove the weak-strong uniqueness property, it is needed that the measure-valued solution fulfills an additional assumption, it has to satisfy an energy inequality.  
\begin{defi}[Suitable measure-valued solutions]
A measure-valued solution is said to be a suitable measure-valued solution if it fulfills Definition~\ref{def:meas} and additionally the energy inequality
\begin{align}
\begin{split}
 &\frac{1}{2}\|\f v (t)\|_{\Le}^2 + \ll{\nu_t, F} + \frac{1}{2}\ll{\mu_t , 1 }   + \inttet{(\mu_1+\lambda(\mu_2+\mu_3))\|\f d\cdot \sy{v}\f d\|_{L^2}^2 }  \\
& +\inttet{ \left [
  \mu_4 \|\sy{v}\|_{\Le}^2+( \mu_5+\mu_6-\lambda(\mu_2+\mu_3))\|\sy{v}\f d\|_{\Le}^2  +  \|\f d \times \f q\|_{\Le}^2\right ]}
\\
& \qquad\qquad \leq  \left ( \frac{1}{2}\|\f v_0 \|_{\Le}^2 + \F( \f d_0)\right )
 + \inttet{\left [\langle \f g , \f v \rangle 
 +(  ( \mu_2+ \mu_3) - \lambda ) \left (\f d \times  \f q , \f d \times \sy{v} \f d \right ) 
 \right ]}\, .
\end{split}
\label{energyin}
\end{align}
a.e. in $(0,T)$.

\end{defi}
\begin{rem}[Existence of suitable weak solutions]
\label{rem:ex}
In our recent work, we proved the existence of measure-valued solutions to the system~\eqref{eq:strong} in the sense of Definition~\ref{def:meas}. 
The proof relies on the existence of weak solutions $(\f v_\delta, \f d_\delta)$ to a regularised system, where the free energy~\eqref{frei} is changed by adding the regularising term $\delta| \Delta \f d |^2$. 
We obtain generalized Young measure-valued solutions for vanishing regularisation. 
 
 In the case of Parodi's relation~\eqref{parodi}, we can prove the following energy inequality for vanishing regularisation (see~\cite{masswertig})
\begin{align}
\begin{split}\label{weaklow}
& \frac{1}{2}\|\fd v(t) \|_{\Le}^2 +  \ll{\nu_t, F} + \frac{1}{2}\ll{\mu_t , 1 } + (\mu_1+\lambda^2)\inttet{\|\fd d\cdot \syd{v}\fd d\|_{L^2}^2} \\
& \quad +\inttet{\left [
  \mu_4 \|\syd{v}\|_{\Le}^2 + ( \mu_5+\mu_6-\lambda^2)\|\syd{v}\fd d\|_{\Le}^2
%   \right ]}\\
%& \quad + \inttet{ 
%\left [
+
\|\f d \times \f q\|_{\Le}^2 - \langle \f g , \fd v\rangle \right ]} \\
\leq &\frac{1}{2}\| \fd v(0)\|_{\Le}^2 + \frac{\delta}{2}\| \Delta \fd d(0)\|_{\Le}^2 + \F( \f d(0))  
% &\geq \Big ( \frac{1}{2}\|\f v \|_{\Le}^2 + \F( \f d)\Big )( t) + \intt{\mu_1\|\f d\cdot \sy{v}\f d\|_{L^2}^2} \\
%& \quad +\intt{
%  \mu_4 \|\sy{v}\|_{\Le}^2 + ( \mu_5+\mu_6-\lambda^2)\|\sy{v}\f d\|_{\Le}^2 }\\
%& \quad + \intt{\gamma \|\f q\|_{\Le}^2 - \langle \f g , \f v \rangle }
\, .
\end{split}
\end{align}
In the case that Parodi's relation~\eqref{parodi} does not hold, we can not prove the existence of suitable measure-valued solutions.

\end{rem}
\subsection*{Strong solutions and main theorem}
 \begin{defi}[strong solution]\label{def:strong}
A pair $( \vv , \dd )$ is called a strong solution if it fulfills system~\eqref{eq:strong} and exhibits the  regularity
 \begin{align}
 \vv &\in L^\infty(0,T; \Ha)\cap L^2(0,T; \f L^\infty)\cap  L^2 (0,T; \f W^{1,3}_{0,\sigma})\cap L^1(0,T; \f W^{1,\infty})\cap W^{1,2}(0,T;\Vd)  \, ,\notag\\ 
 \dd &\in L^\infty(0,T;\f W^{1,\infty
}) \cap  L^2( 0,T ; \f W^{2,3}) \cap  L^4( 0,T ; \f W^{1,6})\cap W^{1,1 }(0,T;  \f W^{1,3}\cap \f L^\infty )   \, ,\notag
 \intertext{as well as}
 \nabla \vv &\in  \C(\ov \Omega \times [0,T]) \quad \text{and} \quad  \nabla \dd \in \C(\ov\Omega \times [0,T])\,.
\label{strongreg}
 \end{align}
 \begin{rem}
 The continuity assumptions on the solution are especially needed to be able to insert the formulation of the measure-valued solution (see Definition~\ref{def:meas}). This assumptions can presumably be generalized.  The assumptions on the differentiability of $\vv$ with respect to $t$ follows from equation~\eqref{nav} and the other assumptions on $\vv$.
 \end{rem}
 \begin{rem}
 To our best knowledge, there is no existence result in the class of strong solutions for the Ericksen-Leslie model in its full generality. But there are similar results for simpler models, and they can possibly be generalized to the case presented here.
 \end{rem}
 
 \end{defi}
 \begin{theorem}\label{thm:main}
 Let $(\f v , \f d)$ be a suitable measure-valued solution to the Ericksen--Leslie model according to Definition~\ref{def:meas}
and $(\vv,\dd)$ a strong solution according to definition~\ref{def:strong} to the same initial and boundary values.  
 Then both solutions coincide, i.e.~$\f v = \vv $ and $\f d = \dd$ in $\Omega\times (0,T)$ such that 
 \begin{align*}
 \nu^o = \delta_{\nabla \dd} , \quad \mu_t \equiv 0 , \quad \text{and } m_t \equiv 0 ,.
 \end{align*}
 \end{theorem}
The \textit{proof} of this main result is carried out in the following sections. It is a simple consequence of Proposition~\ref{lem:main}.
\begin{rem}
In Proposition~\ref{lem:main} even a continuous dependence on the initial values is proven. 
\end{rem}
 
 \section{Preliminaries\label{sec:pre}}
 This section collects some auxiliary lemmas and the energy equality for the strong solution.

\begin{lemma}
\label{prop:lemvar}
Let $f,g\in L^1(0,T)$ be fulfilling
\begin{align*}
-\int_0^T \phi'(t) g(t) \de t \leq \int_0^T \phi(t) f(t) \de t  \quad \text{for all }\phi \in \C_c^\infty(0,T) \text{ with }\phi\geq 0\,. 
\end{align*}
Then it holds 
\begin{align*}
g(t)-g(s) \leq \int_s^t f(\tau ) \de \tau\quad \text{for a.\,e.~} t,s\in (0,T)\,. 
\end{align*}
\end{lemma} 
\begin{proof}
The proof is similar to the proof of the standard variational lemma, just by using the sequence of smooth functions $\phi_\varepsilon\in \C_c^\infty(0,T)$ defined by 
\begin{align*}
\phi_\varepsilon(\tau) := (\rho_\varepsilon \ast \chi_{[t,s]})(\tau) = \int_0^T \rho(\tau-r) \chi_{[t,s]}(r) \de r \,.
\end{align*}
Here the functions $\rho _\varepsilon$ are the usual mollifier functions (see~Emmrich~\cite[Definition 3.1.6]{emmrich}) and $\chi_{[t,s]}$ denotes the characteristic function. The limit $\varepsilon\ra 0$ gives the assertion. 
\end{proof}
 \begin{lemma}
 \label{lem:norm1}
 For the measure-valued solution, there holds
 \begin{align*}
 \t \f d = \rot{\f d}^T \rot {\f d} \t \f d \quad \text{in } L^2( 0,T; \f L^{\nicefrac{3}{3}})\,.
 \end{align*}
 For the strong solution, there holds
 \begin{align*}
 \t \dd = \rot{\dd}^T\rot{\dd} \t \dd \quad \text{in }L^\infty(0,T; \f L^\infty)\cap L^\infty(0,T;\f W^{1,3})\,.
\end{align*}  
 \end{lemma}
 \begin{proof}
 First, we remark that $\f d \in L^\infty(0,T;\f L^{\infty})$.
 Indeed, $ | \f d (\f x ,t) | =1$ a.e.~in $\Omega \times (0,T)$. 
  For a test function $\f \varphi \in \C_c^\infty(\Omega\times (0,T)) $ we can calculate
 \begin{align*}
 \intt{ \rot{\f d }^T \rot{\f d}\t \f d , \f \varphi } ={}& \intt{ (| \f d |^2 I - \f d \o \f d ) \t \f d , \f \varphi  } \\={}& \intt{\t \f d , \f \varphi} + \frac{1}{2} \intt{ | \f d |^2, \t ( \f d \cdot \f \varphi) }=\intt{\t \f d , \f \varphi} \,.
 \end{align*}
 The weak time derivative of $| \f d|^2$ vanishes, since it is constant a.\,e. in $ \Omega \times (0,T)$. 
 The density of the test functions in $L^2(0,T; \f L^6)$ proves the assertion. 
 
 The proof for the strong solution is similar. The additional regularity is due to~\eqref{strongreg}. 
 \end{proof}
\begin{lemma}\label{lem:vschlange}
For the measure-valued solution, there holds \begin{align*}
(\f v \cdot \nabla)\f d = \rot{\f d }^T \rot{\f d} (\f v \cdot \nabla)\f d \quad \text{in } L^2( 0,T; \f L^{\nicefrac{3}{2}})\,.
\end{align*}
 For the strong solution, there holds
 \begin{align*}
 (\vv \cdot \nabla) \dd = \rot{\dd}^T\rot{\dd} (\vv \cdot \nabla) \dd \quad \text{in }L^\infty(0,T; \f L^\infty)\cap L^\infty(0,T;\f W^{1,3})\,.
\end{align*}  
\end{lemma}
\begin{proof}
For a test function  $\f \psi \in \C_c^\infty(\Omega\times (0,T)) $ we can calculate
\begin{align*}
\intte{\left (\rot{\f d }^T \rot{\f d }(\f v \cdot \nabla)\f d , \f \psi\right )}={}&  \intte{\left (| \f d |^2 (\f v \cdot \nabla)\f d , \f \psi\right )} - \frac{1}{2}\intte{\left (( \f v  \cdot \nabla )| \f d |^2, \f \psi\cdot \f d \right )}\\={}&  \intte{\left ( (\f v \cdot \nabla)\f d , \f \psi\right )} + \frac{1}{2}\intt{\left (| \f d |^2,\di ( \f v  \f \psi\cdot \f d) \right )}\,.
\end{align*}
Since $| \f d |=1$ a.\,e.~in $\Omega\times (0,T)$, the weak derivatives of $| \f d |^2$ in the last term on the right-hand side of the forgoing equation  vanishes.
 The density of the test functions in $L^2(0,T; \f L^3)$ proves the assertion. 
 
 The proof for the strong solution is similar. The additional regularity is due to~\eqref{strongreg}. 

\end{proof}

\begin{corollary}\label{prop:e}
 For the measure-valued solution, there holds
 \begin{align*}
 \f e = \t \f d +( \f v \cdot \nabla ) \f d - \sk v \f d = \rot{\f d }^T \rot{\f d } \f e = -\rot{\f d }^T \rot{\f d } \left (\lambda \sy v \f d +  \f q\right )
%  = \rot{\f d }^T \rot{\f d } \left (\t \f d +( \f v \cdot \nabla ) \f d - \sk v \f d\right )
  \quad \text{in } L^2( 0,T; \f L^{{2}})\,.
 \end{align*}
 For the strong solution, there holds
 \begin{align*}
 \te =\t \dd + ( \vv \cdot \nabla )\dd - \skv \dd =  \rot{\dd}^T\rot{\dd} \te= - \rot{\dd}^T\rot{\dd} ( \lambda \syv \dd + \tq)  \quad \text{in }L^\infty(0,T; \f L^3)\,.
\end{align*}  
\end{corollary}
\begin{proof}
Since $\rot{\f d}^T \rot{\f d}= I - \f d \o \f d $ and $\f d \cdot \sk v \f d = 0$, the identity $ \f e =  \rot{\f d }^T \rot{\f d } \f e $ in $ L^2( 0,T; \f L^{\nicefrac{6}{5}})$  is obvious from Lemma~\ref{lem:norm1} and Lemma~\ref{lem:vschlange}.

The term $\f e$ can be estimated by 
\begin{align*}
\| \f e \|_{L^2(L^{\nicefrac{3}{2}})}\leq \| \t \f d \|_{L^2(\f L^{\nicefrac{3}{2}})}+ \| \f v \|_{L^2(\f L^6)} \| \f d \|_{L^\infty(\f W^{1,2})} + \| \f v \|_{L^2( \He)} \| \f d \|_{L^\infty( \f L^6)} \,.
\end{align*}
Due to the energy inequality~\eqref{energyin} and equation~\eqref{eq:mdir} we can estimate
\begin{align*}
\| \rot{\f d }^T \rot{\f d }\f e\|_{L^2( \f L^{{2}})}\leq{}& 
%\| \f d \|_{L^\infty( \f L^\infty)} \| \f d \times \f e \|_{L^2(\Le)} \leq 
\| \f d \|_{L^\infty( \f L^6)} \left (|\lambda|\| \f d \times  \sy v \f d  \|_{L^2(\Le)} + \| \f d \times \f q \|_{L^2(\Le)}\right )\\
\leq{}& c\| \f d \|_{L^\infty( \f L^\infty)} \left (\| \f d\cdot  \sy v \f d  \|_{L^2(L^2)}+\| \sy v \f d  \|_{L^2(\Le)}+\| \f d \times \f q \|_{L^2(\Le)}\right ) \,.
\end{align*}
The density of the test functions in $L^2(\f L^{{2}})$ and equation~\eqref{eq:mdir} gives the first assertion.
The second one follows similar.

\end{proof}
  \begin{proposition}
 Let $(\vv,\dd)$ be a strong solution  according to Definition~\ref{def:strong}.
 Then it fulfils the following energy equality
 \begin{align*}
\begin{split}
 &\frac{1}{2}\|\vv (t)\|_{\Le}^2 + \mathcal{F}(\dd(t))   + \inttet{\left [(\mu_1+\lambda(\mu_2+\mu_3))\|\dd\cdot \syv\dd\|_{L^2}^2 +
  \mu_4 \|\syv\|_{\Le}^2 \right ]}  \\
& +\inttet{\left [( \mu_5+\mu_6-\lambda(\mu_2+\mu_3))\|\syv\dd\|_{\Le}^2  +  \|\dd \times \tq\|_{\Le}^2\right ]}
\\
& \qquad =  \left ( \frac{1}{2}\|\vv(0) \|_{\Le}^2 + \F( \dd(0))\right )
 + \inttet{\left [\langle \f g , \vv \rangle 
 +(  ( \mu_2+ \mu_3) - \lambda ) \left ( \dd \times \tq , \dd \times \syv \dd \right ) 
 \right ]}\, .
\end{split}
\end{align*}
 \end{proposition}
 \begin{proof}
 The proof is similar to the proof of~Proposition~4.1 in~\cite{unsere}. Here we only focus on the necessary modification.
 Equation~\eqref{nav} is tested with $\vv$ and equation~\eqref{dir} with $\dd \times \tq $. 
Remark that the identities  $ \rot{\dd}^T \rot{\dd} = I - \dd \o \dd $ , $ \t | \dd|^2=0 $ as well as $  (\vv \cdot\nabla)| \dd|^2=0$ hold for the unit vector $\dd$. 
For the Leslie-stress~$\f T^L$ tested with $\syv$, we use Corollary~\ref{prop:e} to be able to insert equation~\eqref{dir} twice
\begin{align*}
  (\tilde{\f T}^L& ; \syv )- ( \dd \times \skv \dd , \dd \times \tq )\\
  ={}& \mu_1 | \dd \cdot \syv \dd|^2 +  \mu_4 |\syv|^2 + ( \mu_5 + \mu_6) | \syv \dd|^2 + ( \mu_2 + \mu_3) (\te ,\syv \dd) 
  \\&+ \lambda ( \syv \dd , \skv \dd) + ( \te , \skv \dd) - ( \dd \times \skv \dd , \dd \times \tq )\\
  ={}&(\mu_1+ \lambda(\mu_2+\mu_3)) | \dd \cdot \syv \dd|^2 +  \mu_4 |\syv|^2 + ( \mu_5 + \mu_6- \lambda(\mu_2+\mu_3)) | \syv \dd|^2 \\
  &-  ( \mu_2 + \mu_3) (\dd \times \tq  ,\dd \times \syv \dd) 
  %+ ( \dd \times \skv \dd , \dd \times \tq ) -( \dd \times \skv \dd , \dd \times \tq )
  \,.
\end{align*}
Summing up both tested equations and integrating in time gives the desired energy equality.

 \end{proof}

\begin{lemma}\label{lem:vab}
Let $\f v $ be a measure-valued solution (see~Definition~\ref{def:meas})  and $\vv$ a strong solution (see~Definition~\ref{def:strong}). It holds that  
\begin{align*}
 \| \f v - \vv \|_{\f L^6} ^2 + \| \sk v - \skv \|_{\Le}^2 
\leq c\|\sy v-\syv \|_{\Le}^2 
 \, .
\end{align*} 
\end{lemma}
\begin{proof}
This Lemma is a simple application of the Sobolev embedding in three dimensions and Korn's inequality~(see~\textsc{McLean} \cite[Theorem 10.1]{mclean}).
\end{proof}

%%%%%%%%%%%%%%%%%%%%%%%%%%%%%
%%%%%%%%%%%%%%%%%%%%%%%%%%%%
%%%%%%%%%%%%%%%%%%%%%%%%%%%
\section{Relative energy\label{sec:rel}}
The relative energy is defined by 
\begin{align}
\begin{split}
\mathcal{E}(t) :
={}& \frac{1}{2} \left \| \f v(t) - \vv(t) \right \| _{\Le}^2+ \frac{1}{2}\ll{\mu_t,1} +  \frac{1}{2}\ll{\nu_t ,(\f S - \nabla \dd( t) ) : \f \Lambda : (\f  S- \nabla \dd( t)) } \\& +\frac{1}{2} \ll{\nu_t ,(\f S\o \f h - \nabla \dd( t) \o \dd( t) ) \dreidots \f \Theta  \dreidots (\f S\o \f h  
 - \nabla \dd( t) \o \dd( t) )}
  \,
  \end{split}
  \label{relEn}
\end{align}
and the relative dissipation by
\begin{align}
\begin{split}
\mathcal{W}(t) : ={}& (\mu_1+\lambda(\mu_2+\mu_3)) \| \f d(t) \cdot ( \nabla \f v (t))_{\sym} \f d(t) - \dd(t) \cdot ( \nabla \vv(t))_{\sym } \dd (t)\|_{\Le} ^2\\
&+ ( \mu_5 + \mu_6 -\lambda(\mu_2+\mu_3))) \| ( \nabla \f v(t) )_{\sym} \f d(t) - ( \nabla \vv(t))_{\sym} \dd(t)\|_{\Le}^2\\& + \mu_4 \| ( \nabla \f v(t))_{\sym} - ( \nabla \vv(t))_{\sym} \|_{\Le}^2  +  \| \f d(t) \times \f q(t) - \dd(t) \times \tq(t) \|_{\Le}^2 \, . 
\end{split}
\label{relW}
\end{align}
Inserting the definitions of the tensors $\f \Lambda $ and $\f \Theta $,
%as well as the definition of the generalized Young measures, 
the relative energy can be expressed as
\begin{align*}
\mathcal{E}(t)={}& \frac{1}{2}\int_\Omega \mu_t(\de \f x) + 
 \frac{1}{2}\ll{\nu_t ,k_1(\tr{(\f S)} - \tr{(\nabla \dd( t) )})^2   + 2 k_2 | (\f S)_{\skw} - ( \nabla\dd(t))_{\skw} |^2 }\notag \\&+\frac{1}{2} \ll{\nu_t, k_3| \tr(\f S) \f h - (\di \dd(t)) \dd(t) |^2 + k_4((\f S)_{\skw}:  [ \f h]_{\f X} - ( \nabla \dd(t))_{\skw}:  [ \dd(t) ]_{\f X})^2    }\notag
 \\
  &+\frac{1}{2} \ll{\nu_t,  4 k_5  | (\f S)_{\skw} \f h - ( \nabla\dd(t))_{\skw} \dd (t)|^2   } + \frac{1}{2} \left \| \f v(t) - \vv(t) \right \| _{\Le}^2\,.
\end{align*}
We remark that due to the regularity shown in~\cite{masswertig}, there holds $\mathcal{E}\in L^\infty(0,T)$ and $\mathcal{W}\in L^1(0,T)$. The overall goal is to apply Gronwalls estimate to the inequality of the form 
\begin{align*}
\mathcal{E}(t) + \int_0^t \mathcal{W}(s)\de s \leq c_0+  \zeta \int_0^t \mathcal{W}(s) \de s + \int_0^t  \mathcal{K}(s) \mathcal{E}(s) \de s \,,
\end{align*}
with $\zeta<1$ and $\mathcal{K}\in L^1(0,T)$. The constant $c_0$ is a constant depending on the initial values of the solutions $(\f v , \f d)$ and $(\vv ,\dd)$, which vanishes if those are the same. 

The term $\mathcal{K}$ is given by
\begin{align*}
\mathcal{K}(s) ={}& C_{\delta}  \left (  \|\vv(s)\|_{\f L^\infty}^2+ \| \vv(s)\|_{\f W^{1,3}}^2 + \| \dd\|_{\f W^{2,3}}^2 + \| \dd (s)\|_{\f W^{1,6}}^4 + \| \t \dd(s)\|_{\f L^\infty} + \| \t \dd(s)\|_{\f W^{1,3}} \right ) \\ & +C_{\delta}  \left ( \|(\nabla\vv(s))_{\sym}\|_{\f L^\infty}+1  \right )\,,
\end{align*}
where $C_\delta$ is a possible large constant depending on the norms $\| \f d \|_{L^\infty(\f L^\infty)} $, $\| \dd \|_{L^\infty(\f W^{1,\infty})}$ and $\delta$. It is obvious that $\mathcal{K}$ is bounded in $L^1(0,T)$ due to the regularity of the strong solution. 
%%%%%%%%%%%%%%%%%%%%%%%%%%%%%%%%%%%%%%%
%%%%%%%%%%%%%%%%%%%%%%%%%%%%%%%%%%%%%%%%
%%%%%%%%%%%%%%%%%%%%%%%%%%%%%%%%%%%%%%%%%%
For the generalized Young measures, a result similar to the Sobolev embedding theorem holds true.
\begin{lemma}\label{Sobolev}
Let $\f d$ be a measure-valued solution (see~Definition~\ref{def:meas}) and $( \nu,m,\nu^\infty) $ the associated generalized Young measure and $\dd$ a strong solution (see~Definition~\ref{def:strong}) to the same initial and boundary values. Let the associated relative Energy be given as above (see\eqref{relEn}). Then there exists a constant $c>0$ such that
\begin{subequations}
\begin{align}
\int_0^t \| \f d(s)- \dd (s)  \|_{\f L^6}^2 \de s + \int_0^t \ll{\nu_s, | \f S - \nabla \dd(s) |^2}\de s \leq  c {}&\int_0^t \mathcal{E}(s)\de s \,,\label{absch1}\\
\int_0^t \ll{\nu_s , \left | \f \Theta \dreidots \left ( \f S \o \f d - \nabla \dd (s)\o \dd(s) \right ) \right |^2 }\de s \leq c {}&\int_0^t \mathcal{E}(s)\de s \,,\label{absch2}\\
\int_0^t \ll{\nu_s , | \f S - \nabla \dd (s)|^2 |\f h - \dd(s) | ^2 } \de s \leq c(1 +\|\nabla \dd \|_{L^\infty(\f L^3)} ^2 + \| \dd \|_{L^\infty(\f L^\infty)}^2){}& \int_0^t \mathcal{E}(s)\de s\,,\label{absch3}
\\
\int_0^t \| \f d (s) - \dd(s)\|_{\f L^{12}}^4 \de s \leq c(1 +\|\nabla \dd \|_{L^\infty(\f L^3)} ^2 + \| \dd \|_{L^\infty(\f L^\infty)}^2){}& \int_0^t \mathcal{E}(s)\de s\,.\label{absch4}
\end{align}
\end{subequations}
\end{lemma}

\begin{proof}
Since $\f d$ and $\dd$ are measure and strong solutions to the same boundary data, the difference fulfills Dirichlet boundary conditions, i.\,e.~$\f d(t) - \dd(t) \in \Hb$. 
The function $\f d-\dd$ gives rise to a generalized gradient Young measure. 
Due to Definition~\ref{def:gradmeas} there exists a sequence $\{\fn d \}$ such that $\fn d \ra \f d - \dd $ in the sense of generalized Young measures. Due to Remark~\ref{rem:boundary} the sequence fulfills homogeneous boundary conditions.
%%%%%%%%%%%%%%%%%%%%%%%%%%%%%%%%%%%%%%%%
%%%%%%%%%%%%%%%%%%%%%%%%%%%%%%%%%%%%%%%%%%
%%%%%%%%%%%%%%%%%%%%%%%%%%%%%%%%%%%%%%%%
For $\f d_n$, \cite[Proposition~4.2]{masswertig} yields
\begin{align}
\| \f d_n  (t)\|_{\Hb} ^2 \leq{}& c (  \nabla \f d_n (t); \f \Lambda :\nabla \f d_n(t))\,,\notag\\ \int_\Omega | \f d_n (\f x,t)|^2 | \nabla \f d_n(\f x,t)|^2 \de \f x \leq{}&  c\left (   \nabla \f d_n(t) \o \f d_n (t) \dreidotkom \f \Theta \dreidots \nabla \f d_n(t) \o \f d_n (t)  \right ) \,.\label{abnon}
\end{align}
The convergence result for generalized gradient Young measures for the sequence $\{ \f d_n \}$ applied to the test functions $(\f h, \f S) \mapsto |\f S|^2$ gives
 \begin{align*}
\int _0^T \phi(t) \left |\nabla \f d_n(t)\right |^2 \de t \ra \int _0^T \phi(t) \ll{\nu_t , | \f S - \nabla \dd(t)|^2 }\de t \,
\end{align*}
and for the test function  $( \f h , \f S) \mapsto \f S : \f \Lambda : \f S $ the convergence
\begin{align*}
\int _0^T \phi(t) \left (  \nabla \f d_n (t); \f \Lambda :\nabla \f d_n(t)\right ) \de t \ra \int _0^T \phi(t) \ll{\nu_t , ( \f S - \nabla \dd(t)) : \f \Lambda : ( \f S - \nabla \dd(t)) }\de t \,
\end{align*}
Since $m_t$ is a positive measure and $\nu^o$ is a probability measure, we get with Jensen's inequality 
\begin{align*}
\int _0^T \phi(t) \ll{\nu_t , | \f S - \nabla \dd(t)|^2 }\de t \geq{}& \int_0^T \phi(t)\int_{\Omega} \int_{\R^{3\times 3}} | \f S - \nabla \dd(\f x , t)|^2 \nu^o_{(\f x ,t) } ( \de \f S) \de \f x \de t \\ \geq{}& \int_0^T \phi(t)\int_{\Omega} \left |\int_{\R^{3\times 3}} ( \f S - \nabla \dd(\f x ,t) )\nu^o_{(\f x ,t) } ( \de \f S) \right |^2 \de \f x \de t\,.
\end{align*}
The right-hand side can be identified due to~\eqref{identify} with the weak limit $\nabla \f d $, such that with the embedding in three dimensions
\begin{align*}
\int _0^T \phi(t) \ll{\nu_t , | \f S - \nabla \dd(t)|^2 }\de t \geq{}& \int_0^T \phi(t) \| \nabla \f d (t) - \nabla \dd(t) \|_{\Le}^2\de t \geq  c \int_0^T \phi(t) \|  \f d (t) -  \dd(t) \|_{\f L^6}^2\de t  \,.
\end{align*}
With the fundamental lemma of variational calculus the estimate~\eqref{absch1} is obvious.

The estimate~\eqref{absch2} follows, if one uses the purely algebraic inequality 
\begin{align}
  \left | \f \Theta \dreidots  \left (\f S \o \f h- \nabla \dd \o \dd \right ) \right |^2 \leq  c \left (\f S \o \f h- \nabla \dd \o \dd \right )  \dreidots \f \Theta \dreidots \left (\f S \o \f h- \nabla \dd \o \dd \right )  \,.\label{alin}
\end{align} 
The proof of the inequality~\eqref{alin} is deferred to the Appendix.

Choosing the same approximating sequence as above, application of the convergence result for generalized gradient Young measures
applied to the test function $( \f h , \f S) \mapsto |\f h |^2 | \f S|^2$ gives
\begin{align*}
 \int _0^T \phi(t)\left  \| \nabla \f d_n(t)  | \f d_n(t) | \right \|^2\de t \ra \int _0^T \phi(t) \ll{\nu_t , | \f S - \nabla \dd(t)|^2| \f h - \dd(t)|^2  }\de t \,,
\end{align*}
and for the test function
$( \f h , \f S) \mapsto \f S \o \f h \dreidots \f \Theta \dreidots \f S \o \f h$ the convergence
\begin{multline*}
 \int _0^T \phi(t) \left (   \nabla \f d_n(t) \o \f d_n (t) \dreidotkom \f \Theta \dreidots \nabla \f d_n(t) \o \f d_n (t)  \right ) \de t \ra \\\int _0^T \phi(t) \ll{\nu_t , ( \f S - \nabla \dd(t))\o ( \f h - \dd(t))  : \f \Theta : ( \f S - \nabla \dd(t))\o ( \f h - \dd(t))  }\de t \,.
\end{multline*}
%%%%%%%%%%%%%%%%%%%%%%%%%%%%%%%%%%%%%
%%%%%%%%%%%%%%%%%%%%%%%%%%%%%%%%%%%%%
%%%%%%%%%%%%%%%%%%%%%%%%%%%%%%%%%%%%%
Together with~\eqref{abnon} we get 
\begin{align*}
 \ll{\nu_s, | \f S- \nabla \dd | ^2 | \f h - \dd|^2} \de s \leq c  \ll{\nu_s, (\f S- \nabla \dd)  \o (\f h- \dd)   \dreidots  \f \Theta \dreidots  (\f S- \nabla \dd)  \o (\f h- \dd)  } \,
\end{align*}
a.\,e.~in $(0,T)$. 
A regrouping of the terms shows
\begin{align*}
( \f S - \nabla \dd ) \o( \f h - \dd ) ={}& \f S \o \f h - \nabla \dd \o \f h - \f S \o \dd + \nabla \dd \o \dd \\
 ={}& ( \f S \o \f h - \nabla \dd \o \dd ) + ( \nabla \dd \o \dd - \nabla \dd \o \f h ) + ( \nabla \dd \o \dd - \f S \o \dd )\\
 ={}& ( \f S \o \f h - \nabla \dd \o \dd ) +  \nabla \dd \o ( \dd -\f h) + ( \nabla \dd - \f S ) \o \dd \,
\end{align*}
and, thus
\begin{align*}
& \ll{\nu_s, (\f S- \nabla \dd)  \o (\f h- \dd)   \dreidots  \f \Theta \dreidots  (\f S- \nabla \dd)  \o (\f h- \dd)  }\\ \leq{}&  \ll{ \nu_s, ( \f S \o \f h - \nabla \dd \o \dd)\dreidots \f \Theta \dreidots ( \f S \o \f h - \nabla \dd \o \dd) } +  \ll{ \nu_s,  \left | \f \Theta \dreidots \left ( \f S \o \f h - \nabla \dd \o \dd \right ) \right | ^2} \\&+ c \| \nabla \dd \|_{\f L^3} ^2  \| \dd - \f d \|_{\f L^6}^2 + c \| \dd \|_{\f L^\infty} ^2  \ll{\nu_s, | \nabla \dd- \f S|^2  } \,
\end{align*}
a.\,e.~in $(0,T)$. 
The estimates~\eqref{absch1} and~\eqref{absch2} imply inequality~\eqref{absch3}.

To prove the last inequality, we consider the left-hand side of~\eqref{absch3}.
Since $m_t$ is a positive measure and $\nu^o$ is a probability measure, we get with Jensen's inequality 
\begin{align*}
\int _0^T \phi(t) \ll{\nu_t , | \f S - \nabla \dd(t)|^2 | \f h - \dd(t)|^2 }\de t \geq{}& \int_0^T \phi(t)\int_{\Omega} \int_{\R^{3\times 3}} | \f S - \nabla \dd(\f x, t)|^2 \nu^o_{(\f x ,t) } ( \de \f S)| \f d(\f x ,t)- \dd(\f x ,t)|^2 \de \f x \de t \\ \geq{}& \int_0^T \phi(t)\int_{\Omega} \left |\int_{\R^{3\times 3}} ( \f S - \nabla \dd(\f x ,t) )\nu^o_{(\f x ,t) } ( \de \f S) \right |^2 | \f d(\f x ,t)- \dd(\f x ,t)|^2 \de \f x \de t\,.
\end{align*}
The right-hand side can be identified due to~\eqref{identify} with the weak limit $\nabla \f d $, such that with the embedding in three dimensions
\begin{align*}
\int _0^T \phi(t) \ll{\nu_t , | \f S - \nabla \dd(t)|^2 | \f h - \dd(t)|^2  }\de t \geq{}& \int_0^T \phi(t)\left  \| \nabla \f d (t) - \nabla \dd(t) | \f d (t) - \dd(t)|\right  \|_{\Le}^2\de t \\\geq{}&
\int_0^T \phi(t)\left  \| \nabla |\f d (t) -  \dd(t)|^2\right  \|_{\Le}^2\de t\\
\geq {}&
  c \int_0^T \phi(t) \left \|  |\f d (t) -  \dd(t) |^2\right \|_{\f L^6}^2\de t\\
  \geq {}&
  c \int_0^T \phi(t) \left \|  \f d (t) -  \dd(t) \right \|_{\f L^{12}}^4\de t  \,.
\end{align*}
Remark that the difference of both solutions fulfills homogeneous Dirichlet boundary conditions.
With the fundamental lemma of variational calculus the estimate~\eqref{absch4} is obvious.

\end{proof}

%%%%%%%%%%%%%%%%%%%%%%%%%%%%%%%%%%%%%%%%%%%
%%%%%%%%%%%%%%%%%%%%%%%%%%%%%%%%%%%%%%%%%%%%
%%%%%%%%%%%%%%%%%%%%%%%%%%%%%%%%%%%%%%%%%%%%%%%
%%%%%%%%%%%%%%%%%%%%%%%%%%%%%%%%%%%%%%%%%%%%%%%%
\begin{lemma}\label{lem:timederi}
Let $\f d$ and $\dd$ be given as in Lemma~\ref{Sobolev} as well as $\f a \in L^1(0,T, \f L^3)$. 
Then it holds 
\begin{align*}
\int_0^t \left (  \rot{\f d(s) } \t \f d(s) - \rot{\dd(s)} \t \dd(s)  ,\left ( \rot{\f d(s) } - \rot{\dd(s)} \right ) \f a(s) \right ) \de s \leq  \int_0^t \mathcal{K}(s)\mathcal{E}(s)  + \delta  \mathcal{W}(s)\de s\,. 
\end{align*}
Here $\mathcal{K}$ is given by $\mathcal{K}= C_\delta \left ( \| \vv\| _{\f L^\infty}^2 + \| \vv\|_{\f W^{1,3}}^2+ \| \f a \|_{L^3}^2\right )(\| \f d \|_{L^\infty(\f L^\infty)}\| \dd \|_{L^\infty(\f L^\infty)}^2  \| \dd \|_{L^\infty(\f W^{1,3})}^2+1)$ and $C_\delta$ is a constant depending on $\delta$.

\end{lemma}
\begin{proof}
First we insert equation~\eqref{eq:mdir} and \eqref{dir} for the measure-valued solution and the strong solution respectively. This gives
\begin{subequations}\label{gleichung}
\begin{align}
\int_0^t \left (  \rot{\f d      -\dd      }^T \left ( \rot{\f d       } \t \f d       - \rot{\dd      } \t \dd       \right ) , \f a       \right ) \de s 
={}& \int_0^t \left ( \left ( \rot{\f d       } - \rot{\dd      } \right )^T \left ( \rot{\dd      } ( \vv       \cdot \nabla ) \dd      - \rot{\f d       } ( \f v        \cdot \nabla ) \f d         \right ) , \f a       \right ) \de s\label{gleichung1}\\
&+  \int_0^t \left ( \left ( \rot{\f d       } - \rot{\dd      } \right )^T \left (      \rot{\f d       } \sk v  \f d -\rot{\dd      } \skv \dd          \right ) , \f a       \right ) \de s\label{gleichung2}\\
&+ \lambda \int_0^t \left ( \left ( \rot{\f d}     -\rot{\dd      } \right )^T \left ( \rot{\dd      } \syv \dd      - \rot{\f d       } \sy v\f d         \right ) , \f a       \right ) \de s\label{gleichung3}\\
&+  \int_0^t \left ( \left ( \rot{\f d       } - \rot{\dd      } \right )^T \left ( \rot{\dd      } \tq      - \rot{\f d       } \f q          \right ) , \f a       \right ) \de s\,.\label{gleichung4}
\end{align}
\end{subequations}
The dependence on $s$ is not written out to remain the lucidity.
We are going to estimate the lines individually
For the line~\eqref{gleichung1} we observe
\begin{align*}
\int_0^t& \left ( \left ( \rot{\f d       } - \rot{\dd      } \right )^T \left ( \rot{\dd      } ( \vv       \cdot \nabla ) \dd      - \rot{\f d       } ( \f v        \cdot \nabla ) \f d         \right ) , \f a       \right ) \de s \\
={}&
 \int_0^t \left ( \left ( \rot{\f d      -\dd      } \right )^T \left ( (\rot{\dd      -\f d }) ( \vv       \cdot \nabla ) \dd + \rot{\f d } ( \vv \cdot \nabla) ( \dd - \f d )           + \rot{\f d } \left (( \vv - \f v ) \cdot\nabla\right ) \f d  \right ) , \f a       \right ) \de s
\\
={}&
 \int_0^t \left ( \left ( \rot{\f d      -\dd      } \right )^T \left ( (\rot{\dd      -\f d }) ( \vv       \cdot \nabla ) \dd + \rot{\f d- \dd } ( \vv \cdot \nabla) ( \dd - \f d )     + \rot{ \dd } ( \vv \cdot \nabla) ( \dd - \f d )        \right ) , \f a       \right ) \de s\\
&+\int_0^t \left ( \left ( \rot{\f d       -\dd      } \right )^T \left (  \rot{\f d } \left (( \vv - \f v ) \cdot\nabla\right )( \f d-\dd)        \right )          , \f a       \right ) \de s\\
&+\int_0^t \left ( \left ( \rot{\f d       -\dd      } \right )^T \left ( \rot{\f d -\dd } \left (( \vv - \f v ) \cdot\nabla\right ) \dd                    +\rot{\dd } \left (( \vv - \f v ) \cdot\nabla\right ) \dd            \right )          , \f a       \right ) \de s\,,
\end{align*}
which can be estimated by
\begin{align*}
\int_0^t& \left ( \left ( \rot{\f d       } - \rot{\dd      } \right )^T \left ( \rot{\dd      } ( \vv       \cdot \nabla ) \dd      - \rot{\f d       } ( \f v        \cdot \nabla ) \f d         \right ) , \f a       \right ) \de s \\
\leq{}&c \int_0^t\| \vv \|_{\ L^\infty} \|\nabla \dd \|_{\f L^3} \| \f a\|_{\f L^3} \| \f d - \dd \|_{\f L^6}^2\de s \\ 
&+ c \int_0^t \| \vv\|_{\f L^\infty} \| \f a\|_{\f L^3}\left (\| \f d - \dd \|_{\f L^{12} } ^4+  \| \dd \|_{\f L^\infty} \| \f d - \dd \|_{\f L^6 } ^2 + \ll{\nu_s, | \f S - \nabla \dd | ^2 } \right ) \de s  \\
& +  \delta \int_0^t \| \vv - \f v \|_{\f L^6}^2 \de s + C_\delta  + \int_0^t\| \f d \|_{\f L^\infty}^2\| \f a\|_{\f L^3}^2  \ll{\nu_s, | \f S - \nabla \dd | ^2 | \f h -\dd|^2 }   \de s \\
& +  C_\delta   \int_0^t\| \f d \|_{\f W^{1,3}}^2\| \f a\|_{\f L^3}^2  \left ( \| \f d - \dd \|_{\f L^{12} } ^4+  \| \dd \|_{\f L^\infty}^2 \| \f d - \dd \|_{\f L^6 } ^2 \right ) \de s  \,.
\end{align*}
For line~\eqref{gleichung2}, we get a similar rearrangement
\begin{align*}
  \int_0^t& \left ( \left ( \rot{\f d       } - \rot{\dd      } \right )^T \left (      \rot{\f d       } \sk v  \f d -\rot{\dd      } \skv \dd          \right ) , \f a       \right ) \de s\\
={}&\int_0^t \left ( \left ( \rot{\f d       } - \rot{\dd      } \right )^T \left (   \rot{\f d       } \left (\sk v - \skv\right ) \f d + \rot{\f d } \skv ( \f d - \dd)      +   \left (\rot{\f d }-\rot{\dd      }\right ) \skv \dd  \right ) , \f a       \right ) \de s\\
%&+  \int_0^t \left ( \left ( \rot{\f d       } - \rot{\dd      } \right )^T \left (   \rot{\f d } \skv ( \f d - \dd)      +   \left (\rot{\f d }-\rot{\dd      }\right ) \skv \dd   \right ) , \f a       \right ) \de s\\
={}&\int_0^t \left ( \left ( \rot{\f d       } - \rot{\dd      } \right )^T \left (   \rot{\f d       } \left (\sk v - \skv\right ) (\f d - \dd) +\rot{\f d       } \left (\sk v - \skv\right )  \dd \right ) , \f a       \right ) \de s\\
&+  \int_0^t \left ( \left ( \rot{\f d       } - \rot{\dd      } \right )^T \left (   \rot{\f d } \skv ( \f d - \dd)      +   \left (\rot{\f d }-\rot{\dd      }\right ) \skv \dd   \right ) , \f a       \right ) \de s
\end{align*}
such that the terms of line~\eqref{gleichung2} can be estimated by
\begin{align*}
 \int_0^t& \left ( \left ( \rot{\f d       } - \rot{\dd      } \right )^T \left (      \rot{\f d       } \sk v  \f d -\rot{\dd      } \skv \dd          \right ) , \f a       \right ) \de s\\
\leq{}&
\delta \int_0^t \| \sk v - \skv \|_{\Le}^2 \de s + C_\delta \int_0^t  \| \f d \|_{\f L^\infty} ^2 \| \f a \|_{\f L^3}^2\left (\| \f d - \dd \|_{\f L^{12}}^4 +\| \dd \|_{\f L^\infty}^2 \| \f d - \dd \|_{\f L^6}^2\right )\de s \\
&+ c \int_0^t\left ( \| \f d \|_{\f L^\infty} +\| \dd \|_{\f L^\infty} \right ) \| \skv \|_{\f L^3}\| \f a \|_{\f L^3}  \| \f d - \dd \|_{\f L^6}^2\de s\,.
\end{align*}
For line~\eqref{gleichung3} we observe
\begin{align*}
 \lambda \int_0^t& \left ( \left ( \rot{\f d       } - \rot{\dd      } \right )^T \left ( \rot{\dd      } \syv \dd      - \rot{\f d       } \sy v\f d         \right ) , \f a       \right ) \de s\\
={}& \lambda \int_0^t \left ( \left ( \rot{\f d       -\dd      } \right )^T \left (   \left (  \rot{\dd-\f d       }\right ) \syv\dd     + \rot{\f d    } \left (\syv \dd - \sy v \f d \right )      \right ) , \f a       \right ) \de s
\\
={}& \lambda \int_0^t \left ( \left ( \rot{\f d       -\dd      } \right )^T \left (   \left (  \rot{\dd-\f d       }\right ) \syv\dd     + \rot{\f d  - \dd     } \left (\syv \dd - \sy v \f d \right )      \right ) , \f a       \right ) \de s
\\
& \lambda \int_0^t \left ( \left ( \rot{\f d       -\dd      } \right )^T    \rot{ \dd     } \left (\syv \dd - \sy v \f d     \right ) , \f a       \right ) \de s
\,.
\end{align*}
Thus, the terms of lines \eqref{gleichung3} and \eqref{gleichung4} can be estimated by
\begin{align*}
 \int_0^t& \left ( \left ( \rot{\f d       } - \rot{\dd      } \right )^T \left ( \lambda\rot{\dd      } \syv \dd      - \lambda\rot{\f d       } \sy v\f d         \right ) +\rot{\dd      } \tq      - \rot{\f d       } \f q  , \f a       \right ) \de s\\
\leq{}& c \int_0^t\| \dd \|_{\f L^\infty}  \| \syv \|_{\f L^3} \|\f a \|_{\f L^3} \| \f d - \dd \|_{\f L^6 }^2 \de s \\
&+ \delta \int_0^t \| \sy v \f d - \syv \dd \|_{\Le}^2 \de s + C_{\delta}  \int_0^t \| \f a\|_{\f L^3}^2\left (\| \f d - \dd \|_{\f L^{12} }^4+\| \dd \|_{\f L^\infty}   \| \f d - \dd \|_{\f L^6 }^2 \right )\de s\\
&+ \delta \int_0^t \| \f d \times \f q - \dd \times \tq  \|_{\Le}^2 \de s + C_{\delta }  \int_0^t \| \f a\|_{\f L^3}^2 \| \f d - \dd \|_{\f L^6}^2\de s
\end{align*}

The above estimates together with Lemma~\ref{lem:vab} and the definitions~\eqref{relEn} and~\eqref{relW} prove the assertion.

\end{proof}

%%%%%%%%%%%%%%%%%%%%%%%%%%%%%%%%%%%%%%%%%%
%%%%%%%%%%%%%%%%%%%%%%%%%%%%%%%%%%%%%%%%%%%%
%%%%%%%%%%%%%%%%%%%%%%%%%%%%%%%%%%%%%%%%%%%
\section{Integration-by-parts formulae\label{sec:int}}
\begin{proposition}[Integration-by-parts fomula]\label{prop:int}
Let $(\f v, \f d)$ be a measure-valued solution and $( \vv , \dd )$ a strong solution. Then the following integration-by-parts formulae hold for a.e. $t,s\in(0,T)$. 
\begin{align}
( \f v ( t) , \vv(t)) - ( \f v (s) , \vv(s)) = \int_s^t( \f v ( \tau) , \t \vv (\tau) ) + ( \t \f v ( \tau ) , \vv( \tau)) \de \tau \, ,\label{intv}\\
\| \f d(t) - \dd(t) \|_{\Le}^2- \| \f d(s) -\dd(s)\|_{\Le}^2  \leq \delta\inttes{\mathcal{W}(\tau)}+ C_\delta \inttes{\mathcal{E}(\tau) }
% 2 \int_s^t ( \t\f d (\tau) - \t \dd(\tau), \f d ( \tau) - \dd(\tau)) \de \tau 
 \, , \label{intd1}
\end{align}
\begin{align}
\begin{split}
\left ( \nabla \f d(t) ;\, \f \Lambda : \nabla \dd(t)\right ) &- \left (  \nabla \f d(s) ;\, \f \Lambda : \nabla\dd(s)\right )\\ 
\leq{}&  \inttes{\left [ \ll{\nu_t,  \f S   :\f \Lambda :\left (\f \Upsilon : \left (\rot{\dd}\t \dd \o \f S \right ) \right )}+ \left ( \nabla \f d ; \f \Lambda : \rot{\f d }^T \nabla ( \rot{\dd} \t \dd )   \right )\right ]}\\
&- \inttes{\left ( \rot{\dd}  \Lap \dd , \rot{\f d } \t \f d  \right )}+ \delta \inttes{\mathcal{W}(\tau)}+ c \inttes{\mathcal{E}(\tau)} \,,
\end{split}\label{intd2}
\end{align}
 and
 \begin{subequations}\label{intd}
\begin{align}
\big ( \nabla \f d(t) \o \f d (t) \dreidotkom \f \Theta&  \dreidots \nabla \dd(t) \o \dd (t) \big )
-
\left ( \nabla \f d(s) \o \f d (s) \dreidotkom \f \Theta \dreidots \nabla \dd(s) \o \dd (s) \right )\notag\\
&-
 \left ((\nabla \f d (t)-\nabla \dd(t)  )\o (\f d(t) -  \dd(t) )\right ) \dreidots \f \Theta \dreidots (\nabla \dd(t)  \o \dd(t) )\label{intdd5} \\&+\left ((\nabla \f d (s)-\nabla \dd(s)  )\o (\f d(s) -  \dd(s) )\right ) \dreidots \f \Theta \dreidots (\nabla \dd(s)  \o \dd(s) )\notag
\\
\leq {}& \inttes{\phi\left ( \rot{\f d} \t \f d , \rot{\dd} \left (- \di \left ( \dd \cdot \f \Theta \dreidots \nabla \dd \o \dd \right )+ \nabla \dd : \f \Theta \dreidots \nabla \dd \o \dd \right ) \right )}\label{intdd1}
\\
&+ \inttes{\phi \ll{\nu_\tau , \f S \o \f h \dreidots \f \Theta \dreidots \left (\f \Upsilon : ( \rot \dd^T  \t \dd \o \f S )\right ) \o \f h }}\label{intdd2}
\\&+ \inttes{\phi \left ( \nabla \f d \o \f d \dreidotkom \f \Theta \dreidots \rot{\f d}^T \nabla \left (\rot \dd \t \dd \right ) \o \f d \right )}\label{intdd3}\\
&+ \inttes{\phi \ll{\nu_\tau , \f S \o \f h \dreidots \f \Theta \dreidots \f S \o \rot{\f h}^T \rot \dd \t \dd }}+ \delta \inttes{\mathcal{W}} + c \inttes{\mathcal{E}}\,.\label{intdd4}
\end{align}
\end{subequations}

\end{proposition}
\begin{proof}
 Let $\phi\in \C_c^\infty(0,T)$ we calculate the weak time derivative and neglect the dependence on the time variable $t$ under the integral for more convenient writing.

Remark that $\f v$ and $\vv$ fulfill the regularity assumptions
\begin{align*}
\f v \in L^2(0,T; \V)\cap W^{1,2}(0,T; ( \f W^{1,3}_{0,\sigma})^*) \quad \text{and} \quad \vv\in L^2(0,T; (\f W^{1,3}_{0,\sigma})^*) \cap W^{1,2}(0,T; \Vd)\,.
\end{align*}  
 We consider the weak derivative of the $\Le$-product of $\f v $ and $\vv$
 \begin{align*}
 \int_0^T \phi'(t) ( \f v (t) , \vv(t)) \de t = \int_0^T \phi(t) \left ( ( \t \f v(t) , \vv(t)) + ( \f v (t) , \t \vv(t)) \right ) \de t \,. 
 \end{align*}
 The variational lemma implies the assertion~\eqref{intv}.
 
To prove the formula~\eqref{intd1} on can simply calculate the weak derivative, since the regularity of $\f d$ is sufficient to do so:
\begin{align*}
\intte{\phi'\left  \| \f d - \dd \right \|_{\Le}^2 } =2  \intte{ \phi \left ( \t \f d - \t \dd , \f d - \dd \right )}
\end{align*}
Applying Lemma~\ref{lem:norm1} and regrouping the terms shows
\begin{align*}
\intte{\phi'\left  \| \f d - \dd \right \|_{\Le}^2 } ={}& 2  \intte{ \phi \left ( \rot{\f d}^T \rot{\f d} \t \f d - \rot \dd ^T \rot \dd \t \dd , \f d - \dd \right )}\\
={}& 2 \intte{\phi \left (  \rot \dd \t \dd , \left ( \rot \dd - \rot { \f d} \right ) ( \f d - \dd)\right )} \\
&+ 2 \intte{\phi \left ( \rot { \f d} \t \f d - \rot\dd \t \dd , (\rot{ \f d} - \rot \dd ) \dd \right )}
\end{align*}
The right hand side can be estimated by  Lemma~\ref{lem:timederi} and by
\begin{align*}
\intte{\phi \left (  \rot \dd \t \dd , \left ( \rot \dd - \rot { \f d} \right ) ( \f d - \dd)\right )} \leq \| \dd\|_{L^\infty(\f L^\infty)}\| \t \dd\|_{L^\infty(\f L^\infty)} \intte{\phi \| \f d - \dd \|_{\Le}^2 } \,.
\end{align*}
  Lemma~\ref{prop:lemvar} implies the formula~\eqref{intd1}.

For the next integration-by-parts formula~\eqref{intd2}, the weak time derivative of the quadratic term of the gradient of the director is calculated. Remark that there is no additional regularity known for the time derivative of $\nabla \t \f d $. We write
\begin{align*}
-\intte{\phi' \left ( \nabla \f d ; \f \Lambda : \nabla \dd \right )}
= \intt{\phi\left (\nabla \f d ; \f \Lambda : \nabla \t \dd \right )- \left ( \nabla \f d ; \f \Lambda :\t( \phi \nabla \dd )  \right )}
\end{align*}
%%%%%%%%%%%%%%%%%%%%
%%%%%%%%%%%%%%%%%%%%%%
%%%%%%%%%%%%%%%%%%%%%
The measure-valued solution is approximated by smooth functions $\{\f d_n\}$ fulfilling the same boundary conditions, such that $\f d_n \ra \f d $ in $L^2(0,T; \He)\cap W^{1,2}(0,T; \f L^{3/2})$.
For this approximation, the following integration-by-parts holds true
\begin{align*}
%\begin{split}
-\intt{    \nabla  \f d _n  ;  \t \left (\phi\f \Lambda: \nabla \dd   \right  ) } 
={}&\intte{  \phi \left ( \nabla \t  \f d_n ; \f \Lambda   : \nabla \dd  \right )} 
\\
={}& \intte{\phi \left ( \t   \f d  _n ,  - \Lap \dd\right )   } \,.
%\end{split}\label{nichtin}
\end{align*}
The boundary values vanish since  $\phi \in \C_c^\infty(0,T)$ and since the boundary values of $ \f d_n   $ are constant  in time and, hence, its time derivative vanishes on the boundary.
Going to the limit in the approximation gives 
\begin{align*}
-\intte{    \left ( \nabla  \fk d   ; \t \left (\phi \f \Lambda : \nabla \dd \right )\right ) } = \intte{\phi \left ( \t   \fk d   ,  -\Lap \dd \right )   } \,.
\end{align*}
 Lemma~\ref{lem:norm1} implies that
\begin{align}
\begin{split}
-\intte{    \phi' \left ( \nabla \f d ; \f \Lambda : \nabla \dd \right ) } 
={}& \intte{\phi\left ( \left ( \nabla \f d ; \f \Lambda :\nabla \left (\rot{\dd} ^T \rot{\dd} \t \dd\right )\right )+ \left (\rot{\f d }^T \rot{\f d} \t \f d , - \Lap \dd \right )\right )}
\\
={}& \intte{\phi\left ( \left ( \nabla \f d ; \f \Lambda :\left (\f \Upsilon:  \left (\rot{\dd} \t \dd\right )\otimes \nabla \dd + \rot{\dd }^T \nabla ( \rot{\dd} \t \dd )\right )   \right )+ \left (\rot{\f d }^T \rot{\f d} \t \f d , - \Lap \dd \right )\right )}
\end{split}
\end{align}
To show the formula~\eqref{intd2}, we add and subtract the desired terms of the right hand side of~\eqref{intd2}. 
Some additional rearrangement of the terms show
\begin{align}
\begin{split}
-\int_0^T&    \phi' \left ( \nabla \f d ; \f \Lambda : \nabla \dd \right ) \de t \\
={}& \intte{\phi  \left (\ll{\nu_t, \f S   :\f \Lambda :\f \Upsilon : \left (\rot{\dd}\t \dd \o \f S \right ) }+ \left ( \nabla \f d ; \f \Lambda : \rot{\f d }^T \nabla ( \rot{\dd} \t \dd )   \right )  -\left ( \rot{\dd}  \Lap \dd , \rot{\f d } \t \f d  \right )
\right )}\\
&+ \intte{ \phi\left ( \ll{\nu_t, \f S  : \f \Lambda : \f \Upsilon : \rot{\dd}\t \dd \o (  \nabla \dd - \f S )}
+ \left ( \nabla \f d ; \f \Lambda : \left (\rot{\dd}-\rot{\f d}\right ) \nabla ( \rot{\dd} \t \dd) \right )\right )
}
\\
%&+ \intte{\phi\left ( \nabla \f d ; \f \Lambda : \left (\rot{\dd}-\rot{\f d}\right ) \nabla ( \rot{\dd} \t \dd) \right )
%}\\
&+ \intte{\phi\left (( \rot{\dd}-\rot{\f d } )^T  \rot{\f d }\t \f d  ,  \Lap \dd \right )}
\\
%={}& \intte{\phi  \ll{\nu_t, \f S   :\f \Lambda :\f \Upsilon : \left (\rot{\dd}\t \dd \o \f S \right ) }+\phi \left ( \nabla \f d ; \f \Lambda : \rot{\f d }^T \nabla ( \rot{\dd} \t \dd )   \right )}\\
%&- \intte{\phi \left ( \rot{\dd}  \Lap \dd , \rot{\f d } \t \f d  \right )+ \phi\ll{\nu_t, (\f S -\nabla \dd ) : \f \Lambda : \f \Upsilon : \rot{\dd}\t \dd \o ( \f S - \nabla \dd )}}\\
%&+ \intte{\phi\left ( \nabla \f d -\nabla \dd ; \f \Lambda : \left (\rot{\dd}-\rot{\f d}\right ) \nabla ( \rot{\dd} \t \dd) \right )
%}\\
%&+ \intte{\phi\left (( \rot{\dd}-\rot{\f d } )^T ( \rot{\dd}\t \dd - \rot{\f d }\t \f d ) , - \Lap \dd \right )}
\end{split}
\end{align}
Adding and subtracting the term
\begin{align*}
 \intte{\phi\left (( \rot{\dd}-\rot{\f d } )^T  \rot{\dd}\t \dd  , - \Lap \dd \right )}= &
\intte{ \phi \left (\nabla \dd  : \f \Lambda : \f \Upsilon : \rot{\dd}\t \dd \o ( \nabla\dd - \nabla \f d )\right )}\\
&+ \intte{\phi\left ( \nabla \dd ; \f \Lambda : \left (\rot{\dd}-\rot{\f d}\right ) \nabla ( \rot{\dd} \t \dd) \right ) }
\end{align*}
gives 
\begin{subequations}\label{calc000}
\begin{align}
-\int_0^T&    \phi' \left ( \nabla \f d ; \f \Lambda : \nabla \dd \right ) \de t \notag\\
={}& \intte{\phi  \left (\ll{\nu_t, \f S   :\f \Lambda :\f \Upsilon : \left (\rot{\dd}\t \dd \o \f S \right ) }+ \left ( \nabla \f d ; \f \Lambda : \rot{\f d }^T \nabla ( \rot{\dd} \t \dd )   \right )  -\left ( \rot{\dd}  \Lap \dd , \rot{\f d } \t \f d  \right )
\right )}\notag\\
%&+ \intte{ \phi\left ( \ll{\nu_t, \f S  : \f \Lambda : \f \Upsilon : \rot{\dd}\t \dd \o (  \nabla \dd - \f S )}
%+ \left ( \nabla \f d ; \f \Lambda : \left (\rot{\dd}-\rot{\f d}\right ) \nabla ( \rot{\dd} \t \dd) \right )\right )
%}
%\\
%%&+ \intte{\phi\left ( \nabla \f d ; \f \Lambda : \left (\rot{\dd}-\rot{\f d}\right ) \nabla ( \rot{\dd} \t \dd) \right )
%%}\\
%&+ \intte{\phi\left (( \rot{\dd}-\rot{\f d } )^T  \rot{\f d }\t \f d  ,  \Lap \dd \right )}
%\\
%={}& \intte{\phi  \ll{\nu_t, \f S   :\f \Lambda :\f \Upsilon : \left (\rot{\dd}\t \dd \o \f S \right ) }+\phi \left ( \nabla \f d ; \f \Lambda : \rot{\f d }^T \nabla ( \rot{\dd} \t \dd )   \right )}\\
&+\intte{ \phi\left (\ll{\nu_t, (\f S -\nabla \dd ) : \f \Lambda : \f \Upsilon : \rot{\dd}\t \dd \o (  \nabla \dd- \f S  )} 
+\left ( \nabla \f d -\nabla \dd ; \f \Lambda : \left (\rot{\dd}-\rot{\f d}\right ) \nabla ( \rot{\dd} \t \dd) \right )
\right )}\label{calc00}\\
%&+ \intte{\phi\left ( \nabla \f d -\nabla \dd ; \f \Lambda : \left (\rot{\dd}-\rot{\f d}\right ) \nabla ( \rot{\dd} \t \dd) \right )
%}\\
&+ \intte{\phi\left (( \rot{\dd}-\rot{\f d } )^T ( \rot{\dd}\t \dd - \rot{\f d }\t \f d ) , - \Lap \dd \right )}
\label{calc0}
\end{align}
\end{subequations}
We can estimate  the terms of line~\eqref{calc00} by 
\begin{align*}
%&\intte{\phi\left ( \ll{\nu_t, (\f S -\nabla \dd ) : \f \Lambda : \f \Upsilon : \rot{\dd}\t \dd \o ( \f S - \nabla \dd )}}\\
%&+\intte{\phi\left ( \nabla \f d -\nabla \dd ; \f \Lambda : \left (\rot{\dd}-\rot{\f d}\right ) \nabla ( \rot{\dd} \t \dd) \right )
%}\\
\int_0^T& \phi\left (\ll{\nu_t, (\f S -\nabla \dd ) : \f \Lambda : \f \Upsilon : \rot{\dd}\t \dd \o (  \nabla \dd- \f S  )} 
+\left ( \nabla \f d -\nabla \dd ; \f \Lambda : \left (\rot{\dd}-\rot{\f d}\right ) \nabla ( \rot{\dd} \t \dd) \right )
\right )
\\
\leq{}& c \| \dd \|_{L^\infty(\f L^\infty)}  \intte{\phi\| \t\dd \|_{\f L^\infty} \ll{\nu_t, | \f S- \nabla \dd|^2}} \\&+ c  \intte{\phi( \| \nabla \dd\|_{L^\infty(\f L^3)} \| \t \dd\|_{\f L^\infty}+ \|  \dd\|_{L^\infty(\f L^\infty)}\| \nabla \t\dd\|_{\f L^3})\left ( \ll{\nu_t, |\f S -\nabla \dd|^2}+ \| \f d - \dd \|_{\f L^6}^2\right )}
\end{align*}
and terms of line~\eqref{calc0} by Lemma~\ref{lem:timederi}. 
Inserting those estimates back into \eqref{calc000} and applying 
Lemma~\ref{prop:lemvar} implies the integration-by-parts formula~\eqref{intd2}.
%%%%%%%%%%%%%%%%%%%%
%%%%%%%%%%%%%%%%%%%%%
%%%%%%%%%%%%%%%%%

%%%%%%%%%%%%%%%%%%%%%%%%%%%%%%%%%%%
%%%%%%%%%%%%%%%%%%%%%%%%%%%%%%%%%%%
%%%%%%%%%%%%%%%%%%%%%%%%%%%%%%%%%%%

Finally, we are going to prove the integration-by-parts formula~\eqref{intd}.
Consider the following term
\begin{align*}
-\int_0^T&\phi'\left  (\nabla \fk d \o \fk d \dreidotkom \f \Theta \dreidots \nabla \dd \o \dd\right ) + \intte{\left ( (\nabla \fk d-\nabla  \dd)  \o (\fk d-\dd) \dreidotkom \f \Theta \dreidots \t \left ( \phi \nabla \dd \o \dd\right )   \right )}\hspace{-2cm}\\
={}& - \intte{\left ( \nabla \fk d \o \fk d \dreidotkom \f \Theta \dreidots \t \left ( \phi\nabla \dd \o \dd\right )\right ) }+ \intte{\phi\left ( \nabla \fk d \o \fk d \dreidotkom \f \Theta \dreidots \t\left (\nabla \dd \o \dd\right ) \right )}\hspace{-2cm}
\\
&+ \intte{\left ( \nabla \fk d  \o \fk d \dreidotkom \f \Theta \dreidots \t \left ( \phi \nabla \dd \o \dd\right )   \right )}
- \intte{\left ( \nabla  \dd  \o \fk d \dreidotkom \f \Theta \dreidots \t \left ( \phi \nabla \dd \o \dd\right )   \right )}\hspace{-2cm}
\\
&
- \intte{\left ( \nabla \fk d  \o \dd \dreidotkom \f \Theta \dreidots \t \left ( \phi \nabla \dd \o \dd\right )   \right )}
+ \intte{\left ( \nabla  \dd  \o \dd \dreidotkom \f \Theta \dreidots \t \left ( \phi \nabla \dd \o \dd\right )   \right )}\hspace{-2cm}
\end{align*}
The first and the third term on the right-hand cancel each other.
Calculating the weak time derivatives of the products gives
\begin{align}
-\int_0^T&\phi'\left  (\nabla \fk d \o \fk d \dreidotkom \f \Theta \dreidots \nabla \dd \o \dd\right )\de t  + \intte{\left ( (\nabla \fk d-\nabla  \dd)  \o (\fk d-\dd) \dreidotkom \f \Theta \dreidots \t \left ( \phi \nabla \dd \o \dd\right )   \right )}\notag\\
={}&  \intte{\phi\left ( \nabla \fk d \o \fk d \dreidotkom \f \Theta \dreidots \nabla\t \dd \o \dd \right )}\notag
+ \intte{\phi\left ( \nabla \fk d \o \fk d \dreidotkom \f \Theta \dreidots \nabla \dd \o \t \dd \right )}
\\
&+ \intte{\phi\left ( \nabla  \dd  \o \t \fk d \dreidotkom \f \Theta \dreidots    \nabla \dd \o \dd  \right )}
+ \intte{\phi\left ( \nabla \t \dd  \o  \fk d \dreidotkom \f \Theta \dreidots    \nabla \dd \o \dd  \right )}\notag
\\
&
- \intte{\left ( \nabla  \fk d  ; \t \left (\phi \dd \cdot  \f \Theta \dreidots \nabla \dd \o \dd\right )   \right )}\label{joo}
+ \intte{\phi\left ( \nabla  \fk d  \o \t \dd  \dreidotkom \f \Theta \dreidots \nabla \dd \o \dd   \right )}
\\
&\notag
- \intte{\phi \left (  \nabla  \t\dd  \o \dd \dreidotkom \f \Theta \dreidots   \nabla \dd \o \dd \right )}
- \intte{\phi \left (  \nabla  \dd  \o \t\dd \dreidotkom \f \Theta \dreidots   \nabla \dd \o \dd \right )}
\end{align}
In the the first term of the line~\eqref{joo} the direktor of the measure-valued solution is again approximated by a sequence $\{ \f d_n \}\subset \C^1(\Omega \times (0,T))$ fulfilling the same boundary values, such that $ \f d_n \ra \f d$ in $L^2(0,T; \He)\cap W^{1,2}(0,T; \f L^{3/2})$. 
For this approximations, the following integration-by-parts is valid
\begin{align}
\begin{split}
-\intte{    \left ( \nabla  \f d_n   ; \t \left (\phi \dd  \cdot \f \Theta \dreidots (\nabla \dd  \o \dd )\right )\right ) } 
={}&\intte{    \left (  \nabla \t  \f d_n   , \phi \dd  \cdot \f \Theta \dreidots (\nabla \dd  \o \dd  )\right ) } 
\\
={}& \intte{\phi \left ( \t   \f  d_n   ,  -\di \left ( \dd  \cdot \f \Theta \dreidots (\nabla \dd  \o \dd  )\right )\right )   } \,.
\end{split}\label{nichtin}
\end{align}
The boundary terms vanish, since $\phi \in \C_c^\infty(0,T)$ and since the prescribed boundary value for $\f d_n$ is constant in time, so that is time derivative vanishes.
Going to the limit in the approximation gives 
\begin{align*}
-\intte{    \left ( \nabla  \fk d   ; \t \left (\phi \dd  \cdot \f \Theta \dreidots (\nabla \dd  \o \dd  )\right )\right ) } = \intte{\phi \left ( \t   \fk d   ,  -\di \left ( \dd  \cdot \f \Theta \dreidots (\nabla \dd  \o \dd  )\right )\right )   } \,.
\end{align*}

Adding and subtracting the terms
\begin{align*}
\intte{\phi \left (  \nabla \fk d \o \fk d \dreidotkom \f \Theta \dreidots \nabla \t \dd \o \f d   \right ) } + \intte{\phi \ll{\nu_t , \f S \o \f h \dreidots\f \Theta \dreidots \f S \o \t \dd   }}
\end{align*}
to~\eqref{joo} shows
\begin{subequations}\label{abs}
\begin{align}
&\hspace{-2em}-\intte{\phi'\left  (\nabla \fk d \o \fk d \dreidotkom \f \Theta \dreidots \nabla \dd \o \dd\right )} + \intte{\phi' \left ( (\nabla \fk d-\nabla  \dd)  \o (\fk d-\dd) \dreidotkom \f \Theta \dreidots   \nabla \dd \o \dd\right  )  }\notag\\
 &\hspace{-2em}+ \intte{\phi  \left ( (\nabla \fk d-\nabla  \dd)  \o (\fk d-\dd) \dreidotkom \f \Theta \dreidots   \t \left (  \nabla \dd \o \dd\right )   \right )}\hspace{-2cm} \label{abs1}\\
={}&  \intte{\phi \left ( \t \fk d ,  -\di \left ( \dd  \cdot \f \Theta \dreidots (\nabla \dd  \o \dd  )\right ) + \nabla \dd : \f \Theta \dreidots \nabla \dd \o \dd   \right )}\label{abs3}
\\
& + \intte{
\phi \left ( \nabla \t \dd \o \fk d \dreidotkom \f \Theta \dreidots   \nabla \fk d \o \fk d    \right ) }
+ \intte{\phi \ll{\nu_t , \f S \o \f h \dreidots \f \Theta \dreidots \f S \o \t \dd }}\label{abs4}\\
&+ \intte{ \phi \left ( \left ( \nabla \dd \o \dd - \nabla \fk d \o \fk d \right ) \dreidotkom \f \Theta \dreidots \nabla \t \dd \o ( \fk d - \dd ) \right ) } \label{abs5}\\
&+ \intte{\phi   \ll{\nu_t , \left ( \nabla \dd \o \dd - \f S \o \f h  \right )\dreidots \f \Theta \dreidots ( \f S -\nabla \dd ) \o \t \dd  }}\label{abs6}
\end{align}
\end{subequations}
We observe that the last two terms of the lines \eqref{abs5} and \eqref{abs6} can be estimated by
\begin{align*}
&\hspace{-2em}+ \intte{ \phi \left ( \left ( \nabla \dd \o \dd - \nabla \fk d \o \fk d \right ) \dreidotkom \f \Theta \dreidots \nabla \t \dd \o ( \fk d - \dd ) \right ) } \\
&\hspace{-2em}+ \intte{\phi   \ll{\nu_t , \left ( \nabla \dd \o \dd - \f S \o \f h  \right )\dreidots \f \Theta \dreidots ( \f S -\nabla \dd ) \o \t \dd  }}
\\
\leq{}&  \intte{\phi\| \nabla \t \dd\|_{ \f L^3}\left (\ll{\nu_t, \left | \f \Theta \dreidots \left (\f S \o \f h - \nabla \dd \o \dd \right )\right |^2 } + \| \f d -\dd \|_{\f L^6}^2 \right ) }
\\
&+   \intte{\phi\| \t \dd\|_{\f L^\infty}\left (\ll{\nu_t, \left | \f \Theta \dreidots \left (\f S \o \f h - \nabla \dd \o \dd \right )\right |^2 } + \ll{\nu_t , | \f S - \nabla \dd|^2 } \right ) }\,.
\end{align*}
The line~\eqref{abs1} can be estimated by 
\begin{align*}
&\hspace{-2em}\intte{\phi  \left ( (\nabla \fk d-\nabla  \dd)  \o (\fk d-\dd) \dreidotkom \f \Theta \dreidots   \t \left (  \nabla \dd \o \dd\right )   \right )} \\
 \leq{}& \| \dd \|_{L^\infty( \f L^\infty)} \intt{\phi\| \nabla \t \dd \|_{\f L^3} \ll{\nu_t, | \f S - \nabla \dd|^2} + \phi\left\| \f d - \dd   \right \|_{\f L^6}^2 } 
 \\
 &  \| \nabla  \dd \|_{L^\infty(\f L^3)}   \intt{\phi \| \t \dd \|_{ \f L^\infty} \ll{\nu_t, | \f S - \nabla \dd|^2} + \phi\left\| \f d - \dd   \right \|_{\f L^6}^2 } \,.
\end{align*}
The remaining terms of the right-hand side of~\eqref{abs}, in line~\eqref{abs3} and in line~\eqref{abs4} are considered further on.
Due to Lemma~\ref{lem:timederi} we see
\begin{subequations}\label{neu}
\begin{align}
\int^T_0&\phi \left ( \t \fk d ,  -\di \left ( \dd  \cdot \f \Theta \dreidots (\nabla \dd  \o \dd  )\right ) + \nabla \dd : \f \Theta \dreidots \nabla \dd \o \dd   \right )\de t \notag
\\
 +& \intte{
\phi \left ( \nabla \t \dd \o \fk d \dreidotkom \f \Theta \dreidots   \nabla \fk d \o \fk d    \right ) }
+ \intte{\phi \ll{\nu_t , \f S \o \f h \dreidots \f \Theta \dreidots \f S \o \t \dd }}\notag
\\
={}& \intte{\phi \left (  \rot{\f d }^T \rot{\f d} \t \fk d ,  -\di \left ( \dd  \cdot \f \Theta \dreidots (\nabla \dd  \o \dd  )\right ) + \nabla \dd : \f \Theta \dreidots \nabla \dd \o \dd   \right )}\label{neu1}
\\
& + \intte{
\phi \left ( \nabla \left ( \rot{\dd}^T \rot{\dd}\t \dd\right ) \o \fk d \dreidotkom \f \Theta \dreidots   \nabla \fk d \o \fk d    \right ) }
+ \intte{\phi \ll{\nu_t , \f S \o \f h \dreidots \f \Theta \dreidots \f S \o\rot{\dd}^T \rot{\dd}  \t \dd }}\,.\label{neu2}
\end{align}
\end{subequations}
Considering now the difference of the right-hand side of the above equality and the desired terms due to the integration-by-parts formula~\eqref{intd}, we get for the terms of the lines~\eqref{intdd1} and~\eqref{neu1} 
\begin{align}
\begin{split}
&\intte{\phi \left ( \rot{\f d }^T \rot{\f d} \t \fk d ,  -\di \left ( \dd  \cdot \f \Theta \dreidots (\nabla \dd  \o \dd  )\right ) + \nabla \dd : \f \Theta \dreidots \nabla \dd \o \dd   \right )}\hspace{-2cm}\\
&- 
\intte{\phi\left ( \rot{\f d} \t \f d , \rot{\dd} \left (- \di \left ( \dd \cdot \f \Theta \dreidots \nabla \dd \o \dd \right )+ \nabla \dd : \f \Theta \dreidots \nabla \dd \o \dd \right ) \right )}
\\
={}& 
\intte{\phi\left ( \rot{\f d} \t \f d - \rot \dd \t \dd , \left ( \rot{\f d}-\rot{\dd}\right ) \left (- \di \left ( \dd \cdot \f \Theta \dreidots \nabla \dd \o \dd \right )+ \nabla \dd : \f \Theta \dreidots \nabla \dd \o \dd \right ) \right )}\\
&+ \intte{\phi\left (  \rot \dd \t \dd , \left ( \rot{\f d}-\rot{\dd}\right ) \left (- \di \left ( \dd \cdot \f \Theta \dreidots \nabla \dd \o \dd \right )+ \nabla \dd : \f \Theta \dreidots \nabla \dd \o \dd \right ) \right )}\,,
\end{split}\label{calcu1}
\end{align}
as well as for the terms of the lines~\eqref{intdd3},~\eqref{intdd4} and the first term in line~\eqref{neu2}
\begin{align}
\begin{split}
& \hspace{-2em} \intte{
\phi \left ( \nabla \left ( \rot{\dd}^T \rot{\dd}\t \dd\right ) \o \fk d \dreidotkom \f \Theta \dreidots   \nabla \fk d \o \fk d    \right ) }
 \\
&\hspace{-2em} - \intte{\phi \ll{\nu_\tau , \f S \o \f h \dreidots \f \Theta \dreidots ( \rot \dd^T  \t \dd \o \f S ) \o \f h }}
- \intte{\phi \left ( \nabla \f d \o \f d \dreidotkom \f \Theta \dreidots \rot{\f d}^T \nabla \left (\rot \dd \t \dd \right ) \o \f d \right )}\\
={}& 
 \intte{\phi \ll{\nu_\tau , \f S \o \f h \dreidots \f \Theta \dreidots \left (\f \Upsilon : ( \rot \dd^T  \t \dd \o ( \nabla \dd - \f S ))\right ) \o \f h }}\\&
+\intte{\phi \left ( \nabla \f d \o \f d \dreidotkom \f \Theta \dreidots \left ( \rot{ \dd} -\rot{\f d}\right )^T \nabla \left (\rot \dd \t \dd \right ) \o \f d \right )}\,.
\end{split}\label{calcu2}
\intertext{Simmilar, we get for the terms in line~\eqref{intdd2} subtracted from the second term in line~\eqref{neu2} that}
\begin{split}
&
\hspace{-2em}\intte{\phi \ll{\nu_t , \f S \o \f h \dreidots \f \Theta \dreidots \f S \o\rot{\dd}^T \rot{\dd}  \t \dd }} -\intte{\phi \ll{\nu_\tau , \f S \o \f h \dreidots \f \Theta \dreidots \f S \o \rot{\f h}^T \rot \dd \t \dd }}\\
={}& \intte{\phi \ll{\nu_\tau , \f S \o \f h \dreidots \f \Theta \dreidots \f S \o \left ( \rot\dd- \rot{\f h}\right )^T \rot \dd \t \dd }}\,.
\end{split}\label{calcu3}
\end{align}
We recognise that the first term on the right hand side of~\eqref{calcu1} can be estimated by Lemma~\ref{lem:timederi},
The last term of equation~\eqref{calcu1} can, via an integration-by-parts, be expressed as
\begin{subequations}\label{calcu4}
\begin{align}
&\hspace{-2em}\intte{\phi\left (  \rot \dd \t \dd , \left ( \rot{\f d}-\rot{\dd}\right ) \left (- \di \left ( \dd \cdot \f \Theta \dreidots \nabla \dd \o \dd \right )+ \nabla \dd : \f \Theta \dreidots \nabla \dd \o \dd \right ) \right )}\notag
\\
={}& 
\intte{\phi \left ( \nabla \dd \o \dd \dreidotkom \f \Theta \dreidots  \left (\f \Upsilon : \left ( \rot{\dd}^T \t \dd \o (\nabla \f d-\nabla \dd)   \right ) \o \dd+ \left (\rot{\f d } - \rot \dd \right )^T \nabla (\rot\dd \t \dd) \o \dd   \right )   \right )\label{calcu41}
}
%+ \intte{\phi\left (  \nabla \dd \o \dd \dreidotkom \f \Theta \dreidots \left (\rot{\f d } - \rot \dd \right )^T \rot\dd \t \dd   \right )}
\\
&+ \intte{\phi \left ( \nabla \dd \o \dd \dreidotkom \f \Theta \dreidots \nabla \dd \o  \left ( \rot{\f d} - \rot \dd \right )^T \rot \dd \t \dd  \right )}\,.\label{calcu42}
\end{align}
\end{subequations}
First we consider the difference of the right-hand side of~\eqref{calcu2} and the terms of line~\eqref{calcu41}:
\begin{align*}
& \hspace{-2em}\intte{\phi \ll{\nu_\tau , \f S \o \f h \dreidots \f \Theta \dreidots \f \Upsilon : ( \rot \dd^T  \t \dd \o ( \nabla \dd - \f S )) \o \f h }}\\
&\hspace{-2em}+\intte{\phi \left ( \nabla \f d \o \f d \dreidotkom \f \Theta \dreidots \left ( \rot{ \dd} -\rot{\f d}\right )^T \nabla \left (\rot \dd \t \dd \right ) \o \f d \right )}\\
&\hspace{-2em}+ \intte{\phi \left ( \nabla \dd \o \dd \dreidotkom \f \Theta \dreidots  \left (\left (\f \Upsilon : \left ( \rot{\dd}^T \t \dd \o (\f S-\nabla \dd)   \right )\right ) \o \dd+ \left (\rot{\f d - \dd} \right )^T \nabla (\rot\dd \t \dd) \o \dd   \right )   \right )}\\
={}&\intte{\phi \ll{\nu_\tau , \left (\f S \o \f h - \nabla \dd \o \dd \right )\dreidots \f \Theta \dreidots \left ( \f \Upsilon : ( \rot \dd^T  \t \dd \o ( \nabla \dd - \f S ))\right ) \o (\f h- \dd )  }}\\&
+ \intte{\phi \ll{\nu_\tau , \left (\f S \o \f h - \nabla \dd \o \dd \right )\dreidots \f \Theta \dreidots \left (\f \Upsilon : ( \rot \dd^T  \t \dd \o ( \nabla \dd - \f S ))\right ) \o  \dd   }}\\
&+ \intte{\phi \left ( \nabla \dd \o \dd \dreidotkom \f \Theta \dreidots  \left (\f \Upsilon : \left ( \rot{\dd}^T \t \dd \o (\f S-\nabla \dd)   \right )\right ) \o ( \dd- \f d)    \right )}\\
& +\intte{\phi \left ( \left (\nabla \f d \o \f d - \nabla \dd \o \dd\right ) \dreidotkom \f \Theta \dreidots \left ( \rot{ \dd} -\rot{\f d}\right )^T \nabla \left (\rot \dd \t \dd \right ) \o (\f d- \dd) \right )}\\
& +\intte{\phi \left ( \left (\nabla \f d \o \f d - \nabla \dd \o \dd\right ) \dreidotkom \f \Theta \dreidots \left ( \rot{ \dd} -\rot{\f d}\right )^T \nabla \left (\rot \dd \t \dd \right ) \o \dd \right )}\\
&+ \intte{\phi \left ( \nabla \dd \o \dd \dreidotkom \f \Theta \dreidots  \left (\rot{\f d } - \rot \dd \right )^T \nabla (\rot\dd \t \dd) \o (\dd- \f d)    \right )}\,.
\end{align*}
This can now be estimated by 
\begin{align*}
& \hspace{-2em}\intte{\phi \ll{\nu_\tau , \f S \o \f h \dreidots \f \Theta \dreidots \f \Upsilon : ( \rot \dd^T  \t \dd \o ( \nabla \dd - \f S )) \o \f h }}\\
&\hspace{-2em}+\intte{\phi \left ( \nabla \f d \o \f d \dreidotkom \f \Theta \dreidots \left ( \rot{ \dd} -\rot{\f d}\right )^T \nabla \left (\rot \dd \t \dd \right ) \o \f d \right )}\\
&\hspace{-2em}+ \intte{\phi \left ( \nabla \dd \o \dd \dreidotkom \f \Theta \dreidots  \left (\f \Upsilon : \left ( \rot{\dd}^T \t \dd \o (\f S-\nabla \dd)   \right ) \o \dd+ \left (\rot{\f d } - \rot \dd \right )^T \nabla (\rot\dd \t \dd) \o \dd   \right )   \right )}\\
\leq{}& \| \dd\|_{L^\infty( \f L^\infty)}   \intte{\phi \| \t \dd\|_{ \f L^\infty} \ll{\nu_t , \left | \f \Theta \dreidots \left ( \f S \o \f h - \nabla \dd \o \dd  \right )\right |^2 + | \f S - \nabla \dd |^2 | \f h - \dd |^2 }}\\
&+ \| \dd\|_{L^\infty( \f L^\infty)}^2  \intte{\phi \| \t \dd\|_{ \f L^\infty}\ll{\nu_t , \left | \f \Theta \dreidots \left ( \f S \o \f h - \nabla \dd \o \dd  \right )\right |^2 + | \f S - \nabla \dd |^2  }}\\
&+\| \dd\|_{L^\infty( \f W^{1,3})}\| \dd\|_{L^\infty( \f L^\infty)}^2   \intte{\phi \| \t \dd\|_{ \f L^\infty}\left ( \ll{\nu_t ,  | \f S - \nabla \dd |^2  }+  \| \f d - \dd \|_{\f L^6}^2\right )}\\
&+  \| \dd\|_{L^\infty( \f W^{1,3})}  \intte{\phi\| \t \dd\|_{ \f L^\infty}\left ( \ll{\nu_t ,  \left | \f \Theta \dreidots \left ( \f S \o \f h - \nabla \dd \o \dd  \right )\right |^2  }+  \| \f d - \dd \|_{\f L^{12}}^4\right )}\\
&+ \| \dd\|_{L^\infty( \f L^\infty)} \intte{\phi\| \t \dd\|_{ \f W^{1,3}}\left ( \ll{\nu_t ,  \left | \f \Theta \dreidots \left ( \f S \o \f h - \nabla \dd \o \dd  \right )\right |^2  }+  \| \f d - \dd \|_{\f L^{12}}^4\right )}\\
&+  \| \dd\|_{L^\infty( \f W^{1,3})}\|  \dd\|_{L^\infty( \f L^\infty)}   \intte{\phi\| \t \dd\|_{ \f L^\infty}\left (\phi \ll{\nu_t ,  \left | \f \Theta \dreidots \left ( \f S \o \f h - \nabla \dd \o \dd  \right )\right |^2  }+ \phi \| \f d - \dd \|_{\f L^6}^2\right )}\\
&+ \| \dd\|_{L^\infty( \f L^\infty)}^2  \intte{\phi\| \t \dd\|_{ \f W^{1,3}}\left ( \ll{\nu_t ,  \left | \f \Theta \dreidots \left ( \f S \o \f h - \nabla \dd \o \dd  \right )\right |^2  }+  \| \f d - \dd \|_{\f L^6}^2\right )}\\
&+\| \dd\|_{L^\infty( \f W^{1,3})}( \| \dd\|_{L^\infty( \f L^\infty)}+  \| \dd\|_{L^\infty( \f W^{1,3})}) \intte{ \phi \left (\| \t \dd\|_{ \f W^{1,3}} +\| \t \dd\|_{ \f L^\infty}\right ) \| \f d - \dd \|_{\f L^6}^2}\,.
\end{align*}
All the terms on the right-hand side can be bounded due to Lemma~\ref{Sobolev}.

The last line of~\eqref{calcu3} and~\eqref{calcu42} can expressed as
\begin{align*}
&\hspace{-2em}\intte{\phi \ll{\nu_\tau , \f S \o \f h \dreidots \f \Theta \dreidots \f S \o \left ( \rot\dd- \rot{\f h}\right )^T \rot \dd \t \dd }}
%\\
%&\hspace{-2em}
+ \intte{\phi \left ( \nabla \dd \o \dd \dreidotkom \f \Theta \dreidots \nabla \dd \o  \left ( \rot{\f d} - \rot \dd \right )^T \rot \dd \t \dd  \right )}\\
={}& 
\intte{ \phi \ll{\nu_\tau ,\left ( \f S \o \f h - \nabla \dd \o \dd \right )\dreidots \f \Theta \dreidots \left (\f S- \nabla \dd \right ) \o \left ( \rot\dd- \rot{\f h}\right )^T \rot \dd \t \dd } }\\
& +
\intte{ \phi \ll{\nu_\tau ,\left ( \f S \o \f h - \nabla \dd \o \dd \right )\dreidots \f \Theta \dreidots  \nabla \dd \o \left ( \rot\dd- \rot{\f h}\right )^T \rot \dd \t \dd } }\\
&+ \intte{\phi \left ( \nabla \dd \o \dd \dreidotkom \f \Theta \dreidots (\nabla \dd- \f S)  \o  \left ( \rot{\f d} - \rot \dd \right )^T \rot \dd \t \dd  \right )}\,.
\end{align*}
Estimating the right hand side gives
\begin{align*}
&\hspace{-2em}\intte{\phi \ll{\nu_\tau , \f S \o \f h \dreidots \f \Theta \dreidots \f S \o \left ( \rot\dd- \rot{\f h}\right )^T \rot \dd \t \dd }}
%\\
%&\hspace{-2em}
+ \intte{\phi \left ( \nabla \dd \o \dd \dreidotkom \f \Theta \dreidots \nabla \dd \o  \left ( \rot{\f d} - \rot \dd \right )^T \rot \dd \t \dd  \right )}\\
\leq {}& 
\| \dd\|_{L^\infty( \f L^\infty)}   \intte{\phi  \| \t \dd\|_{ \f L^\infty}\ll{\nu_t , \left | \f \Theta \dreidots \left ( \f S \o \f h - \nabla \dd \o \dd  \right )\right |^2 + | \f S - \nabla \dd |^2 | \f h - \dd |^2 }}\\
&+  \| \dd\|_{L^\infty( \f W^{1,3})} \| \dd \|_{L^\infty(\f L^\infty)}  \intte{\phi \| \t \dd\|_{ \f L^\infty}\left ( \ll{\nu_t ,  \left | \f \Theta \dreidots \left ( \f S \o \f h - \nabla \dd \o \dd  \right )\right |^2  }+  \| \f d - \dd \|_{\f L^6}^2\right )}\\
&+\| \dd\|_{L^\infty( \f W^{1,3})}\| \dd\|_{L^\infty( \f L^\infty)}^2    \intte{\phi\| \t \dd\|_{ \f L^\infty}\left ( \ll{\nu_t ,  | \f S - \nabla \dd |^2  }+  \| \f d - \dd \|_{\f L^6}^2\right )}\,,
\end{align*}
which can be bounded due to Lemma~\ref{Sobolev} by the relative Energy.

Putting all the estimates back into~\eqref{abs} and using Lemma~\ref{prop:lemvar} gives
the asserted integration-by-parts formula~\eqref{intd}.
%
%
%

%%%%%%%%%%%%%%%%%%%%%%%%%%%%%%%%%%%%%%
%%%%%%%%%%%%%%%%%%%%%%%%%%%%%%%%%%%%%%%
%%%%%%%%%%%%%%%%%%%%%%%%%%%%%%%%%%%%%%%%%
\end{proof}
%%%%%%%%%%%%%%%%%%%%%%%%%%%%%%%%%%%%%
%%%%%%%%%%%%%%%%%%%%%%%%%%%%%%%%%%%%%
%%%%%%%%%%%%%%%%%%%%%%%%%%%%%%%%%%%%%

\begin{corollary}\label{cor:cor1}
Let $(\f v, \f d)$ be a measure-valued solution and $( \vv , \dd )$ be a strong solution. Then there exists for every $\delta>0$ a possibly large constant $C_\delta$ such that
\begin{align}
\begin{split}
- ( \f v(t) , \vv (t))  & - \left ( \nabla \f d (t) ; \f \Lambda : \nabla \dd(t)\right ) - \left (  \nabla \f d(t)\o \f d(t) \dreidotkom \f \Theta \dreidots \nabla \dd(t) \o \dd(t)  \right ) 
\\
\leq {}& -  \int_0^t\left  [( \f v  , \t \vv  ) + ( \t \f v  , \vv) \right ]\de s - ( \f v(0), \vv(0))-  \int_0^t   \left [ ( \f d \times \t \f d  ,\dd \times  \tq) + ( \f d \times \f q ,\dd \times  \t \dd)  \right ]    \de s \\&- 
\left  ( \nabla \f d(0); \f \Lambda : \nabla \dd(0)\right ) - \left ( \nabla \f d(0) \o \f d(0) \dreidotkom \f \Theta \dreidots \nabla \dd(0) \o \dd(0)\right )\\
 & + \frac{1}{2} \mathcal{E}(t) + \delta \int_0^t \mathcal{W}(s) \de s + C_\delta \int_0^t \mathcal{E}(s) \de s \\
&
+\left ((\nabla \f d(0)-\nabla \dd (0))\o (\f d(0)-  \dd(0))\right ) \dreidots \f \Theta \dreidots (\nabla \dd (0)\o \dd(0))
+ c\| \nabla \dd \o \dd \|_{L^\infty( \f L^\infty)}^2 \left \|   \f d(0)- \dd(0)    \right \|_{\Le}^2 
\end{split}\label{auscor}
\end{align}
holds for a.e.~$t\in (0,T)$.  
\end{corollary}
\begin{proof}

%%%%%%%%%%%%%%%%%%%%%%%%%%%%%%%%%%
%%%%%%%%%%%%%%%%%%%%%%%%%%%%%%%%%%%
%%%%%%%%%%%%%%%%%%%%%%%%%%%%%%%%%%
First, we observe that the integration-by-parts formulae of Lemma~\ref{prop:int} holds true for $s=0$.  
Observe that $\f v \in  L^\infty(0,T; \Ha) \cap W^{1,2}(0,T; (\f W^{1,3}_{0,\sigma})^*) \hookrightarrow \C_w([0,T]; \Ha)$ and $ \vv \in \C([0,T]; \Ha)$ such that the limit $s\ra 0$ in~\eqref{intv} is well defined. 
Since
\begin{align*}
\f d \in L^\infty (0,T; \He) \cap \AC([0,T]; \f L^{3/2})\quad\text{and}\quad \dd \in L^\infty(0,T;\f W^{2,3}) \cap \AC([0,T]; \f W^{1,3}) \,.
\end{align*}
A standard lemma (see Lions and Magenes~\cite[p. 297]{magnes} shows that $\f d \in \C_w([0,T]; \He)$. With the compact Sobolev embedding $\He\hookrightarrow^c \f L^5$ we see $ \f d \in \C( [0,T]; \f L^5)$. Similarly, we see $\f d \in \C([0,T]; \f W^{1,{10}/{3}})$.  Together, we can estimate
\begin{align*}
&\left |( \nabla \f d(0) \o \f d(0) ) \dreidots \f \Theta \dreidots ( \nabla \dd(0) \o \dd(0))-( \nabla \f d(s) \o \f d(s) ) \dreidots \f \Theta \dreidots ( \nabla \dd(s) \o \dd(s))   \right |
\\
& \leq \left | ((\nabla \f d(0)- \nabla \f d(s)) \o \f d(0) ) \dreidots \f \Theta \dreidots ( \nabla \dd(0) \o \dd(0))\right | \\
& \quad + \| \nabla \f d (s)\| _{\Le}^2 \| \f d (0)-\f d (s)\| _{\f L^5}^2 \| \nabla \dd (0)\|_{\f L^{10/3}} \| \dd (0)\|_{\f L^\infty} \\
& \quad  +  \| \nabla \f d (s)\| _{\Le}^2 \| \f d (s)\| _{\f L^6}^2 \| \nabla \dd (0)-\nabla \dd (s) \|_{\f L^{10/3}} \| \dd (0)\|_{\f L^{15/2}}\\
& \quad  +  \| \nabla \f d (s)\| _{\Le}^2 \| \f d (s)\| _{\f L^6}^2 \| \nabla \dd (s) \|_{\f L^{10/3}} \|  \dd (0)-\dd (s)\|_{\f L^{15/2}}
\,,
\end{align*}
and the limit $s\ra 0$  vanishes due to the weak convergence of the gradient of $\f d$ and the strong convergence of the other terms.
The additional terms depending on $s$ in\eqref{intd1},~\eqref{intd2} and~\eqref{intd}  can be handled similar, even somehow simpler since they are more regular.
For the term~\eqref{intdd5} depending on $t$ we observe that we can represent it due to~\eqref{identify} almost everywhere via the generalized gradient Young measure (see Definition~\eqref{def:gradmeas}) and apply Youngs inequality 
\begin{align}
\begin{split}
&\left ((\nabla \f d(t)-\nabla \dd (t))\o (\f d(t)-  \dd(t)) \dreidotkom  \f \Theta \dreidots \nabla \dd (t)\o \dd(t) \right )\\
&= \int_{\Omega} \left \langle  \nu^o _{( \f x ,t)},\left ( ( \f S - \nabla \dd (t) ) \o ( \f d (t) - \dd (t)) \right ) \dreidots \f \Theta \dreidots \left ( \nabla \dd (t) \o \dd (t) \right )  \right \rangle \de \f x  \\
&\leq  \frac{k}{4}\int_{\Omega} \left \langle  \nu^o _{( \f x ,t)},\left | \f S - \nabla \dd (t) \right |^2 \right \rangle \de \f x + c \left \|\nabla \dd\o \dd  \right \|^2_{L^\infty( \f L^\infty)} \int_{\Omega}| \f d (t) - \dd (t)|^2  \de \f x  \\
& \leq \frac{k}{4}\ll{\nu_t,  | \f S - \nabla \dd (t) |^2  }   + c \left \|\nabla \dd\o \dd  \right \|^2_{L^\infty( \f L^\infty)} \left\| \f d (t) - \dd (t) \right \|_{\Le}^2   
\end{split}\label{Youngint}
\end{align}
Remark that the defect measure $m_t$ is a positive measure, $m_t \in \M^+(\ov \Omega)$.
The constant $k$ is chosen small enough, such that 
\begin{align*}
k\ll{  \nu _{t}, | \f S - \nabla \dd (t) |^2   }  \leq \ll{\nu _{t},  ( \f S - \nabla \dd (t)) :\f \Lambda : ( \f S - \nabla \dd (t))  }  \leq 2 \mathcal{E}(t)
\end{align*} 
The second term on the right-hand side can be estimated by formula~\eqref{intd1}.
%%%%%%%%%%%%%%%%%%%%%%%%%%%%%%%%
%%%%%%%%%%%%%%%%%%%%%%%%%%%%%%%%
%%%%%%%%%%%%%%%%%%%%%%%%%%%%%%%%
Adding the integration-by-parts formulae~\eqref{intv},~\eqref{intd1} and~\eqref{intd} for $s=0$ and estimating the term in line~\eqref{intdd5} by~\eqref{Youngint} implies the assertion~\eqref{auscor}. 
Therefore, the definitions~\eqref{qdef} and~\eqref{eq:q} of the variational derivatives are inserted for the strong solution and the measure-valued solution, respectively.

\end{proof}

\section{Proof of the main result\label{sec:main}}
The following lemma implies the main result Theorem~\ref{thm:main} and is proven in this Section.
\begin{proposition}\label{lem:main}
Let $( \f v , \f d)$ and $( \nu, m, \nu^\infty)$ be a suitable measure-valued solution for initial values $( \f v_0 , \f d_0) $ and let $(\vv, \dd)$ be a strong solution for the initial values~$(\vv_0, \dd_0) $ fulfilling the same boundary conditions.
Then there exists a possible large constant $c$ such that
\begin{align}
\frac{1}{2}\mathcal{E} (t) \leq{}& \Big ( \mathcal{E}(0)+\left ((\nabla \f d(0)-\nabla \dd (0))\o (\f d(0)-  \dd(0)) \dreidotkom \f \Theta \dreidots \nabla \dd (0)\o \dd(0)\right )
\\
&+ c\left \|   \f d(0)- \dd(0)    \right \|_{\Le}^2 \Big) e ^{\int_0^t\mathcal{K}(s)\de s } \, ,
\end{align} 
where $\mathcal{K}$ is bounded in $L^1(0,T)$.
\end{proposition}
\begin{rem}
The result of Proposition~\ref{lem:main} implies a continuous dependence of the relative energy~\eqref{relEn} on the difference of the initial values. 
\end{rem}
\begin{proof}
Considering the relative energy, we observe
\begin{align*}
\mathcal{E}(t) ={}& \ll{\mu_t , 1 }+\ll{\nu_t , F } + \frac{1}{2}\| \f v(t)\|_{\Le}^2  + \F( \dd(t), \nabla \dd(t)) + \frac{1}{2}\| \vv(t)\|_{\Le}^2 \\
&- ( \f v(t) , \vv (t))  - \left ( \nabla \f d (t) ;\, \f \Lambda : \nabla \dd(t)\right ) - \left (  \nabla \f d(t)\o \f d(t) \dreidotkom \f \Theta \dreidots \nabla \dd(t) \o \dd(t)  \right )  \,.
\end{align*}
We insert the energy inequality for the measure-valued solution and the energy equality for the strong solution. Adding the integral over the relative dissipation gives
\begin{align*}
\mathcal{E}(t) + \int_0^t\mathcal{W}(s) \de s \leq{}& \F( \f d_0) + \frac{1}{2}\| \f v_0\|_{\Le}^2 + \F( \dd_0) + \frac{1}{2}\| \vv_0\|_{\Le}^2+ \int_0^t \langle \f g , \f v + \vv \rangle \de s \\
&- 2 (\mu_1 + \lambda( \mu_2+ \mu_3 ) )\int_0^t ( \f d \cdot \sy v \f d , \dd\cdot \syv  \dd )   \de s    - 2 \mu_4 \int_0^t ( \sy v ; \syv ) \de s \\
& - 2( \mu_5 + \mu_6 - \lambda( \mu_2+ \mu_3 )) \int _0^t ( \sy v \f d , \syv \dd ) \de s - 2  \int_0^t (\f d \times \f q ,\dd\times  \tq ) \de s\\
& + (  ( \mu_2+ \mu_3 ) - \lambda ) \int_0^t \left [ ( \f d \times \f q ,\f d \times \sy v \f d ) + ( \dd \times \tq , \dd \times \syv \dd )\right ] \de s \\
&- ( \f v(t) , \vv (t))   - \left ( \nabla \f d (t) ; \f \Lambda : \nabla \dd(t)\right ) - \left (  \nabla \f d(t)\o \f d(t) \dreidotkom \f \Theta \dreidots \nabla \dd(t) \o \dd(t)  \right )  \,.
\end{align*}
Using Corollary~\ref{cor:cor1} gives
\begin{align}
\begin{split}
\mathcal{E}(t) + \int_0^t\mathcal{W}(s) \de s \leq{}& \F( \f d_0) + \frac{1}{2}\| \f v_0\|_{\Le}^2 + \F( \dd_0) + \frac{1}{2}\| \vv_0\|_{\Le}^2 \\
&- 2 (\mu_1 +\lambda( \mu_2+ \mu_3 ))\int_0^t ( \f d \cdot \sy v \f d , \dd\cdot \syv  \dd )   \de s   \\
& - 2( \mu_5 + \mu_6 - \lambda ( \mu_2+ \mu_3 )) \int _0^t ( \sy v \f d , \syv \dd ) \de s \\
& - 2 \mu_4 \int_0^t ( \sy v ; \syv ) \de s - 2  \int_0^t ( \f d\times \f q , \dd \times \tq ) \de s  + \int_0^t \langle \f g , \f v + \vv \rangle \de s \\
& + (  ( \mu_2+ \mu_3 ) - \lambda ) \int_0^t \left [ ( \f d \times \f q , \f d \times \sy v \f d ) + ( \dd \times \tq ,\dd \times  \syv \dd )\right ] \de s- ( \f v(0), \vv(0))  \\
&-  \int_0^t\left [( \f v  , \t \vv  ) + ( \t \f v  , \vv) \right ]\de s  - 
\left  ( \nabla \f d(0); \f \Lambda : \nabla \dd(0)\right )- \left ( \nabla \f d(0) \o \f d(0) \dreidotkom \f \Theta \dreidots \nabla \dd(0) \o \dd(0)\right )
\\& 
-  \int_0^t   \left [ ( \f d \times \t \f d  ,\dd \times \tq) + ( \f d \times \f q ,\dd \times  \t \dd)  \right ]    \de s
%\\
%\\
%&\inttes{ \ll{\nu_t,  \f S   :\f \Lambda :\f \Upsilon : \left (\rot{\dd}\t \dd \o \f S ^T\right ) }+ \left ( \nabla \f d ; \f \Lambda : \rot{\f d }^T \nabla ( \rot{\dd} \t \dd )   \right )}\\
%&+ \inttes{\left ( \rot{\dd} - \Lap \dd , \rot{\f d } \t \f d  \right )}\\
%& \inttes{\phi\left ( \rot{\f d} \t \f d , \rot{\dd} \left (- \di \left ( \dd \cdot \f \Theta \dreidots \nabla \dd \o \dd \right )+ \nabla \dd : \f \Theta \dreidots \nabla \dd \o \dd \right ) \right )}
%\\
%&+ \inttes{\phi \ll{\nu_\tau , \f S \o \f h \dreidots \f \Theta \dreidots \f \Upsilon : ( \rot \dd^T  \t \dd \o \f S ) \o \f h }}
%\\&+ \inttes{\phi \left ( \nabla \f d \o \f d \dreidotkom \f \Theta \dreidots \rot{\f d}^T \nabla \left (\rot \dd \t \dd \right ) \o \f d \right )}\\
%&+ \inttes{\phi \ll{\nu_\tau , \f S \o \f h \dreidots \f \Theta \dreidots \f S \o \rot{\f h}^T \rot \dd \t \dd }}
%\\
% &
  + \frac{1}{2} \mathcal{E}(t) + \delta \int_0^t \mathcal{W}(s) \de s + C_\delta \int_0^t \mathcal{E}(s) \de s \\
&
+\left ((\nabla \f d(0)-\nabla \dd (0))\o (\f d(0)-  \dd(0))\right ) \dreidots \f \Theta \dreidots (\nabla \dd (0)\o \dd(0))
%\\
%&
+ c\| \nabla \dd \o \dd \|_{L^\infty( \f L^\infty)}^2 \left \|   \f d(0)- \dd(0)    \right \|_{\Le}^2 
\end{split}
\label{vormeins}
\end{align}
Regarding the terms incorporating the initial values, we see
\begin{align*}
\F( \f d_0) + \frac{1}{2}\| \f v_0\|_{\Le}^2 + \F( \dd_0) + \frac{1}{2}\| \vv_0\|_{\Le}^2- ( \f v(0), \vv(0)) - 
\left  ( \nabla \f d(0); \,\f \Lambda : \nabla \dd(0)\right )&\\ - \left ( \nabla \f d(0) \o \f d(0) \dreidotkom \f \Theta \dreidots \nabla \dd(0) \o \dd(0)\right ) &= \mathcal{E}(0)\, .
\end{align*}
In the next step, we use that $(\f v, \f d)$ and $(\vv, \dd)$  are measure-valued and strong solutions, respectively.
For the measure-valued solution, we consider equation~\eqref{eq:velo} tested with $\vv$ and the equation~\eqref{eq:mdir} tested with $\dd \times \tq$ and added up:%\rand{Kreuzprodukt}
\begin{align*}
-{}&\intt{( \t \f v , \vv) + ( \f d \times \t \f d , \dd \times \tq) + \mu_4 \left ( \sy v ; \syv \right ) + \left ( \f d \times \f q , \dd \times \tq\right )}\\
={}& \intt{\left ( ( \f v \cdot \nabla ) \f v , \vv  \right )+ \mu_1 \left ( \f d \cdot \sy v , \f d \cdot\syv \f d\right ) + ( \mu_5+\mu_6) \left ( \sy v \f d , \syv \f d \right ) }\\
& + \intt{(\mu_2+\mu_3) \left ( \syv \f d , \f e\right ) + \lambda \left ( \sy v \f d , \skv \f d  \right ) + \left (\f e , \skv \f d \right )}\\
& -\intt{\ll{\nu_t , \f S^T F_{\f S}( \f h ,\f S) : \nabla \vv }+ 2 \ll{\mu_t, \f \Gamma \dreidots (\f \Gamma\cdot \nabla \vv) }}\\
&+ \intt{\left ( \rot{\f d }( \f v \cdot \nabla ) \f d , \rot \dd \tq \right ) - \left ( \rot{\f d} \sk v \f d , \rot{\dd} \tq \right )+ \lambda  \left ( \rot{\f d} \sy v \f d , \rot\dd \tq\right )}\,.
\end{align*}
Using Corollary~\ref{prop:e} and equation~\eqref{eq:mdir} to replace the term $\f e$ shows
\begin{align}
\begin{split}
-{}&\intt{( \t \f v , \vv) + ( \f d \times \t \f d , \dd \times \tq) + \mu_4 \left ( \sy v ; \syv \right ) + \left ( \f d \times \f q , \dd \times \tq\right )}\\
={}& \intt{\left ( ( \f v \cdot \nabla ) \f v , \vv  \right )+ \mu_1 \left ( \f d \cdot \sy v , \f d \cdot\syv \f d\right ) + ( \mu_5+\mu_6) \left ( \sy v \f d , \syv \f d \right ) }\\
& - \intt{(\mu_2+\mu_3) \left (\f d \times  \syv \f d ,\f d \times  \f q\right )+ \lambda(\mu_2+\mu_3) \left (\f d \times  \syv \f d ,\f d \times  \sy v\right ) + \left (\f d \times \f q ,\f d \times  \skv \f d \right )}\\
& -\intt{\ll{\nu_t , \f S^T F_{\f S}( \f h ,\f S) : \nabla \vv }+ 2 \ll{\mu_t, \f \Gamma \dreidots( \f \Gamma\cdot \nabla \vv) }}\\
&+ \intt{\left ( \rot{\f d }( \f v \cdot \nabla ) \f d , \rot \dd \tq \right ) - \left ( \rot{\f d} \sk v \f d , \rot{\dd} \tq \right )+ \lambda  \left ( \rot{\f d} \sy v \f d , \rot\dd \tq\right )}\,.
\end{split}\label{measeq}
\end{align}
Remark that $\f d\cdot \skv \f d=0$. 
The strong solution is treated similar, by testing equation~\eqref{nav} with $\f v$ and equation~\eqref{dir} with $\f d \times \f q$ and adding them up:
\begin{align*}
-{}& \intt{\left ( \t \vv ,\f v \right ) + \left ( \dd \times \t \dd , \f d \times \f q\right ) + \mu_4 \left ( \syv, \sy v \right ) + \left ( \dd \times \tq , \f d \times \f q \right ) }\\
={}& \intt{\left ( ( \vv \cdot \nabla ) \vv , \f v  \right )+ \mu_1 \left (\dd \cdot \syv \dd , \dd \cdot \sy v \dd \right ) + ( \mu_5+ \mu_6) \left ( \syv \dd , \sy v \dd \right ) }\\
&+\intt{( \mu_2+ \mu_3 ) \left (\te , \sy v \dd \right )+ \lambda \left (\syv \dd , \sk v \dd \right )+ \left (\te , \sk v \dd\right )}\\
&- \intt{ \left (\rot \dd\left ( ( \vv \cdot \nabla) \dd -\skv \dd +\lambda \syv \dd\right ), \rot \dd  \f q \right )+ \left (\nabla \dd ^T F_{\f S}( \dd , \nabla \dd) ; \nabla \f v \right )}\,.
\end{align*}
Using Corollary~\ref{prop:e} and equation~\eqref{dir} to replace the term $\te$ shows
\begin{align}
\begin{split}
-{}& \intt{\left ( \t \vv ,\f v \right ) + \left ( \dd \times \t \dd , \f d \times \f q\right ) + \mu_4 \left ( \syv, \sy v \right ) + \left ( \dd \times \tq , \f d \times \f q \right ) }\\
={}& \intt{\left ( ( \vv \cdot \nabla ) \vv , \f v  \right )+ \mu_1 \left (\dd \cdot \syv \dd , \dd \cdot \sy v \dd \right ) + ( \mu_5+ \mu_6) \left ( \syv \dd , \sy v \dd \right ) }\\
&-\intt{( \mu_2+ \mu_3 ) \left (\dd\times \tq ,\dd \times  \sy v \dd \right )+\lambda( \mu_2+ \mu_3 ) \left (\dd\times \syv \dd  ,\dd \times  \sy v \dd \right )+ \left (\dd\times\tq , \dd\times\sk v \dd\right )}\\
&- \intt{ \left (\rot \dd\left ( ( \vv \cdot \nabla) \dd -\skv \dd +\lambda \syv \dd\right ), \rot \dd  \f q \right )+ \left (\nabla \dd ^T F_{\f S}( \dd , \nabla \dd) ; \nabla \f v \right )}\,.
\end{split}\label{strongeq}
\end{align}
Remark that $\dd\cdot \sk v \dd =0$.
Inserting the equations~\eqref{measeq} and~\eqref{strongeq} in~\eqref{vormeins} gives 
\begin{align*}
\begin{split}
\frac{1}{2}\mathcal{E}(t)& + \int_0^t\mathcal{W}(s) \de s - \mathcal{E}(0)  
-\left ((\nabla \f d(0)-\nabla \dd (0))\o (\f d(0)-  \dd(0))\right ) \dreidots \f \Theta \dreidots (\nabla \dd (0)\o \dd(0))
\\
&- c\| \nabla \dd \o \dd \|_{L^\infty( \f L^\infty)}^2 \left \|   \f d(0)- \dd(0)    \right \|_{\Le}^2 
\\
\leq{}& \delta\int_0^t\mathcal{W}(s) \de s + C_{\delta} \int_0^t \mathcal{E}(s)\de s + \int_0^t \left [( ( \vv \cdot \nabla ) \vv , \f v ) + ( ( \f v \cdot\nabla ) \f v , \vv ) \right ]\de s \\
& + \mu_1 \int_0^t  \left [  \left ( \dd \cdot \syv \dd , \sy v : \left ( \dd \o \dd - \f d \o \f d   \right )   \right ) +  \left ( \f d  \cdot \sy v \f d  , \syv : \left ( \f d  \o \f d  - \dd \o \dd   \right )   \right ) \right ] \de s \\
%&
% +  \mu_1 \int_0^t\left ( \f d  \cdot \sy v \f d  , \syv : \left ( \f d  \o \f d  - \dd \o \dd   \right )   \right ) \de s \\
& + ( \mu_5 + \mu_6  ) \int_0^t  \left [\left ( \syv \dd , \sy v ( \dd - \f d )\right  )  + \left (    \sy v \f d , \syv ( \f d - \dd ) \right ) \right ]   \de t\\
%&+ \inttet{\left  ( \f e , \skv \f d \right  )+\lambda  \left ( \sy v \f d , \skv\f d \right ) + (\te , \sk v \dd )+\lambda (\syv \dd , \sk v \dd  )}
%\\
& +\inttet{\left [ \left (\dd \times \skv \dd- \f d \times \skv \f  d , \f d \times \f q \right )+ \left (\f d \times \sk v \f d- \dd \times \sk v \dd  , \dd \times \tq \right )  \right ]}\\
%&
%+\inttet{ \left (\f d \times \sk v \f d- \dd \times \sk v \dd  , \dd \times \tq \right )}\\
& -  ( \mu_2 + \mu_3 ) \int_0^t \left [\left ( \dd\times( \sy v - \syv)\dd , \dd \times \tq \right )+ \left ( \f d \times (\syv - \sy v)\f d  , \f d \times \f q \right )\right ]  \de s
 \\
% &-  ( \mu_2 + \mu_3 ) \int_0^t \left ( \f d \times (\syv - \sy v)\f d  , \f d \times \f q \right )\de s \\
& + \lambda \int_0^t  \left [ \left (\f d \times  \sy v \f d -\dd \times \sy v \dd , \dd\times \tilde{ \f q}  \right )+ \left ( \dd \times \syv \dd - \f d \times \syv \f d , \f d \times \f q \right ) \right ]
\de s \\
%&+ \lambda \int_0^t \left ( \dd \times \syv \dd - \f d \times \syv \f d , \f d \times \f q \right ) \de s \\
& - 2\lambda( \mu_2+ \mu_3 )  \inttett{\left ( \dd \cdot \syv \dd , \f d \cdot\sy v \f d \right )- \left (\syv \dd ,\sy v \f d \right )}\\
& - \lambda( \mu_2+ \mu_3 )  \inttett{\left ( \f d \times \sy v \f d , \f d\times  \syv \f d\right ) + \left (\dd \times \sy v \dd , \dd \times \syv \dd\right )}\\
&+ \int_0^t \left [(\dd \times ( \vv \cdot \nabla ) \dd ,\f d \times \f q) + (\f d \times( \f v \cdot \nabla )\f d ,\dd \times  \tilde{\f q})\right ] \de s \\
&-\int_0^t\left [ ( \nabla \dd^T F_{\f S}(\dd,\nabla \dd) ; \nabla \f v )  + \ll{ \nu_s ,  \f S ^T F_{\f S}(\f h , \f S) : \nabla \vv    }\right ]\de s  -2\int_0^t \ll{\mu_s, \f \Gamma \dreidots   (\f \Gamma \cdot \nabla \vv) } \de s 
\\
={}&  \delta\int_0^t\mathcal{W}(s) \de s + C_{\delta} \int_0^t \mathcal{E}(s)\de s +I_1  + \mu_1 I_2 + ( \mu_5 +\mu_6 ) I_3 + I_4 -(\mu_2+\mu_3)I_5 +\lambda I_6\\ & \qquad\qquad\qquad\qquad\qquad\qquad-2\lambda(\mu_2+\mu_3) I_7- \lambda(\mu_2+\mu_3) I_8 + I_9 - 2I_{10}\, . 
\end{split}
\end{align*}

The terms are estimated individually by integrals over $\mathcal{E}$ and $\mathcal{W}$. 
We start with $I_1$. After rearranging one can estimate
\begin{align*}
I_1 &= \int_0^t{\left [( ( \f v \cdot \nabla ( \f v - \vv ) , \vv - \f v ) + ( ( ( \f v - \vv )\cdot \nabla ) \vv , \vv- \f v ) \right ]}\de s\\&
 = \int_0^t{\left ((\f v - \vv )\otimes ( \vv- \f v ) ; \syv\right )  }\de s \\ & 
  \leq C_\delta  \inttet{\|\syv\|_{\f L^{3}}^2\|\f v - \vv \|_{\Le}^2} + \delta \inttet{\| \f v - \vv\|_{\f L^6}^2}\, .
\end{align*}

The term $I_2$ can be rearranged as
%We go back to equation~\eqref{anfangen} and consider the term $I_2$. First, rearranging the term gives
\begin{align*}
I_2  ={}&\intte{\left ( \dd \cdot \syv \dd , \sy v : \left ( \dd \otimes \dd - \f d \otimes\f d\right )\right )}
% \\& 
   +\inttet{\left (\f d \cdot \sy v \f d , \syv: \left ( \f d \otimes \f d - \dd \otimes \dd \right ) \right )}\\
 ={}& \inttet{ \left (\dd \cdot \syv \dd , \left (\sy v -\syv\right ): \left ( \dd \otimes \dd - \f d \otimes\f d\right )\right )}
 \\&  
  +\inttet{\left (\f d \cdot \sy v \f d - \dd \cdot \syv \dd , \syv: \left ( \f d \otimes \f d - \dd \otimes \dd \right )\right ) }\\
  ={}& \inttet{ \left (\dd \cdot \syv \dd , \left (\sy v -\syv\right ): \left ( \dd \otimes (\dd - \f d) \right )\right )}\\
& + \inttet{ \left (\dd \cdot \syv \dd , \left (\sy v\f d -\syv\dd\right )\cdot  \left ( \dd - \f d \right )\right )}
%\\
%&  
+  \inttet{\left (\dd \cdot \syv \dd, \syv: (\dd - \f d ) \otimes ( \dd - \f d )  \right )}
 \\&   +2\inttet{\left (\f d \cdot \sy v \f d - \dd \cdot \syv \dd , \syv:  (\f d - \dd )\otimes \dd \right ) }
  \\&   +\inttet{\left (\f d \cdot \sy v \f d - \dd \cdot \syv \dd , \syv:  (\f d- \dd) \otimes (\f d - \dd)  \right )}\,.
\end{align*}
It is thus possible to estimate $I_2$ by
\begin{align*}
I_2\leq{}&  \delta\inttet{\| \sy v- \syv\|_{\Le}^2}+  C_\delta  \| \dd \|_{L^\infty(\f L^\infty)}^2 \inttet{\|\dd \cdot \syv \dd\|_{L^3}^2\|\f d - \dd\|_{\f L^6}^2} \\
 &  +\delta\inttet{\| \sy v\f d- \syv\dd\|_{\Le}^2}+ C_\delta  \inttet{\|\dd \cdot \syv \dd\|_{ L^3}^2\|\f d - \dd\|_{\f L^6 }^2} \\ 
 &  +   \inttet{ \|\dd \cdot \syv \dd\|_{ L^3} \| \syv \|_{\f L^3}\|\f d - \dd \|_{\f L^6}^2} \\ 
 &  + \delta\inttet{\| \f d \cdot \sy v\f d- \dd \cdot \syv\dd\|_{L^2}^2}+ C_\delta  \| \dd \|_{L^\infty(\f L^\infty)}^2 \inttet{\| \syv \|_{\f L^3}^2\| \f d - \dd\|_{\f L^6}^2} \\
&  + \delta\inttet{\| \f d \cdot \sy v\f d- \dd \cdot \syv\dd\|_{L^2}^2}+ C_\delta   \inttet{\| \syv \|^2_{\f L^3}\| \f d -\dd \|_{\f L^{12}}^4 }  \\
\leq{}& \delta \inttet{\mathcal{W}(s)} +  \inttet{\mathcal{K}(s)\mathcal{E}(s)}
\end{align*}
Regarding the term $I_3$ we observe
\begin{align*}
I_3={}&  \inttett{\left ( \syv \dd, \sy v( \dd -\f d)\right ) + \left ( \sy v \f d , \syv ( \f d - \dd)\right ) }\\
 ={}&   \inttett{\left ( \left ( \syv \dd, \left (\sy v- \syv\right )( \dd -\f d)\right )+\left ( \sy v \f d- \syv \dd  , \syv ( \f d - \dd)\right ) \right )  }\\ 
%&  + \inttet{\left ( \sy v \f d- \syv \dd  , \syv ( \f d - \dd)\right ) }\\
\leq{}& \delta\inttet{\| \sy v- \syv\|_{\Le}^2} +C_\delta   \inttet{\| \syv\dd  \|^2_{\f L^3}\|\f d - \dd \|_{\f L^6}^2} \\
&  +\delta\inttet{\| \sy v\f d- \syv\dd \|_{\Le}^2} +C_\delta   \inttet{\| \syv  \|^2_{\f L^3}\|\f d - \dd \|_{\f L^6}^2} \\ 
\leq{}& \delta \inttet{\mathcal{W}(s)} +  \inttet{\mathcal{K}(s)\mathcal{E}(s)}
\end{align*}
%%%%%%%%%%%%%%%%%%%%%%%%%%%%%%%%%%%%%%%
%%%%%%%%%%%%%%%%%%%%%%%%%%%%%%%%%%%%%%%%%%
%%%%%%%%%%%%%%%%%%%%%%%%%%%%%%%%%%%%%%
Similar, the term $I_4$ can be rearranged
\begin{align*}
I_4
={}&   \inttett{\left  (\f d \times\f q , (\dd -\f d) \times \skv\f d \right ) + (\dd \times \tq, (\f d- \dd) \times \sk v \dd  )}
\\& +\inttett{ \left (\dd \times \skv (\dd- \f d ) , \f d \times \f q \right )+ \left (\f d \times \sk v( \f d- \dd) , \dd \times \tq \right )}
\\
={}&  
% \inttet{\left  (\f d \times\f q- \dd \times \tq  , (\dd -\f d) \times \skv\f d \right )
%}
%\\
%&+
\inttett{  \left  (\f d \times\f q- \dd \times \tq  , (\dd -\f d) \times \skv\f d \right ) + \left (\dd \times \tq, (\dd- \f d) \times \left (\skv (\f d -  \dd )+ \left ( \sk v - \skv \right ) \dd  \right ) \right )}
%\\& +\inttet{  \left ((\f d- \dd) \times \skv( \f d- \dd) , \dd \times \tq \right )+ \left ((\f d-\dd) \times \left (\sk v- \skv\right )( \f d- \dd) , \dd \times \tq \right )}
\\
& +\inttett{ \left (\dd \times \skv (\dd- \f d ) , \f d \times \f q -  \dd \times \tq  \right )+ \left (\left (( \dd- \f d) \times \skv + \f d \times (\skv - \sk v)\right )( \dd- \f d) , \dd \times \tq \right ) }
\\
={}&   \inttett{\left  (\f d \times\f q- \dd \times \tq  , (\dd -\f d) \times \skv(\f d- \dd) \right )+ \left  (\f d \times\f q- \dd \times \tq  , (\dd -\f d) \times \skv\dd \right )  
}
\\
&+\inttett{ (\dd \times \tq, (\f d- \dd) \times \left (\sk v - \skv\right ) \dd )+\left (\dd \times \skv (\dd- \f d ) , \f d \times \f q- \dd \times \tq \right )}
\\& +\inttett{  2\left ((\f d- \dd) \times \skv( \f d- \dd) , \dd \times \tq \right )+ \left ((\f d-\dd) \times \left (\sk v- \skv\right )( \f d- \dd) , \dd \times \tq \right ) }\\
& + \inttet{\left (\dd \times \left (\sk v- \skv\right )( \f d- \dd) , \dd \times \tq \right )}
\end{align*}
This can now be estimated by
\begin{align*}
I_4
%+I_5
\leq{}& \delta\inttet{\left \|\f d \times\f q- \dd \times \tq\right \|_{\Le}^2} + C_{\delta}  \inttet{\|\skv \|_{\f L^3}^2\left (\|\f d - \dd\|_{\f L^{12}}^4+\|\dd\|_{L^\infty(\f L^\infty)}^2\|\f d - \dd\|_{\f L^{6}}^2\right )}\\
& +
\delta \inttet{\| \sk v -\skv \|_{\Le}^2} + C_{\delta} \|\dd\|_{L^\infty(\f L^\infty)}^2 \inttet{\|\tq \|_{\f L^3}^2\left (\|\f d - \dd\|_{\f L^{12}}^4+\|\dd\|_{L^\infty(\f L^\infty)}^2\|\f d - \dd\|_{\f L^{6}}^2\right )}\\
& 
%+ 
%C_\delta  \|\dd\|_{L^\infty(\f L^\infty)}^2  \inttet{\|\tq\|_{\f L^3}^2\|\f d - \dd\|_{\f L^{12}}^4} + 
%C_\delta  \|\dd\|_{L^\infty(\f L^\infty)}^4  \inttet{\|\tq\|_{\f L^3}^2\|\f d - \dd\|_{\f L^{6}}^2}
+ 2   \|\dd\|_{L^\infty(\f L^\infty)}\inttet{\|\skv \|_{\f L^3}  \|\tq\|_{\f L^3}\|\f d - \dd\|_{\f L^{6}}^2}\\
\leq{}& \delta \inttet{\mathcal{W}} +  \inttet{\mathcal{K}(s)\mathcal{E}}\,.
\end{align*}

%%%%%%%%%%%%%%%%%%%%%%%%%%%%%%%%%%%%%
%%%%%%%%%%%%%%%%%%%%%%%%%%%%%%%%%%%%%%
%%%%%%%%%%%%%%%%%%%%%%%%%%%%%%%%%%%%%%

%
The term $I_5$ can be rearranged as
\begin{align}\label{I5}
\begin{split}
I_5 ={}&
 \int_0^t \left [\left ( \dd\times( \sy v - \syv)\dd , \dd \times \tq \right )+ \left ( \f d \times (\syv - \sy v)\f d  , \f d \times \f q \right )\right ]  \de s
 \\
={}&
\inttet{\left (\left (  \dd \times \sy v \dd - \f d \times \sy v \f d , \dd \times \tq  \right )+\left (\f d \times \syv \f d - \dd \syv \dd , \f d \times \f q \right )\right )}
\\
& - \inttet{\left (\f d \times \sy v \f d- \dd \times \syv \dd   , \f d \times \f q - \dd \times \tq  \right )}
\, .
\end{split}
\end{align}
Using the  simple reformulation
\begin{align*}
\f d \times \sy v \f d- \dd \times \sy v \dd= (\f d-\dd) \times \sy v (\f d-\dd)+  (\f d -\dd ) \times \sy v \dd + \dd \times \sy v ( \f d -\dd)\,,
\end{align*}
 the first line of the right-hand side of~\eqref{I5} can be rearranged to
\begin{align*}
\int_0^t&\left (\left (  \dd \times \sy v \dd - \f d \times \sy v \f d , \dd \times \tq  \right )+\left (\f d \times \syv \f d - \dd \syv \dd , \f d \times \f q \right )\right )\de s \\
= {}& \inttet{\left ((\f d-\dd) \times (\sy v- \syv) (\f d-\dd)+  (\f d -\dd ) \times (\sy v- \syv) \dd  , \dd \times \tq \right ) }
\\
& +\inttet{\left (  \dd \times (\sy v - \syv)( \f d -\dd) , \dd \times \tq \right ) }
\\
& +\inttet{\left ((\f d-\dd) \times \syv (\f d-\dd)+  (\f d -\dd ) \times \syv \dd + \dd \times \sy v ( \f d -\dd), \f d \times \f q - \dd \times \tq    \right )}
\end{align*}
and estimated by
\begin{align*}
&\inttett{\left (  \dd \times \sy v \dd - \f d \times \sy v \f d , \dd \times \tq  \right )+\left (\f d \times \syv \f d - \dd \syv \dd , \f d \times \f q \right )}\\
\leq{}& 
\delta \inttet{\| \sy v - \syv \|_{\Le}^2 }+ \delta \inttet{\| \f d \times \f q - \dd \times \tq\|_{\Le}^2}\\
&+ C_\delta \| \dd\|_{L^\infty(\f L^\infty)}^2 \inttet{ \left ( \|  \tq\|_{\f L^3}^2+ \| \syv \|_{\f L^{3}}^2\right ) \left  (\| \f d -\dd \|_{\f L^{12}}^4 +\| \dd\|_{L^\infty(\f L^\infty)}^2 \| \f d -\dd \|_{\f L^6}^2\right )}\,.
\end{align*}
For $I_5$ and $I_6$ together we have
\begin{align*}
-(\mu_2+\mu_3)I_5 + \lambda I_6 \leq{}& ((\mu_2+\mu_3)-\lambda)  \inttet{\left ( \dd \times \syv \dd -\f d \times \sy v \f d ,\dd\times \tq -\f d \times \f q  \right )}\\& + \delta \inttet{\mathcal{W}(s)} + \inttet{\mathcal{K}(s)\mathcal{E}(s)}\, .
\end{align*}
%%%%%%%%%%%%%%%%%%%%%%%%%%%%%%%%%
%%%%%%%%%%%%%%%%%%%%%%%%%%%%%%%%%
%%%%%%%%%%%%%%%%%%%%%%%%%%%%%%%%%%

For the term $I_8$ we observe (see Section~\ref{sec:not})
\begin{align*}
 &\inttett{\left ( \f d \times \sy v \f d , \f d\times  \syv \f d\right ) + \left (\dd \times \sy v \dd , \dd \times \syv \dd\right )}
 \\
 &=  \inttett{\left ( \sy v \f d ,  \syv \f d\right )-\left ( \f d \cdot \sy v \f d , \f d\cdot  \syv \f d\right )}\\
 &\quad+  \inttett{ \left ( \sy v \dd ;  \syv \dd\right )- \left (\dd\cdot \sy v \dd , \dd \cdot \syv \dd\right )}
\end{align*}
Comparing the terms $I_7$ and $I_8$ with $I_2$ and $I_3$, we see
\begin{align*}
-2 I_7 -  I_8 = ( I_2 - I_3)\,.
\end{align*}
The appearing terms $I_7$ and $I_8$ can, hence, be estimated as in the cases of $I_2$ and $I_3$.

The definition of the variational derivative~$\f q$~\eqref{qdef} and the derivative $\nicefrac{\partial F}{\partial \nabla \f d}$ of the free energy potential inserted in the term $I_9$ gives
\begin{align*}
I_9  ={}& \int_0^t\left [
\ll{\nu_s , \f S : \f \Lambda : \left (\f \Upsilon: \left ( \rot \dd ( \vv \cdot \nabla )\dd  \otimes \f S  \right )\right )}+ \left (  \left ( \rot{\f d}^T  \nabla  (\rot\dd^T( \vv \cdot \nabla) \dd )\right ); \,\f \Lambda : \nabla \f d  \right )\right ]\de s 
\\
&  + \int_0^t\left [(\f d \times( \f v \cdot \nabla ) \f d , -\dd \times  \Lap \dd ) - \left ( \nabla \f v ; \nabla \dd^T \f \Lambda : \nabla \dd \right ) -\ll{\nu_s, (\f S \nabla \vv ): \f \Lambda : \f S }\right ]\de s 
\\
%- \int_0^t\int_{\Omega} \left  \langle   \f S^T \f \Lambda : \f S : \nabla \vv , \nu    \right \rangle \de \f x + \int_{\ov \Omega} \left  \langle    \f S^T \f \Lambda : \f S : \nabla \vv , \nu ^\infty     \right \rangle  m( \de \f x)\de s  \\
& + \int_0^t
\left [\left (\rot{\f d}^T \nabla (\rot{\dd} ( \vv \cdot \nabla ) \dd ) \o \f d \dreidotkom \f \Theta \dreidots \nabla \f d \otimes \f d \right ) 
+  \left ( \rot{\f d} ( \f v \cdot \nabla ) \f d , -\rot\dd  \di \left ( \dd \cdot \f \Theta \dreidots \nabla \dd \otimes \dd \right )\right ) \right ]
\de s \\&+\int_0^t
\ll{\nu_s, \f \Upsilon : \left ( \rot{\f d }(\vv \cdot \nabla)\dd \o \f S\right )\o \f h \dreidots \f \Theta \dreidots \f S \o \f h}
\de s 
\\
& + \inttett{\left (\nabla \dd \o \left ( \rot{\dd}^T \rot{\f d} ( \f v \cdot \nabla ) \f d \right )\dreidotkom  \f \Theta \dreidots \nabla \dd \o \dd      \right )  +
 \ll{\nu_s, \f S \otimes \left ( \rot{\f h}^T \rot{\dd}( \vv \cdot \nabla ) \dd \right ) \dreidots \f \Theta \dreidots \f S \o \f h }}\\& - 
\inttet{\ll{\nu_s, (\f S\nabla\vv) \o \f h \dreidots  \f \Theta \dreidots \f S \o \f h }}
%&
%+\int_0^t \int_{\Omega } \left \langle  \f S \otimes ( \vv \cdot \nabla ) \dd  \dreidots \f \Theta \dreidots \f S \o \f d , \nu     \right \rangle \de \f x 
%+ \int_{\ov \Omega } \left \langle  \f S \otimes ( \vv \cdot \nabla ) \dd  \dreidots \f \Theta \dreidots \f S \o \f h , \nu^\infty      \right \rangle \de m( \de \f x) 
% \de s 
% \\
% &- \int_0^t
%   \int_{\Omega } \left \langle  \f S^T \left ( \f d \cdot  \f \Theta \dreidots \f S \o \f d\right ): \nabla \vv  , \nu     \right \rangle \de \f x 
%+ \int_{\ov \Omega } \left \langle  \f S^T \left ( \f h \cdot  \f \Theta \dreidots \f S \o \f h\right ): \nabla \vv , \nu^\infty      \right \rangle \de m( \de \f x) 
% \de s \\
  -\int_0^t \left ( \nabla \f v ; \nabla \dd ^T \left ( \dd \cdot \f \Theta \dreidots \nabla \dd \o \dd \right ) \right ) \de s \\
% & + \int_0^t\left ( ( \vv \cdot \nabla ) \dd , \frac{1}{\varepsilon}( | \f d|^2 -1)\f d  \right )\de s   + \int_0^t\left ( ( \f v \cdot \nabla ) \f d , \frac{ 1}{\varepsilon } ( | \dd|^2 -1) \dd \right )\de s \\
 ={}& J_1+J_2\,.
\end{align*}
All terms involving the tensor $\f \Lambda$, i.e. the first two lines, are denoted by the abbreviation $J_1$. The terms involving the tensor $\f \Theta$ are denoted by the abbreviation $J_2$. 
In the following, a vector identity is of need. Therefore, we consider a sequence $\{ \f d _n\}\subset L^2(0,T; \Hc)$ approximating the function $\f d$ such that $ \f d_n  \rightharpoonup \f d$ in $L^2(0,T; \He)$. The sequence $\{ \f d _n\}$ can be chosen in accordance to Definition~\ref{def:gradmeas}.
Since $\vv$ is solenoidal, we see
\begin{align}
\begin{split}
0={}&
\left ( \di \vv , \left ( \nabla \dd: \f \Lambda : \nabla \fn d \right ) \right )  
= - \left ( (\vv\cdot \nabla ) \nabla \dd ; \f \Lambda \nabla \f d    \right )-  \left ( (\vv\cdot \nabla ) \nabla \f d  ; \f \Lambda \nabla \dd    \right ) \\
={}& \left ( \nabla \vv ; \nabla \dd ^T \f \Lambda : \nabla \fn d \right ) + \left ( \nabla \vv; \nabla \fn d ^T \f \Lambda : \nabla \dd \right )  \\
&- \left ( (\vv \cdot \nabla ) \fn d , - \Lap \dd \right ) - \left  (\f \Lambda : \nabla ( ( \vv \cdot \nabla ) \dd ); \nabla \fn d \right )\\
={}& \left ( \nabla \vv ; \nabla \dd ^T \f \Lambda : \nabla \fn d \right ) + \left ( \nabla \vv; (\nabla \fn d- \nabla \dd) ^T \f \Lambda : \nabla \dd \right )   + \left ( \nabla \vv;  \nabla \dd ^T \f \Lambda : \nabla \dd \right )\\
&- \left ( (\vv \cdot \nabla ) \fn d , - \Lap \dd \right ) - \left  (\f \Lambda : \nabla ( ( \vv \cdot \nabla ) \dd ); \nabla \fn d \right )\,.
\end{split}
\label{ZERO1}
\end{align}
The integration in time is omitted for brevity.
Going to the limit in $n$ and using Lemma~\ref{lem:vschlange} shows
\begin{align*}
0={}&\left ( \nabla \vv ; \nabla \dd ^T \f \Lambda : \nabla \f d \right ) + \left ( \nabla \vv; (\nabla \f d- \nabla \dd) ^T \f \Lambda : \nabla \dd \right )   + \left ( \nabla \vv;  \nabla \dd ^T \f \Lambda : \nabla \dd \right )\\
&- \left ( \f d \times (\vv \cdot \nabla ) \f d , - \f d\times \Lap \dd \right ) - \left  ( \nabla ( \rot{\dd}^T\rot{\dd}( \vv \cdot \nabla ) \dd );\f \Lambda : \nabla \f d \right )\,.
\end{align*}
Adding this zero to $J_1$, we get
\begin{align*}
J_1 +0 ={}& 
-\int_0^t \ll{\nu_s,  \left ((\f S-\nabla \dd )^T \f \Lambda : (\f S -\nabla \dd )\right ): \nabla \vv } \de s-  \int_0^t \left ( \nabla \f v - \nabla \vv ; \nabla \dd^T \f \Lambda : \nabla \dd \right ) \de s 
\\&+ \inttett{ 
\left (\f d \times ( \f v \cdot\nabla  ) \f d , -\dd \times \Lap \dd \right )- \left ( \f d \times (\vv \cdot \nabla ) \f d , - \f d\times \Lap \dd \right ) 
 }
\\
&+ \inttett{ 
 \ll{\nu_t, \f S : \f \Lambda : \left ( \f \Upsilon \left ( \rot{\dd} ( \vv \cdot \nabla)\dd \otimes (\f S- \nabla \dd) \right )\right )} 
+  \left ( \rot{\f d-\dd}^T  \nabla  (\rot\dd^T( \vv \cdot \nabla) \dd ); \,\f \Lambda : \nabla \f d \right )
 }
%\\
%&- \inttet{ \left (
%  \left ( \f \Upsilon \left ( \rot{\dd} ( \vv \cdot \nabla)\dd \otimes \nabla \dd  \right ); \f \Lambda : \nabla \f d \right )
%+  \left (  \left ( \rot{\dd}^T  \nabla  (\rot\dd^T( \vv \cdot \nabla) \dd )\right ); \f \Lambda : \nabla \f d  \right )
%\right ) }
\,.
\end{align*}
Since $\f v $ and $\vv $ are solenoidal, an integration-by-parts and Lemma~\ref{lem:vschlange} gives
\begin{align*}
\left ( \nabla \f v - \nabla \vv ; \nabla \dd^T \f \Lambda : \nabla \dd \right ) ={}& - \left (\f v-\vv , \nabla \dd ^T \Lap \dd \right )-\frac{1}{2}\left (\f v-\vv , \nabla \left (\nabla \dd : \f \Lambda : \nabla \dd\right ) \right )\\ ={}&  -\left (\rot{\dd}^T \rot{\dd} \left ((\f v-\vv)\cdot \nabla\right ) \dd , - \Lap \dd \right ) \, .
\end{align*}
After some additional rearrangements, we see
\begin{align*}
J_1+0 ={}& -\int_0^t \ll{\nu_s,  (\f S-\nabla \dd )^T \f \Lambda : (\f S -\nabla \dd ): \nabla \vv } \de s\\&+ \inttet{\left ( ( \f d - \dd ) \times \nabla \dd ( \f v - \vv)  ,- \dd \times\Lap
 \dd   \right )}\\
 &+ \inttet{\left (( \f d - \dd) \times \left ( (\nabla \f d - \nabla \dd) ( \f v - \vv)\right ),- \dd \times \Lap \dd  \right )}\\&+ \inttet{\left ( \dd \times \left ( (\nabla \f d - \nabla \dd) ( \f v - \vv)\right ),- \dd \times \Lap \dd  \right )}\\
 &+ \inttet{
 \ll{\nu_s, ( \f S- \nabla \dd) : \f \Lambda : \left (\f \Upsilon :\left ( \rot\dd ( \vv\cdot\nabla)\dd\o ( \f S - \nabla \dd)\right )\right )}} \\& 
 +\inttet{  \left (   \rot{\f d- \dd }^T  \nabla  (\rot\dd^T( \vv \cdot \nabla) \dd );\, \f \Lambda :( \nabla \f d- \nabla \dd)   \right )
}\\
 &+ \inttet{ 
\left ((\f d -\dd)\times ( \nabla \f d - \nabla \dd)  \vv , (\f d -\dd )\times \Lap \dd \right )}
\\&+ \inttett{ 
\left (\dd\times ( \nabla \f d - \nabla \dd)  \vv , (\f d -\dd )\times \Lap \dd \right )
- \left ( (\f d- \dd)  \times (\vv \cdot \nabla )  \dd , (\dd- \f d)
\times \Lap \dd \right ) 
 }
\end{align*}
Thus, the term $J_1$ can be estimated by
\begin{align*}
J_1 \leq {}&
c \int_0^t\| \nabla \vv \|_{\f L^\infty}\ll{\nu_s,   | \f S - \nabla \dd|^2 } \de s+ \delta \inttet{\| \f v - \vv \|_{\f L^6}^2}\\
&  + C_\delta \| \dd\|_{L^\infty(\f L^\infty)}^2 \inttet{\| \dd\|_{\f W^{2,3}}^2 \left ( \| \dd\|_{L^\infty(\f W^{1,3})}^2\| \f d -\dd\|_{\f L^6}^2 + \ll{\nu_s, | \f S - \nabla \dd|^2 | \f h - \dd|^2}  \right )}
\\
&+ C_\delta \| \dd\|_{L^\infty(\f L^\infty)}^4  \inttet{\| \dd\|_{\f W^{2,3}}^2\ll{\nu_s, | \f S - \nabla \dd|^2 }} 
%\left (\inttet{\| \f d -\dd\|_{\f L^{12}}^4} + \inttet{\ll{\nu_s, | \f S - \nabla \dd|^2 }}  \right )
\\
&+c \| \dd\|_{L^\infty(\f L^\infty)} \inttet{\| \vv\|_{\f L^\infty}\| \dd\|_{\f W^{1,\infty}}\ll{\nu_s, | \f S- \nabla \dd|^2}}\\
&+c \inttet{\| \vv\|_{\f L^\infty}\left ( \| \dd\|_{L^\infty(\f L^\infty)}\| \dd\|_{\f W^{2,3}} + \| \dd\|_{\f W^{1,6}}^2\right ) \left (\ll{\nu_s, | \f S - \nabla \dd|^2 }+ \|\f d -\dd\|_{\f L^6}^2\right )}\\
&+c \inttet{\| \vv\|_{\f W^{1,3}} \| \dd\|_{L^\infty(\f L^\infty)}\| \dd\|_{\f W^{1,\infty}} \left (\ll{\nu_s, | \f S - \nabla \dd|^2 }+ \|\f d -\dd\|_{\f L^6}^2\right )}\\
&+ c \inttet{\| \vv\|_{\f L^\infty} \| \dd\|_{\f W^{2,3}} \left (\| \f d - \dd \|_{\f L^{12}}^4  + \ll{\nu_s , | \f S- \nabla \dd|^2}\right )}
\\
&+ c  \inttet{\| \vv\|_{\f L^\infty}  \| \dd\|_{\f W^{2,3}}  \| \dd\|_{L^\infty(\f L^\infty)}\left (\ll{\nu_s, | \f S- \nabla \dd|^2 }+ \| \f d -\dd\|_{\f L^6}^2\right ) }\\
&+ c  \inttet{\| \vv\|_{\f L^\infty}  \| \dd\|_{\f W^{2,3}}  \| \dd\|_{L^\infty(\f W^{1,3})} \| \f d- \dd \|_{\f L^{6}}^2 }
%
%%&+ C_\delta \| \Lap \dd \|_{L^\infty ( \f L^3)}  \int_0^t\left (\int_{\Omega} \left  \langle   (\f S-\nabla \dd ): \f \Lambda : (\f S -\nabla \dd ) , \nu    \right \rangle \de \f x + \int_{\ov \Omega} \left  \langle    \f S: \f \Lambda : \f S  , \nu ^\infty     \right \rangle  m( \de \f x) \right )\de s \\
%%&+ \delta \int_0^t \| \f v - \vv \|_{\f L^6}^2 \de s 
%\\
%\leq{}& \| \nabla \vv \|_{L^\infty(\f L^\infty)} \int_0^t\mathcal{E} ( s) \de s + \int_0^t\left (\f v-\vv , (\nabla \f d -\nabla \dd) ^T \Lap \dd \right ) \de s
 \, .
\end{align*}

In order to estimate the term $J_2$, there are additional vector identities of need.
The function $\f d$ associated to the generalized gradient Young measure $(\nu^o,m_t,\nu^\infty)$ 
can be approximated by a sequence of functions $\{\fn d \}\subset L^\infty(0,T; \Hc)$ in the sense of generalized Young measures (see Definition~\ref{def:gradmeas}). 
To show the identity, the integration is again left out.

Since $\vv$ is solenoidal, there holds
\begin{align*}
0 ={}& \left ( \di \vv ,   \nabla \f d_{n } \o \f d_{n } \dreidots \f\Theta \dreidots \nabla \dd \o \dd  \right )\\
={}&-  \left ( ( \vv \cdot \nabla ) \nabla \f d_{n } \o \f d_{n } \dreidotkom \f\Theta \dreidots \nabla \dd \o \dd  \right ) - \left (  \nabla \f d_{n } \o ( \vv \cdot \nabla )\f d_{n } \dreidotkom \f\Theta \dreidots \nabla \dd \o \dd  \right )\\
&- \left (  \nabla \f d_{n } \o \f d_{n } \dreidotkom \f\Theta \dreidots ( \vv \cdot \nabla )\nabla \dd \o \dd  \right )- \left (  \nabla \f d_{n } \o \f d_{n } \dreidotkom \f\Theta \dreidots \nabla \dd \o ( \vv \cdot \nabla )\dd  \right )\\
={}& -\left (  ( \vv \cdot \nabla ) \f d_{n } , - \di \left ( \f d_{n } \cdot \f \Theta \dreidots \nabla \dd \o \dd \right ) + \nabla \f d_{n } : \f \Theta \dreidots \nabla \dd \o \dd \right ) \\
& -\left ( \nabla \left  (( \vv \cdot \nabla ) \dd\right ) \o   \dd \dreidotkom \f \Theta \dreidots \nabla \f d_{n } \o \f d_{n } \right )  -\left (  \nabla \dd \o ( \vv \cdot \nabla ) \dd \dreidotkom \f \Theta \dreidots \nabla \f d_{n } \o \f d_{n } \right )
\\
&+ \left ( \nabla \vv ; \nabla \f d_{n } ^T \left ( \f d_{n } \cdot \f \Theta \dreidots \nabla \dd \o \dd \right )\right )
+ \left ( \nabla \vv ; \nabla \dd ^T \left ( \dd \cdot \f \Theta \dreidots \nabla \f d_{n } \o \f d_{n } \right )\right )\\
={}&K_{1,n }+K_{2,n }+K_{3,n }\,.
\end{align*}
The three terms abbreviate the three lines. 
For the first two lines, going to the limit in the sense of generalized Young measures shows
%%%%%%%%%%%%%%%%%%%%%%%%%%%%%%%%%%%
%% In den folgenden Formeln, steht die laengere schreibweise immer in den auskommentierten BEreichen
%%%%%%%%%%%%%%%%%%%%%%%%%%%%%%%%
\begin{align*}
K_{1,\infty}+K_{2,\infty}
={}&
+ \ll{ \nu_s, \left ( \f S  \vv \o \f S  ^T -  \f S\otimes \left (\f S \vv\right )\right )  \dreidots \f \Theta \dreidots \nabla \dd \o \dd      }+ \left ( \nabla \f d  \vv \o  \f d  ;  \di \left (\f \Theta \dreidots \nabla \dd \o \dd\right )  \right)  \\
& -\left ( \nabla \left  (( \vv \cdot \nabla ) \dd\right ) \o   \dd \dreidotkom \f \Theta \dreidots \nabla \f d \o \f d \right )  -\left (  \nabla \dd \o ( \vv \cdot \nabla ) \dd \dreidotkom \f \Theta \dreidots \nabla \f d \o \f d \right )
%\\
%={}
%&-\left (  ( \vv \cdot \nabla ) \f d , - \di \left ( \dd \cdot \f \Theta \dreidots \nabla \dd \o \dd \right ) + \nabla \dd : \f \Theta \dreidots \nabla \dd \o \dd \right ) 
%- \left ( \nabla ( ( \vv \cdot \nabla \dd ) ; \f d \cdot \f \Theta \dreidots \nabla \f d \o \f d  \right )
%\\
%&- \ll{\nu_s, \f S \o ( \vv \cdot \nabla ) \dd \dreidots \f \Theta \dreidots \f S \o \f h    }
%\\
%&- \ll{ \nu_s, \left ( (\f S - \nabla \dd ) \vv \o (  \nabla \dd -\f S  )^T + ( \f S- \nabla \dd ) \otimes \left (( \f S - \nabla \dd )\vv\right )\right )  \dreidots \f \Theta \dreidots \nabla \dd \o \dd      }\\
%& + \left ( ( \nabla \f d - \nabla \dd ) \vv \o ( \f d - \dd ) ;  \di \left (\f \Theta \dreidots \nabla \dd \o \dd\right )  \right) 
%\\
%&- \left ( \nabla ( ( \vv \cdot \nabla ) \dd) ;   (\dd- \f d) \cdot \f \Theta \dreidots \left (\nabla \f d \o \f d - \nabla \dd \o \dd \right ) \right )\\
%&-\ll{ \nu_s, (\f S - \nabla \dd )\o ( \nabla \dd \vv)  \dreidots\f \Theta \dreidots \left (\f S \o \f h - \nabla \dd \o \dd\right )    }
\,.
\end{align*}
The term $ \di \left (\f \Theta \dreidots \nabla \dd \o \dd\right )$ has to be understood as $ \sum_{j=1}^3\partial_{\f x_j} \left (\f \Theta \dreidots \nabla \dd \o \dd\right )_{ijk}$.
Adding and subtracting the terms
\begin{align*}
\left (  ( \vv \cdot \nabla ) \f d , - \di \left ( \dd \cdot \f \Theta \dreidots \nabla \dd \o \dd \right ) + \nabla \dd : \f \Theta \dreidots \nabla \dd \o \dd \right ) 
+ \left ( \nabla ( ( \vv \cdot \nabla \dd ) ; \f d \cdot \f \Theta \dreidots \nabla \f d \o \f d  \right )
+\ll{\nu_s, \f S \o ( \vv \cdot \nabla ) \dd \dreidots \f \Theta \dreidots \f S \o \f h    }\,,
\end{align*}
 and some additional rearrangements give
\begin{align*}
K_{1,\infty}+K_{2,\infty}
%={}&
%+ \ll{ \nu_s, \left ( \f S  \vv \o \f S  ^T -  \f S\otimes \left (\f S \vv\right )\right )  \dreidots \f \Theta \dreidots \nabla \dd \o \dd      }+ \left ( \nabla \f d  \vv \o  \f d  ;  \di \left (\f \Theta \dreidots \nabla \dd \o \dd\right )  \right)  \\
%& -\left ( \nabla \left  (( \vv \cdot \nabla ) \dd\right ) \o   \dd \dreidotkom \f \Theta \dreidots \nabla \f d \o \f d \right )  -\left (  \nabla \dd \o ( \vv \cdot \nabla ) \dd \dreidotkom \f \Theta \dreidots \nabla \f d \o \f d \right )
%\\
={}
&-\left (  ( \vv \cdot \nabla ) \f d , - \di \left ( \dd \cdot \f \Theta \dreidots \nabla \dd \o \dd \right ) + \nabla \dd : \f \Theta \dreidots \nabla \dd \o \dd \right ) 
- \left ( \nabla ( ( \vv \cdot \nabla \dd ) ; \f d \cdot \f \Theta \dreidots \nabla \f d \o \f d  \right )
\\
&- \ll{\nu_s, \f S \o ( \vv \cdot \nabla ) \dd \dreidots \f \Theta \dreidots \f S \o \f h    }
\\
&- \ll{ \nu_s, \left ( (\f S - \nabla \dd ) \vv \o (  \nabla \dd -\f S  )^T + ( \f S- \nabla \dd ) \otimes \left (( \f S - \nabla \dd )\vv\right )\right )  \dreidots \f \Theta \dreidots \nabla \dd \o \dd      }\\
& + \left ( ( \nabla \f d - \nabla \dd ) \vv \o ( \f d - \dd ) ;  \di \left (\f \Theta \dreidots \nabla \dd \o \dd\right )  \right) 
\\
&- \left ( \nabla ( ( \vv \cdot \nabla ) \dd) ;   (\dd- \f d) \cdot \f \Theta \dreidots \left (\nabla \f d \o \f d - \nabla \dd \o \dd \right ) \right )\\
&-\ll{ \nu_s, (\f S - \nabla \dd )\o ( \nabla \dd \vv)  \dreidots\f \Theta \dreidots \left (\f S \o \f h - \nabla \dd \o \dd\right )    }
\,.
\end{align*}

We observe that the last four lines can be estimated by
\begin{align*}
&\hspace{-2em}-\inttet{ \ll{ \nu_s, \left ( (\f S - \nabla \dd ) \vv \o (  \nabla \dd -\f S  )^T + ( \f S- \nabla \dd ) \otimes \left (( \f S - \nabla \dd )\vv\right )\right )  \dreidots \f \Theta \dreidots \nabla \dd \o \dd      }}\\
&\hspace{-2em} +\inttet{ \left ( ( \nabla \f d - \nabla \dd ) \vv \o ( \f d - \dd ) ;  \di \left (\f \Theta \dreidots \nabla \dd \o \dd\right )  \right) }
%\\
%&\hspace{-2em}
- \inttet{\left ( \nabla ( ( \vv \cdot \nabla ) \dd) \otimes   (\dd- \f d) \dreidotkom \f \Theta \dreidots \left (\nabla \f d \o \f d - \nabla \dd \o \dd \right ) \right )}\\
&\hspace{-2em}-\inttet{\ll{ \nu_s, (\f S - \nabla \dd )\o ( \nabla \dd \vv)  \dreidots\f \Theta \dreidots \left (\f S \o \f h - \nabla \dd \o \dd\right )    }}
\\\leq{}& c\| \dd\|_{L^\infty( \f L^\infty)} \inttet{\| \vv\|_{ \f L^\infty}\| \dd\|_{ \f W^{1,\infty}}\ll{\nu_s , | \f S - \nabla \dd|^2}} \\
&+ c \inttet{\| \vv\|_{ \f L^\infty}\left (\| \dd\|_{ \f W^{2,3}}\| \dd\|_{L^\infty( \f L^\infty)}+ \| \dd\|_{ \f W^{1,6}}^2\right )\left (\| \f d - \dd \|_{\f L^6}^2+ \ll{\nu_s, | \f S -\nabla \dd|^2}\right )}\\
&+c   \inttet{ \| \vv\|_{\f L^\infty}\| \dd\|_{\f W^{2,3}}\left (\| \f d -\dd \|_{\f L^6}^2 + \ll{\nu_s, \left | \f \Theta \dreidots (\f S \o \f h - \nabla \dd \o \dd)\right |^2}\right )}\\
&+c  \inttet{\| \vv\|_{ \f W^{1,3}}\| \dd\|_{ \f W^{1,\infty}}\left (\| \f d -\dd \|_{\f L^6}^2 + \ll{\nu_s,\left  | \f \Theta \dreidots (\f S \o \f h - \nabla \dd \o \dd)\right |^2}\right )}\\
&+c \inttet{ \| \vv\|_{ \f L^\infty}\| \dd\|_{ \f W^{1,\infty}}\left (\ll{\nu_s, |\f S -\nabla \dd^2}+ \ll{\nu_s, \left | \f \Theta \dreidots (\f S \o \f h - \nabla \dd \o \dd)\right |^2 }\right )}
\end{align*}
such that with Lemma~\ref{lem:vschlange} the terms $K_{1,\infty}$ and $K_{2,\infty}$ can be estimated by
\begin{align*}
\int_0^t (K_{1,\infty}+K_{2,\infty})\de s 
%={}&
%+ \ll{ \nu_s, \left ( \f S  \vv \o \f S  ^T -  \f S\otimes \left (\f S \vv\right )\right )  \dreidots \f \Theta \dreidots \nabla \dd \o \dd      }+ \left ( \nabla \f d  \vv \o  \f d  ;  \di \left (\f \Theta \dreidots \nabla \dd \o \dd\right )  \right)  \\
%& -\left ( \nabla \left  (( \vv \cdot \nabla ) \dd\right ) \o   \dd \dreidotkom \f \Theta \dreidots \nabla \f d \o \f d \right )  -\left (  \nabla \dd \o ( \vv \cdot \nabla ) \dd \dreidotkom \f \Theta \dreidots \nabla \f d \o \f d \right )
%\\
\leq {}
&-\inttett{\left (  ( \vv \cdot \nabla ) \f d , - \di \left ( \dd \cdot \f \Theta \dreidots \nabla \dd \o \dd \right ) + \nabla \dd : \f \Theta \dreidots \nabla \dd \o \dd \right ) 
- \left ( \nabla ( ( \vv \cdot \nabla \dd ) ; \f d \cdot \f \Theta \dreidots \nabla \f d \o \f d  \right )}
\\
&- \inttet{\ll{\nu_s, \f S \o ( \vv \cdot \nabla ) \dd \dreidots \f \Theta \dreidots \f S \o \f h    }}+ \int_0^t \mathcal{K}(s) \mathcal{E}(s)\de s
\,.
\end{align*}
The term $K_{3,\infty}$ can be transformed to 
\begin{align*}
K_{3,\infty}={}& \left ( \nabla \vv ; \nabla \f d ^T \left ( \f d \cdot \f \Theta \dreidots \nabla \dd \o \dd \right )\right )+ \left ( \nabla \vv ; \nabla \dd ^T \left ( \dd \cdot \f \Theta \dreidots  (\nabla \f d \o \f d - \nabla \dd \o \dd)\right )\right )\\
&+ \left ( \nabla \vv ; \nabla \dd^T \left (  \dd \cdot \f \Theta \dreidots \nabla \dd \o \dd \right )\right )\,.
\end{align*}
Adding the above terms to $J_2$ and rearranging the terms  gives
\begin{subequations}\label{J2}
\begin{align}
J_2+0{}={}&J_2 +K_{1,\infty}+K_{2,\infty}+K_{3,\infty}\notag \\ \leq {}&\int_0^t\left (  \f d \times ( (\f v - \vv) \cdot \nabla ) \f d , - \dd \times\di \left ( \dd \cdot \f \Theta \dreidots \nabla \dd \o \dd \right ) + \dd \times\nabla \dd : \f \Theta \dreidots \nabla \dd \o \dd \right )\de s \label{line1}
\\
&+\int_0^t\left (  \f d \times (  \vv \cdot \nabla ) \f d , - (\dd- \f d) \times\di \left ( \dd \cdot \f \Theta \dreidots \nabla \dd \o \dd \right ) + \nabla \dd : \f \Theta \dreidots \nabla \dd \o \dd \right )\de s\label{line2}\\
 &- \int_0^t
\ll{ \nu_s, (\f S- \nabla \dd)\nabla \vv \o (\f h -\dd)   \dreidots  \f \Theta \dreidots \left (\f S \o \f h -\nabla \dd \o \dd \right )  }\de s  \label{line3}\\
 &- \int_0^t
\ll{ \nu_s,  \left (\f S - \nabla \dd\right )  \nabla \vv \o \dd  \dreidots  \f \Theta \dreidots \left (\f S \o \f h -\nabla \dd \o \dd \right )  }\de s  \label{line4}\\
& + \inttet{\left (( \rot{\f d}-\rot\dd)^T\nabla ( \rot{\dd} \nabla \dd \vv )\o \f d  \dreidotkom \f\Theta \dreidots \nabla \f d \o \f d \right )}\label{line5}
\\
& + \inttet{\ll{\nu_s , \f \Upsilon : \left (\rot\dd ( \vv \cdot \nabla) \dd \o ( \f S -\nabla \dd)\right )\o \f h \dreidots \f \Theta \dreidots \f S \o \f h }} \label{line6}\\
& + \inttet{\ll{\nu_s, \f S \o \left (( \rot{\f d} - \rot{\dd})^T\rot{\dd}( \vv \cdot \nabla ) \dd \dreidots \f \Theta \dreidots \f S \o \f h \right )}}\label{line7}
\\
& - \int_0^t\left ( \nabla \f v -\nabla \vv ; \nabla \dd^T \left (  \dd \cdot \f \Theta \dreidots \nabla \dd \o \dd \right )\right )\de s+ \int_0^t \mathcal{K}(s) \mathcal{E}(s)\de s \label{line8} \, .
\end{align}
\end{subequations}
Since $\f v$ and $\vv$ are solenoidal and due to Lemma~\ref{lem:vschlange}, there holds
\begin{align*}
\left ( \nabla \f v -\nabla \vv ; \nabla \dd^T \left (  \dd \cdot \f \Theta \dreidots \nabla \dd \o \dd \right )\right ) 
={}&  \left (  ( (\f v - \vv) \cdot \nabla ) \dd , - \di \left ( \dd \cdot \f \Theta \dreidots \nabla \dd \o \dd \right ) + \nabla \dd : \f \Theta \dreidots \nabla \dd \o \dd \right )\\
&- \left ( \f v - \vv , \nabla  ( \nabla \dd \o \dd \dreidots \f \Theta \dreidots \nabla \dd \o \dd )\right )
\\
={}&  \left (  \rot{\dd}^T \rot \dd ( (\f v - \vv) \cdot \nabla ) \dd , - \di \left ( \dd \cdot \f \Theta \dreidots \nabla \dd \o \dd \right ) + \nabla \dd : \f \Theta \dreidots \nabla \dd \o \dd \right )\,.
\end{align*} 
 The first~\eqref{line1} and the last line~\eqref{line8} of~\eqref{J2} can be transformed to
\begin{align*}
&\hspace{-2em}\int_0^t\left ( \rot{\dd }^T \rot{\f d}  ( (\f v - \vv) \cdot \nabla ) \f d , - \di \left ( \dd \cdot \f \Theta \dreidots \nabla \dd \o \dd \right ) + \nabla \dd : \f \Theta \dreidots \nabla \dd \o \dd \right )\de s\\
&\hspace{-2em}- \int_0^t\left ( \nabla \f v -\nabla \vv ; \nabla \dd^T \left (  \dd \cdot \f \Theta \dreidots \nabla \dd \o \dd \right )\right )\de s\\
={}
&\int_0^t\left (  \rot{\dd}^T ( \rot{\f d}-\rot{\dd})( (\f v - \vv) \cdot \nabla ) (\f d -\dd) , 
- \di \left ( \dd \cdot \f \Theta \dreidots \nabla \dd \o \dd \right ) + \nabla \dd : \f \Theta \dreidots \nabla \dd \o \dd \right )\de s\\
&+\int_0^t\left (  \rot{\dd}^T \rot{\dd}( (\f v - \vv) \cdot \nabla ) (\f d -\dd) , 
- \di \left ( \dd \cdot \f \Theta \dreidots \nabla \dd \o \dd \right ) + \nabla \dd : \f \Theta \dreidots \nabla \dd \o \dd \right )\de s
\\
&+\inttet{ \left (  \rot{\dd}^T ( \rot {\f d}- \rot \dd) ( (\f v - \vv) \cdot \nabla ) \dd , - \di \left ( \dd \cdot \f \Theta \dreidots \nabla \dd \o \dd \right ) + \nabla \dd : \f \Theta \dreidots \nabla \dd \o \dd \right )}
\,,
\end{align*}
and estimated by
\begin{align*}
&\hspace{-2em}\int_0^t\left ( \rot{\dd }^T \rot{\f d}  ( (\f v - \vv) \cdot \nabla ) \f d , - \di \left ( \dd \cdot \f \Theta \dreidots \nabla \dd \o \dd \right ) + \nabla \dd : \f \Theta \dreidots \nabla \dd \o \dd \right )\de s\\
&\hspace{-2em}- \int_0^t\left ( \nabla \f v -\nabla \vv ; \nabla \dd^T \left (  \dd \cdot \f \Theta \dreidots \nabla \dd \o \dd \right )\right )\de s\\
\leq{}& \delta \inttet{\| \f v - \vv \|_{\f L^6}^2} +  C_\delta  \inttet{\left ( \| \dd\|_{L^\infty( \f L^\infty)}^3\| \dd\|_{ \f W^{2,3}}+\| \dd\|_{L^\infty( \f L^\infty)}^2\| \dd\|_{ \f W^{1,6}}^2  \right )^2\ll{\nu_s , | \f S -\nabla \dd |^2 | \f h - \dd |^2}} 
\\
&+  C_\delta \inttet{ \left ( \| \dd\|_{L^\infty( \f L^\infty)}^4\| \dd\|_{ \f W^{2,3}}+\| \dd\|_{L^\infty( \f L^\infty)}^3\| \dd\|_{ \f W^{1,6}}^2  \right )^2\ll{\nu_s , | \f S -\nabla \dd |^2 }} 
\\
&+  C_\delta \inttet{ \| \dd\|_{L^\infty( \f W^{1,3})}\left ( \| \dd\|_{L^\infty( \f L^\infty)}^3\| \dd\|_{ \f W^{2,3}}+\| \dd\|_{L^\infty( \f L^\infty)}^2 \| \dd\|_{\f W^{1,6}}^2  \right )^2\| \f d-\dd\|_{\f L^6}^2} 
\end{align*}
The line~\eqref{line3} can be estimated
by
\begin{align*}
- \int_0^t&
\ll{ \nu_s, (\f S- \nabla \dd)\nabla \vv \o (\f h -\dd)   \dreidots  \f \Theta \dreidots \left (\f S \o \f h -\nabla \dd \o \dd \right )  }\de s  \\
&\leq c\inttet{ \| \vv\|_{\f W^{1,\infty}} \left (\ll{\nu_s,| \f S -\nabla \dd |^2| \f h -\dd|^2}+ \ll{\nu_s, \left | \f \Theta \dreidots \left (\f S \o \f h - \nabla \dd \o \dd \right )\right |^2}\right )}\,,
\end{align*}
and the line~\eqref{line4} by
\begin{align*}
  \int_0^t&
\ll{ \nu_s,  \left (\f S - \nabla \dd\right )  \nabla \vv \o \dd  \dreidots  \f \Theta \dreidots \left (\f S \o \f h -\nabla \dd \o \dd \right )  }\de s  \\
&\leq c\| \dd\|_{L^\infty(\f L^\infty)} \inttet{ \| \vv\|_{\f W^{1,\infty}}\left (\ll{\nu_s,| \f S -\nabla \dd |^2| \f h -\dd|^2}+ \ll{\nu_s, \left | \f \Theta \dreidots \left (\f S \o \f h - \nabla \dd \o \dd \right )\right |^2}\right )}\,.
\end{align*}
For the line~\eqref{line2} we observe that it can be rearranged to
\begin{align}
\begin{split}
\int_0^t&\left (  \f d \times (  \vv \cdot \nabla ) \f d , - (\dd- \f d) \times\di \left ( \dd \cdot \f \Theta \dreidots \nabla \dd \o \dd \right ) + \nabla \dd : \f \Theta \dreidots \nabla \dd \o \dd \right )\de s\\
={}&\int_0^t\left ( \left (\rot{\dd-\f d}\right )^T (\rot{\f d-\dd})   (\nabla  \f d- \nabla \dd)\vv , -  \di \left ( \dd \cdot \f \Theta \dreidots \nabla \dd \o \dd \right ) + \nabla \dd : \f \Theta \dreidots \nabla \dd \o \dd \right )\de s\\
&+\int_0^t\left ( \left (\rot{\dd-\f d}\right )^T\rot{\dd}  (  \vv \cdot \nabla ) (\f d- \dd) , -   \di \left ( \dd \cdot \f \Theta \dreidots \nabla \dd \o \dd \right ) + \nabla \dd : \f \Theta \dreidots \nabla \dd \o \dd \right )\de s\\
&+\int_0^t\left ( \left (\rot{\dd-\f d}\right )^T (\rot{\f d-\dd})    \nabla \dd\vv , -  \di \left ( \dd \cdot \f \Theta \dreidots \nabla \dd \o \dd \right ) + \nabla \dd : \f \Theta \dreidots \nabla \dd \o \dd \right )\de s\\
 &+\int_0^t\left ( \left (\rot{\dd-\f d}\right )^T  \rot{\dd}  (  \vv \cdot \nabla )  \dd , -   \di \left ( \dd \cdot \f \Theta \dreidots \nabla \dd \o \dd \right ) + \nabla \dd : \f \Theta \dreidots \nabla \dd \o \dd \right )\de s\,.
 \end{split}\label{dieline2}
\end{align}
The first three lines on the right-hand side of~\eqref{dieline2} can be estimated by
\begin{align*}
&\hspace{-2em}\int_0^t\left ( \left (\rot{\dd}-\rot{\f d}\right )^T (\rot{\f d}- \rot{\dd})   \nabla  (\f d- \dd)\vv , -  \di \left ( \dd \cdot \f \Theta \dreidots \nabla \dd \o \dd \right ) + \nabla \dd : \f \Theta \dreidots \nabla \dd \o \dd \right )\de s\\
&\hspace{-2em}+\int_0^t\left ( \left (\rot{\dd}-\rot{\f d}\right )^T\rot{\dd} \times (  \vv \cdot \nabla ) (\f d- \dd) , -   \di \left ( \dd \cdot \f \Theta \dreidots \nabla \dd \o \dd \right ) + \nabla \dd : \f \Theta \dreidots \nabla \dd \o \dd \right )\de s\\
&\hspace{-2em}+\int_0^t\left ( \left (\rot{\dd-\f d}\right )^T (\rot{\f d-\dd})    \nabla \dd\vv , -  \di \left ( \dd \cdot \f \Theta \dreidots \nabla \dd \o \dd \right ) + \nabla \dd : \f \Theta \dreidots \nabla \dd \o \dd \right )\de s\\
\leq{}&   c  \inttet{ \| \vv\|_{\f L^{\infty}} \| \dd\|_{L^\infty( \f L^\infty)}^2\| \dd\|_{\f W^{2,3}}\left (\ll{\nu_s , | \f S -\nabla \dd |^2 | \f h - \dd |^2}+ \|\f d - \dd\|_{\f L^6}^2\right )} 
\\
& +  c    \inttet{ \| \vv\|_{\f L^{\infty}}\| \dd\|_{L^\infty( \f L^\infty)}\| \dd\|_{ \f W^{1,6}}^2  \left (\ll{\nu_s , | \f S -\nabla \dd |^2 | \f h - \dd |^2}+ \|\f d - \dd\|_{\f L^6}^2\right )} 
\\
&+ c \inttet{ \| \vv\|_{\f L^{\infty}}   \| \dd\|_{L^\infty( \f L^\infty)}^3\| \dd\|_{ \f W^{2,3}}\left (\ll{\nu_s , | \f S -\nabla \dd |^2 }+ \|\f d - \dd\|_{\f L^6}^2\right )} 
\\
&+ c  \inttet{\| \vv\|_{\f L^{\infty}}  \| \dd\|_{L^\infty( \f L^\infty)}^2\| \dd\|_{\f W^{1,6}}^2  \left (\ll{\nu_s , | \f S -\nabla \dd |^2 }+ \|\f d - \dd\|_{\f L^6}^2\right )} 
\\
&+ c  \inttet{ \| \vv\|_{\f L^{\infty}}   \| \dd\|_{L^\infty( \f W^{1,3})} \| \dd\|_{L^\infty( \f L^\infty)}^2\| \dd\|_{ \f W^{2,3}}\|\f d - \dd\|_{\f L^6}^2} 
\\
&+ c  \inttet{\| \vv\|_{\f L^{\infty}}  \| \dd\|_{L^\infty( \f W^{1,3})} \| \dd\|_{L^\infty( \f L^\infty)}\| \dd\|_{\f W^{1,6}}^2  \|\f d - \dd\|_{\f L^6}^2} 
\,.
\end{align*}
The last line of~\eqref{dieline2} can be transformed by an integration-by-parts into
\begin{subequations}
\begin{align}
\int_0^t&\left ( \left (\rot{\dd}-\rot{\f d}\right )^T  \rot{\dd}  (  \vv \cdot \nabla )  \dd , -   \di \left ( \dd \cdot \f \Theta \dreidots \nabla \dd \o \dd \right ) + \nabla \dd : \f \Theta \dreidots \nabla \dd \o \dd \right )\de s\notag\\
={}&   \inttet{\left (( \rot{\dd}-\rot{\f d})^T\nabla ( \rot{\dd} \nabla \dd \vv )\o \dd \dreidotkom \f\Theta \dreidots \nabla \dd \o \dd \right )}\label{term1}\\
& +\inttet{\left (  \f \Upsilon : \left (\rot{\dd}  ( \vv \cdot \nabla) \dd \o ( \nabla \dd -\nabla \f d)\right )\o \dd \dreidotkom \f \Theta \dreidots  \nabla \dd  \o  \dd \right )}\label{term2}\\
&+\int_0^t\left (\nabla \dd \o \left (\rot{ \dd}-\rot{\f d}\right )^T \rot{\dd}(  \vv \cdot \nabla ) \dd \dreidotkom  \f \Theta \dreidots \nabla \dd \o \dd \right )\de s\,.\label{term3}
\end{align}
\end{subequations}
It remains to estimate the three terms of the lines \eqref{term1}-\eqref{term3} and the three remaining terms of the lines~\eqref{line5}-\eqref{line7} of \eqref{J2}.
The sum of the lines~\eqref{line5} and~\eqref{term1} can be rearranged and estimated by
\begin{align*}
 &\hspace{-2em}\inttet{\left (( \rot{\dd}-\rot{\f d})^T\nabla ( \rot{\dd} \nabla \dd \vv )\o \dd \dreidotkom \f\Theta \dreidots \nabla \dd \o \dd \right )}
 + \inttet{\left (( \rot{\f d}-\rot\dd)^T\nabla ( \rot{\dd} \nabla \dd \vv )\o \f d  \dreidotkom \f\Theta \dreidots \nabla \f d \o \f d \right )}\\
={}& \inttet{\left (( \rot{\f d}-\rot\dd)^T\nabla ( \rot{\dd} \nabla \dd \vv )\o (\f d- \dd)  \dreidotkom \f\Theta \dreidots\left ( \nabla \f d \o \f d - \nabla \dd \o \dd\right )\right )}\\
&+ \inttet{\left (( \rot{\f d}-\rot\dd)^T\nabla ( \rot{\dd} \nabla \dd \vv )\o  \dd  \dreidotkom \f\Theta \dreidots\left ( \nabla \f d \o \f d - \nabla \dd \o \dd\right )\right )}\\
&+ \inttet{\left (( \rot{\f d}-\rot\dd)^T\nabla ( \rot{\dd} \nabla \dd \vv )\o (\f d- \dd)  \dreidotkom \f\Theta \dreidots  \nabla \dd \o \dd \right )}\\
\leq{}   
&  c \inttet{ \left \| \nabla ( \rot{\dd} \nabla \dd \vv )\right \|_{ \f L^3}\left (\| \f d -\dd\|_{\f L^{12}}^4 + \ll{\nu_s , \left| \f \Theta \dreidots \left (\f S \o \f h - \nabla \dd \o \dd \right ) \right |^2}\right )}\\
&+c\| \dd\|_{L^\infty( \f L^\infty)} \inttet{ \left \| \nabla ( \rot{\dd} \nabla \dd \vv )\right \|_{ \f L^3}\left (\| \f d -\dd\|_{\f L^6}^2 + \ll{\nu_s , \left| \f \Theta \dreidots \left (\f S \o \f h - \nabla \dd \o \dd \right ) \right |^2}\right )}
\\
&+c\| \dd\|_{L^\infty( \f L^\infty)}\| \dd\|_{L^\infty( \f W^{1,3})} \inttet{ \left \| \nabla ( \rot{\dd} \nabla \dd \vv )\right \|_{ \f L^3}\| \f d -\dd\|_{\f L^6}^2}
%
%\left (\| \vv\|_{L^\infty(\f L^{\infty})}   \| \dd\|_{L^\infty( \f L^\infty)}^2\| \dd\|_{L^\infty( \f W^{2,3})}+\| \vv\|_{L^\infty(\f L^{\infty})}  \| \dd\|_{L^\infty( \f L^\infty)}\| \dd\|_{L^\infty( \f W^{1,6})}^2 +\| \vv\|_{L^\infty(\f W^{1,3})}  \| \dd\|_{L^\infty( \f L^\infty)}^2\| \dd\|_{L^\infty( \f W^{1,\infty})} \right ) \inttet{\| \f d -\dd\|_{\f L^6}^2 + \ll{\nu_s , \left| \f \Theta \dreidots \left (\f S \o \f h - \nabla \dd \o \dd \right ) \right |}}
\,.
\end{align*}
We see that the norm on the right hand side can be bounded by
\begin{align*}
\left \| \nabla ( \rot{\dd} \nabla \dd \vv )\right \|_{ \f L^3}\leq {}&
\| \vv\|_{\f L^{\infty}}   \| \dd\|_{L^\infty( \f L^\infty)}\| \dd\|_{ \f W^{2,3}}+\| \vv\|_{\f L^{\infty}}  \| \dd\|_{ \f W^{1,6}}^2 +\| \vv\|_{\f W^{1,3}}  \| \dd\|_{L^\infty( \f L^\infty)}\| \dd\|_{ \f W^{1,\infty}} 
\end{align*}
Regarding~\eqref{line6} and \eqref{term2}, we observe
\begin{align*}
 &\hspace{-2em}\inttet{\ll{\nu_s , \f \Upsilon : \left (\rot\dd ( \vv \cdot \nabla) \dd \o ( \f S -\nabla \dd)\right )\o \f h \dreidots \f \Theta \dreidots \f S \o \f h }} \\
&\hspace{-2em}+\inttet{\left (  \left (\f \Upsilon : \left (\rot{\dd}  ( \vv \cdot \nabla) \dd \o ( \nabla \dd -\nabla \f d)\right )\right )\o \dd \dreidotkom \f \Theta \dreidots  \nabla \dd  \o  \dd \right )}\\
={}&  \inttet{\ll{\nu_s , \left (\f \Upsilon : \left (\rot\dd ( \vv \cdot \nabla) \dd \o ( \f S -\nabla \dd)\right )\right )\o (\f h-\dd) \dreidots \f \Theta \dreidots\left ( \f S \o \f h - \nabla \dd \o \dd \right )  }}\\
&+\inttet{\ll{\nu_s , \left (\f \Upsilon : \left (\rot\dd ( \vv \cdot \nabla) \dd \o ( \f S -\nabla \dd)\right )\right )\o \dd \dreidots \f \Theta \dreidots\left ( \f S \o \f h - \nabla \dd \o \dd \right )  }}\\
&+\inttet{\left ( \left ( \f \Upsilon : \left (\rot{\dd}  ( \vv \cdot \nabla) \dd \right ) \o ( \nabla \dd -\nabla \f d)\right )\o (\dd- \f d) \dreidotkom \f \Theta \dreidots  \nabla \dd  \o  \dd \right )}\,.
\end{align*}
Thus, the right-hand side can be estimated by 
\begin{align*}
&\hspace{-2em}\inttet{\ll{\nu_s , \f \Upsilon : \left (\rot\dd ( \vv \cdot \nabla) \dd \o ( \f S -\nabla \dd)\right )\o \f h \dreidots \f \Theta \dreidots \f S \o \f h }} \\
&\hspace{-2em}+\inttet{\left (  \f \Upsilon : \left (\rot{\dd}  ( \vv \cdot \nabla) \dd \o ( \nabla \dd -\nabla \f d)\right )\o \dd \dreidotkom \f \Theta \dreidots  \nabla \dd  \o  \dd \right )}\\
\leq {}&  \inttet{\| \vv\|_{\f L^{\infty}}  \| \dd\|_{L^\infty( \f L^\infty)}\| \dd\|_{ \f W^{1,\infty}}\left (\ll{\nu_s, | \f S- \nabla \dd|^2 | \f h - \dd|^2}+ \ll{\nu_s,  \left |\f \Theta\dreidots \left (\f S \o \f h- \nabla \dd \o \dd \right )\right |^2}\right )}\\
&+ \inttet{\| \vv\|_{\f L^{\infty}}  \| \dd\|_{L^\infty( \f L^\infty)}^2\| \dd\|_{\f W^{1,\infty}} \left (\ll{\nu_s, | \f S- \nabla \dd|^2 }+ \ll{\nu_s,  \left |\f \Theta\dreidots \left (\f S \o \f h- \nabla \dd \o \dd \right )\right |^2}\right )}\\
&+  \inttet{\| \vv\|_{\f L^{\infty}}  \| \dd\|_{L^\infty( \f L^\infty)}^3\| \dd\|_{ \f W^{1,3}}\left (\ll{\nu_s, | \f S- \nabla \dd|^2 }+ \|\f d -\dd\|_{\f L^6}^2\right )}
\end{align*}
Regarding the lines~\eqref{line7} and~\eqref{term3}, we observe
\begin{align*}
&\hspace{-2em} \inttet{\ll{\nu_s, \f S \o \left (( \rot{\f d} - \rot{\dd})^T\rot{\dd}( \vv \cdot \nabla ) \dd \dreidots \f \Theta \dreidots \f S \o \f h \right )}}+\int_0^t\left (\nabla \dd \o \left (\rot{ \dd}-\rot{\f d}\right )^T \rot{\dd}(  \vv \cdot \nabla ) \dd \dreidotkom  \f \Theta \dreidots \nabla \dd \o \dd \right )\de s\hspace{-2em}\\
={}& \inttet{\ll{\nu_s,( \f S- \nabla \dd) \o \left (( \rot{\f d} - \rot{\dd})^T\rot{\dd}( \vv \cdot \nabla ) \dd\right ) \dreidots \f \Theta \dreidots \left ( \f S \o \f h- \nabla \dd \o \dd \right ) }}\\
&+ \inttet{\ll{\nu_s, \nabla \dd \o \left (( \rot{\f d} - \rot{\dd})^T\rot{\dd}( \vv \cdot \nabla ) \dd\right ) \dreidots \f \Theta \dreidots \left ( \f S \o \f h- \nabla \dd \o \dd \right ) }}\\
&+\int_0^t\left ((\nabla \dd- \nabla \f d) \o \left (\left (\rot{ \dd}-\rot{\f d}\right )^T \rot{\dd}(  \vv \cdot \nabla ) \dd\right ) \dreidotkom  \f \Theta \dreidots \nabla \dd \o \dd \right )\de s\\
\end{align*}
Thus, the right-hand side can be estimated by 
\begin{align*}
& \hspace{-2em}\inttet{\ll{\nu_s, \f S \o \left (( \rot{\f d} - \rot{\dd})^T\rot{\dd}( \vv \cdot \nabla ) \dd \dreidots \f \Theta \dreidots \f S \o \f h \right )}}+\int_0^t\left (\nabla \dd \o \left (\rot{ \dd}-\rot{\f d}\right )^T \rot{\dd}(  \vv \cdot \nabla ) \dd \dreidotkom  \f \Theta \dreidots \nabla \dd \o \dd \right )\de s\hspace{-2em}\\
\leq{}&   \| \dd\|_{L^\infty( \f L^\infty)}\| \dd\|_{L^\infty( \f W^{1,\infty})} \inttet{\| \vv\|_{\f L^{\infty}}\left (\ll{\nu_s, | \f S- \nabla \dd|^2 | \f h - \dd|^2}+ \ll{\nu_s,  \left |\f \Theta\dreidots \left (\f S \o \f h- \nabla \dd \o \dd \right )\right |  ^2 }\right )}\\
&+  \| \dd\|_{L^\infty( \f L^\infty)} \inttet{\| \vv\|_{\f L^{\infty}}\| \dd\|_{ \f W^{1,6}}^2 \left (\|\f d - \dd\|_{\f L^6}^2 + \ll{\nu_s,  \left |\f \Theta\dreidots \left (\f S \o \f h- \nabla \dd \o \dd \right )\right |^2}\right )}\\
&+   \| \dd\|_{L^\infty( \f L^\infty)}^2  \inttet{\| \vv\|_{\f L^{\infty}}\| \dd\|_{ \f W^{1,6}}^2\left (\ll{\nu_s, | \f S- \nabla \dd|^2 }+ \|\f d -\dd\|_{\f L^6}^2\right )}
\end{align*}
Together we get for the term $I_9$
\begin{align*}
I_9 \leq \delta \int_0^t \mathcal{W}(s) \de s +  \int_0^t \mathcal{K}(s) \mathcal{E}(s) \de s \, .
\end{align*}
At last, the term $I_{10}$ can be estimated by 
\begin{align*}
|I_{10}| = \left |\int_0^t \ll{ \mu_s , \f \Gamma \dreidots ( \f \Gamma \cdot \nabla \vv )} \de t  \right |  \leq  c \int_0^t\| (\nabla \vv)_{\sym} \|_{\f L^\infty} \ll{\mu_s , 1} \de s \leq   \int_0^t \| \vv \|_{\f W^{1,\infty}} \mathcal{E}(s) \de s 
\end{align*}

%%%%%%%%%%%%%%%%%%%%%%%%%%%
%%%%%%%%%%%%%%%%%%%%%%%%%%%%%
%%%%%%%%%%%%%%%%%%%%%%%%%%%%

%%%%%%%%%%%%%%%%%%%%%%%%%%%%%%%%%%%%%%%%
%%%%%%%%%%%%%%%%%%%%%%%%%%%%%%%%%%%%%%%%
%%%%%%%%%%%%%%%%%%%%%%%%%%%%%%%%%%%%%%%%% 

 Finally, we insert everything back into~\eqref{vormeins}. We end up with
 \begin{align*}
 \mathcal{E}(t) + \inttet{\mathcal{W}(s)} \leq{}& \mathcal{E}(0)  + \frac{1}{2} \mathcal{E}(t) + \delta \int_0^t \mathcal{W}(s) \de s + C_\delta \int_0^t \mathcal{K}(s)\mathcal{E}(s) \de s \\
&
+\left ((\nabla \f d(0)-\nabla \dd (0))\o (\f d(0)-  \dd(0))\right ) \dreidots \f \Theta \dreidots (\nabla \dd (0)\o \dd(0))
\\
&+ c\|  \nabla \dd  \o \dd \|_{L^\infty(\f L^\infty)}^2\left \|   \f d(0)- \dd(0)    \right \|_{\Le}^2  \\
& +((\mu_2+\mu_3)-\lambda)  \inttet{\left ( \dd \times\syv \dd - \f d \times \sy v \f d , \dd \times \tq - \f d \times \f q  \right )} 
 \end{align*}
Since the constants are assumed to fulfil the dissipative relation~\eqref{con}
we can find a real number $\zeta \in (0,1)$ such that
\begin{align}
(\mu_2+\mu_3)-\lambda \leq \zeta^2 4  ( \mu_5+\mu_6-\lambda(\mu_2+\mu_3)).
\end{align}
we estimate the relative entropy further on with Youngs and H\"{o}lders inequality  
\begin{align*}
\frac{1}{2}\mathcal{E}(t) +\inttet{\mathcal{W}(s)}    \leq{}&  \mathcal{E}(0)+\left ((\nabla \f d(0)-\nabla \dd (0))\o (\f d(0)-  \dd(0))\right ) \dreidots \f \Theta \dreidots (\nabla \dd (0)\o \dd(0))
\\
&+ c\left \|   \f d(0)- \dd(0)    \right \|_{\Le}^2  + \delta  \inttet{\mathcal{W}(s )} + \inttet{\mathcal{K}(s) \mathcal{E}( s )}  \\& +  \zeta \inttet{ \|\f d \times \f q-\dd \times\tq\|_{\Le}^2 } \\&+\zeta\inttet{ ( \mu_5+\mu_6-\lambda(\mu_2+\mu_3))\|\f d \times \sy{v} \f d -\dd \times  \syv \dd \|_{\Le}^2} 
%\\
%& \quad 
%\\
%& \leq \mathcal{E}( 0) +(\zeta +\delta  ) \intte{\mathcal{W}( s )} + c \intte{ \mathcal{E}( s )}
\,.
\end{align*}
The last term on the right hand side can be written as
\begin{align*}
&\hspace{-2em}\inttet{\|\f d \times \sy{v} \f d -\dd \times  \syv \dd \|_{\Le}^2 }\\
\leq{}& \inttett{\|\f d \times( \sy{v} \f d-  \syv \dd) \|_{\Le}^2 + \|(\f d -\dd) \times  \syv \dd \|_{\Le}^2}
\\
\leq{}& 
\inttett{\| \sy{v} \f d - \syv \dd \|_{\Le}^2 - \|\f d \cdot( \sy{v} \f d - \syv \dd) \|_{L^2}^2 }
\\ &+ \|\dd\|_{L^\infty(\f L^\infty)}^2  \inttet{\| \vv\|_{\f W^{1,3}}^2\| \f d -\dd \| _{\f L^6}^2}
\\
\leq{}& \inttet{\| \sy{v} \f d - \syv \dd \|_{\Le}^2}  +\|\dd\|_{L^\infty(\f L^\infty)}^2  \inttet{ \| \vv\|_{\f W^{1,3}}^2\| \f d -\dd \| _{\f L^6}^2}\,,
 \end{align*}
such that we get  the following estimate 
\begin{align*}
\frac{1}{2}\mathcal{E}(t) +\inttet{\mathcal{W}(s)}    \leq{} &\mathcal{E}( 0) 
+\left ((\nabla \f d(0)-\nabla \dd (0))\o (\f d(0)-  \dd(0))\right ) \dreidots \f \Theta \dreidots (\nabla \dd (0)\o \dd(0))
\\
&+ c\|  \nabla \dd  \o \dd \|_{L^\infty(\f L^\infty)}^2\left \|   \f d(0)- \dd(0)    \right \|_{\Le}^2
+(\zeta +\delta  ) \inttet{\mathcal{W}( s )} +  \inttet{ \mathcal{K}(s)\mathcal{E}( s )} \,.
\end{align*}
We now choose $\delta $ sufficiently small, such that $ \delta \leq (1- \zeta )$. 
The assertion of Proposition~\ref{lem:main} immediately follows from Gronwalls estimate.

\end{proof}
\appendix
\section{Appendix\label{sec:app}}
For $\f d \in \C^1(\Omega)$ the terms of the quadratic Oseen--Frank free energy (see Section~\ref{sec:model} and~\eqref{frei}) depending on $\nabla \f d$ can be expressed using a tensor of fourth order:
\begin{align}
\begin{split}
|\nabla \f d|^2 &= \nabla \f d : \f \Lambda^0 : \nabla \f d
\quad \text{mit } \f\Lambda^0_{ijkl} = \f \delta_{ik}\f \delta_{jl}\,,
\\
(\di \f d)^2 & =
\nabla \f d : \f \Lambda^1 : \nabla \f d \quad \text{mit } \f \Lambda^1_{ijkl} = \f \delta_{ij}\f \delta_{kl}\,,
\\
\tr (\nabla \f d ^2) &= \nabla \f d : \f \Lambda^2 : \nabla \f d \quad \text{mit } \f \Lambda^2_{ijkl} = \f \delta_{il}\f \delta_{jk}\,,
\\
|\nabla \times \f d|^2 &= \nabla \f d : \f \Lambda^3 : \nabla \f d \quad \text{mit }  \f \Lambda^3 = \f \Lambda^0 - \f \Lambda^2 \, .
\end{split}
\label{Lambdarech}
\end{align}
Here $\f \Lambda^0$ is the identity in $\R^{3\times 3}$
Similar, we see for the non-quadratic terms in the Oseen--Frank energy (see Section~\ref{sec:model} and~\eqref{frei})
\begin{align*}
\begin{split}
|\nabla \f d|^2|\f d|^2  &= \nabla \f d \o \f d  \dreidots \f \Theta ^0 \dreidots  \nabla \f d\o \f d 
\quad\text{mit } \f\Theta^0_{ijklmn} = \f \delta_{il}\f \delta_{jm}\f \delta_{kn}\,,
\\
(\di \f d)^2 |\f d |^2 & =
\nabla \f d \o\f d \dreidots \f \Theta^1 \dreidots \nabla \f d \o\f d \quad \text{mit } \f \Theta^1_{ijklmn} = \f \delta_{ij}\f \delta_{lm}\f \delta_{kn}\,,
\\
%\tr (\nabla \f d ^2) &= \nabla \f d : \f \Lambda^2 : \nabla \f d \quad \text{mit } \f \Lambda^2_{ijkl} = \f \delta_{il}\f \delta_{jk}
%\\
|\nabla \times \f d|^2|\f d |^2 &=2| \sk d |^2 |\f d |^2 \\&=   \nabla \f d \o \f d \dreidots \f \Theta^2 \dreidots \nabla \f d \o\f d \quad \text{mit }  \f \Theta  ^2_{ijklmn} = \f \delta_{kn} ( \f \delta_{il}\f \delta_{jm} - \f \delta_{im}\f \delta_{jl})  \, ,
\\
|\f d \times \nabla \times \f d|^2 &=4| \sk d\f d  |^2  \\&=   \nabla \f d \o \f d \dreidots \f \Theta^3 \dreidots \nabla \f d \o\f d \quad \text{mit }  \f \Theta  ^3 _{ijklmn}=  \f \delta_{il}\f \delta_{mn}\f \delta_{jk} - \f \delta_{mi}\f \delta_{ln}\f \delta_{jk} \\ & \hspace{6cm} - \f \delta_{lj}\f \delta_{mn}\f \delta_{ik} + \f \delta_{jm}\f \delta_{ln}\f \delta_{ik}\,,\\
| \f d \cdot \curl \f d |^2 & = ( \rot{\f d}:\sk d )^2 \\
& =   \nabla \f d \o \f d \dreidots \f \Theta^4_{ijklmn} \dreidots \nabla \f d \o\f d \quad \text{mit }  \f \Theta  ^4 =   \f \delta_{kn}\f \delta_{jm}\f \delta_{il} + \f \delta_{km}\f \delta_{jl}\f \delta_{in} + \f \delta_{kl}\f \delta_{jn}\f \delta_{im}\\ & \hspace{6cm} - \f \delta_{kn}\f \delta_{jl}\f \delta_{im}- \f \delta_{km}\f \delta_{jn}\f \delta_{il} - \f \delta_{kl}\f \delta_{jm}\f \delta_{in}   \,.
\end{split}
%\label{thetarech}
\end{align*}

%%%%%%%%%%%%%%%%%%%%%%%%%%%%%%%%%%%%%%%%%%%%
With this definitions, we prove the algebraic relations stated in Proposition~\ref{Sobolev}.
\begin{proposition}
For the tensor $\f \Theta = k_3 \f \Theta^1 + k_4 \f \Theta^2 + k_5 \f \Theta^3 $ with $k_i>0$, for all $i\in\{3,4,5\}$ (see~\eqref{ThetaOF}), there exists a constant $c>0$ such that
\begin{align*}
| \f \Theta \dreidots \left (\f S \o \f h- \tilde{\f S}\o \tilde{\f h}\right ) |^2  \leq c \left ( \left (\f S \o \f h- \tilde{\f S}\o \tilde{\f h}\right ) \dreidots \f \Theta \dreidots \left (\f S \o \f h- \tilde{\f S}\o \tilde{\f h}\right )\right )\,, \quad \text{for all } \f S \o \f h \in \R^{3\times 3} \times \R^3 \,.  
\end{align*}
\end{proposition}
\begin{proof}
We consider the tensors  $\f \Theta^i$ with  $i\in\{1,\ldots,3\}$ multiplied with  $\f S \o \f h$
\begin{align*}
\left (\f \Theta^1\dreidots \f S \o \f h \right )_{ijk}:={}& \f \delta_{ij} \tr (\f S) \f h _k \\
%\left (\f \Theta^2\dreidots \f S \o \f h \right )_{ijk}:={}&  2 \left ((\f S)_{\skw}\right )_{ij} \f h _k \\
\left (\f \Theta^3\dreidots \f S \o \f h \right )_{ijk}:={}&2 \f \delta_{jk} \left ((\f S)_{\skw} \f h \right )_i -\f \delta_{ik} \left ((\f S)_{\skw} \f h \right )_j  \,.
\end{align*}

The {Levi--Civita}-tensor is defined in Section~\ref{sec:not}.
It allows to represent the cross product via
\begin{align*}
\f a \times \f b = \f \Upsilon :(\f a \otimes \f b) =\f \Upsilon _{ijk} \f a_j \f b _k\, \quad \text{for all }\f a , \f b \in \R^d \, 
\end{align*}
and the curl of a vectorfield via 
\begin{align*}
\curl \f d = \f\Upsilon_{ijk} \partial_j \f d_k \, \quad\text{for all }\f d \in \C^1 ( \Omega)\, .
\end{align*}
Together, we get
\begin{align*}
( \f d \cdot \curl \f d) =\f  \Upsilon \dreidots  \f d \otimes \nabla \f d ^T = \f\Upsilon _{ijk} \f d _i \f d _{k,j} \, . 
\end{align*}
The definition of the {Levi--Civita}-tensor implies 
\begin{align*}
\f\Upsilon_{kji}\f\Upsilon_{nml}=  \f \delta_{kn}\f \delta_{jm}\f \delta_{il} + \f \delta_{km}\f \delta_{jl}\f \delta_{in} + \f \delta_{kl}\f \delta_{jn}\f \delta_{im} - \f \delta_{kn}\f \delta_{jl}\f \delta_{im}- \f \delta_{km}\f \delta_{jn}\f \delta_{il} - \f \delta_{kl}\f \delta_{jm}\f \delta_{in} = \f \Theta^2
\end{align*}
and thus
\begin{align*}
\left (\f \Theta^2\dreidots \f S \o \f h \right )_{ijk}:={}& \f\Upsilon_{kji} ( (\f S )_{\skw} : \rot{\f h})  \,.
\end{align*}
The different term multiplied with each other gives
\begin{align*}
\left | \f \Theta^1\dreidots (\f S \o \f h- \tilde{\f S}\o \tilde{\f h})  \right |^2 =
{}& 3 |\tr (\f S))  \f h- \tr ( \tilde{\f S}) \tilde{ \f h}|^2\,, \\
%\left | \f \Theta^2\dreidots \f S \o \f h  \right |^2 ={}&4 | (\f S)_{\skw}|^2 | \f h|^2\,, \\
\left | \f \Theta^3\dreidots(\f S \o \f h- \tilde{\f S}\o \tilde{\f h})  \right |^2 ={}& 4 \sum_{ijk=1}^3
\left ( \f \delta_{jk} \left ((\f S)_{\skw} \f h- ( \tilde{\f S})_{\skw} \tilde{\f h} \right )_i -\f \delta_{ik} \left ((\f S)_{\skw} \f h- ( \tilde{\f S})_{\skw} \tilde{\f h} \right )_j \right )^2\\
 ={}&
 4\left (6 | (\f S)_{\skw}\f h- ( \tilde{\f S})_{\skw} \tilde{\f h} |^2 - 2 | (\f S)_{\skw}\f h- ( \tilde{\f S})_{\skw} \tilde{\f h} |^2 |^2\right )\\ ={}& 16 | (\f S)_{\skw}\f h - ( \tilde{\f S})_{\skw} \tilde{\f h}|^2 \,,
 \\
 \left | \f \Theta^2\dreidots \f S \o \f h  \right |^2 ={}& \sum_{ijk=1}^3 \f\Upsilon_{kji} \f\Upsilon_{kji} ( (\f S )_{\skw} : \rot{\f h}- (\tilde{\f S })_{\skw} : \rot{\tilde{\f h}})^2 \\={}& 6 ( (\f S )_{\skw} : \rot{\f h}- (\tilde{\f S })_{\skw} : \rot{\tilde{\f h}})^2 \,.
% \\
%  \f S \o \f h \dreidots  \f \Theta^1\dreidots \f \Theta^2\dreidots \f S \o \f h   ={}& 
%  \sum_{ijk=1}^3 \left (  \f \delta_{ij} \tr (\f S) \f h _k\right )\left (2\left ((\f S)_{\skw}\right )_{ij} \f h _k \right ) \\
%  ={}& 2\tr(\f S) (\f I : ( \f S)_{\skw}) | \f h|^2 = 0\,,\\
%  \f S \o \f h \dreidots  \f \Theta^1\dreidots \f \Theta^3\dreidots \f S \o \f h   ={}& 
%  2\sum_{ijk=1}^3 \left (  \f \delta_{ij} \tr (\f S) \f h _k\right )\left (\f \delta_{jk} \left ((\f S)_{\skw} \f h \right )_i -\f \delta_{ik} \left ((\f S)_{\skw} \f h \right )_j  \right ) \\
%  ={}& 2 \tr(\f S) ( \f h \cdot ( \f S)_{\skw}\f h -\f h \cdot ( \f S)_{\skw}\f h )   = 0\,,\\
%   \f S \o \f h \dreidots  \f \Theta^2\dreidots \f \Theta^3\dreidots \f S \o \f h   ={}& 
%  4 \sum_{ijk=1}^3 \left (  \left ((\f S)_{\skw}\right )_{ij} \f h _k \right )\left (\f \delta_{jk} \left ((\f S)_{\skw} \f h \right )_i -\f \delta_{ik} \left ((\f S)_{\skw} \f h \right )_j  \right ) \\
%  ={}& 4 \left (\f h \cdot ( \f S)_{\skw}^T ( \f S)_{\skw} \f h  - \f h \cdot ( \f S)_{\skw} ( \f S)_{\skw} \f h   \right )
%  = 8 |  ( \f S)_{\skw} \f h|^2 \,.
\end{align*}
For the products of different tensors, we see that 
\begin{align*}
\f S \o \f h \dreidots  \f \Theta^1\dreidots \f \Theta^2\dreidots \f S \o \f h   ={}& 
  \sum_{ijk=1}^3 \left (  \f \delta_{ij} \tr (\f S) \f h _k\right )\left (2\left ((\f S)_{\skw}\right )_{ij} \f h _k \right ) \\
   ={}& 2\tr(\f S) (\f I : ( \f S)_{\skw}) | \f h|^2 = 0\,,
\end{align*}

The definition of the {Levi--Civita}-tensors gives
\begin{align*}
\f \Upsilon_{ijk} \f \delta_{ij} = \f \Upsilon_{ijk} \f \delta_{jk}=\f \Upsilon_{ijk} \f \delta_{ik} = 0\,.
\end{align*}
And thus the products 
\begin{align*}
\f S \o \f h \dreidots \f \Theta ^2 \dreidots \f \Theta^i \dreidots \f S \o \f h \,,
\end{align*}
vanish for all  $i\in\{1 ,3\}$. 
This gives the assertion.
\end{proof}

\addcontentsline{toc}{section}{References}

\small
\bibliographystyle{abbrv}

\end{document}